\documentclass{amsart}

\usepackage[english]{babel} 
\usepackage[utf8]{inputenc}
\usepackage[T1]{fontenc} 
\usepackage{braket} 
\usepackage{physics}
\usepackage{tabularx} 
\usepackage{booktabs} 
\usepackage{rotating}
\usepackage{multirow} 
\usepackage{comment} 
\usepackage{tikz}
\usetikzlibrary{math} 
\usetikzlibrary{3d} 
\usepackage{tikz-3dplot}
\usepackage{tikz-cd} 
\usepackage{microtype}

\tikzcdset{
  cells={font=\everymath\expandafter{\the\everymath\displaystyle}},
}

\usepackage{amssymb} 
\usepackage{subcaption} 
\DeclareCaptionFormat{custom}

\usepackage[pdftex,colorlinks=true]{hyperref} 
\hypersetup{urlcolor=blue, citecolor=red, linkcolor=blue}

\usepackage{float} 
\usepackage{todonotes}

\usepackage[widespace]{fourier} 
\usepackage[capitalise]{cleveref}

\DeclareMathOperator{\Mat}{Mat} 
\DeclareMathOperator{\graph}{graph} 
\DeclareMathOperator{\stab}{Stab}
 
\DeclareMathOperator{\lcm}{lcm}
\DeclareMathOperator{\ord}{ord} 
\DeclareMathOperator{\coeff}{coeff}
 
\renewcommand{\vec}[1]{\mathbf{#1}}

\newcommand{\normalb}[2]{N_{#1/#2}}
\newcommand{\ssum}[2]{
	\underset{#1}{\overset{#2}{\sum}}
} %
\DeclareMathOperator{\SSym}{S}
\DeclareMathOperator{\Bl}{Bl} 
\DeclareMathOperator{\crem}{c}
\DeclareMathOperator{\adj}{adj} 
\DeclareMathOperator{\delp}{dP}
\newcommand{\dP}[1]{\delp_{#1}} 

\newcommand{\CB}{\Calc_{\mbox{\tiny\textit{(B)}}}}
\newcommand{\Calo}{\mathcal{O}} 
\newcommand{\Cale}{\mathcal{E}}
\newcommand{\Calc}{\mathcal{C}} 
\newcommand{\Calp}{\mathcal{P}}
\newcommand{\Calq}{\mathcal{Q}} 
\newcommand{\Calr}{\mathcal{R}}
\newcommand{\Call}{\mathcal{L}} 
\newcommand{\Cals}{\mathcal{S}}
\newcommand{\C}{\mathbb{C}} 
\newcommand{\R}{\mathbb{R}}
\newcommand{\Z}{\mathbb{Z}} 
\newcommand{\N}{\mathbb{N}}
\newcommand{\Pj}{\mathbb{P}}

\DeclareMathOperator{\Gl}{Gl}
 
\DeclareMathOperator{\Id}{Id}
\DeclareMathOperator{\Bir}{Bir} 
\DeclareMathOperator{\Base}{Base}
 
\DeclareMathOperator{\Fix}{Fix}
\DeclareMathOperator{\Div}{Div} 
\newcommand{\PGl}{\Pj\Gl}

\newtheorem{theorem}{Theorem}[section]
\newtheorem{proposition}[theorem]{Proposition}
\newtheorem{corollary}[theorem]{Corollary}
\newtheorem{lemma}[theorem]{Lemma}
\theoremstyle{definition}
\newtheorem{definition}[theorem]{Definition}
\newtheorem{notation}[theorem]{Notation}

\theoremstyle{remark} 
\newtheorem{remark}[theorem]{Remark}

\newtheorem*{conjecture*}{Conjecture}

\numberwithin{equation}{section}

\usepackage[backend=biber,sorting=nyt,maxnames=99,giveninits=true,doi=false,citestyle=numeric-comp]{biblatex}
\usepackage{csquotes}

\title{Growth and integrability of some birational maps in dimension three}

\author[M. Graffeo]{Michele Graffeo} 
\address[M. Graffeo]{Dipartimento di Matematica\\ Politecnico di Milano\\ Piazza Leonardo da Vinci 32\\ 20133 Milan\\ Italy}
\email{michele.graffeo@polimi.it}

\author[G. Gubbiotti]{Giorgio Gubbiotti}
\address[G. Gubbiotti]{Dipartimento di Matematica ``Federigo Enriques'',
	Universit\`a degli Studi di Milano, Via C. Saldini 50, 20133
	Milano, Italy \& INFN Sezione di Milano, Via G. Celoria 16,
	20133 Milano, Italy}
\email{giorgio.gubbiotti@unimi.it}


\date{\today}

\begin{document}

\begin{abstract}
    Motivated by the study of the Kahan--Hirota--Kimura discretisation
    of the Euler top, we characterise the growth and integrability
    properties of a collection of elements in the Cremona group of
    a complex projective 3-space using techniques  from algebraic
    geometry. This collection consists of maps obtained by composing the
    standard Cremona transformation $\crem_3\in\Bir(\Pj^3)$ with projectivities that
    permute the fixed points of $\crem_3$ and the points over which
    $\crem_3$ performs a divisorial contraction.  Specifically,
    we show that three behaviour are possible: \textit{(A)} integrable
    with quadratic degree growth and two invariants, \textit{(B)} periodic
    with two-periodic degree sequences and more than two invariants, and \textit{(C)}
    non-integrable with submaximal degree growth and one invariant.
\end{abstract}

\maketitle

\setcounter{tocdepth}{1}
\tableofcontents

\section{Introduction}
\label{sec:intro}

This paper is devoted to the characterisation of the growth and
integrability properties of a collection of birational transformations of
the complex projective space $\Pj^3$, i.e. a subset of the so-called Cremona group,
denoted by $\Bir (\Pj^{3})$ \cite[Section 7.4]{Dolgachev2012book}. Let $\crem_M\in\Bir (\Pj^{M})$ be the  \textit{standard Cremona transformation} of $\Pj^{M}$, i.e. the birational map defined as follows: 
\begin{equation}
        \label{eq:C3} \begin{tikzcd}[row sep=tiny]
                \crem_{M}\colon\Pj^M \arrow[dashed]{r} & \Pj^M
                \\ {[}x_1:\cdots:x_{M+1}{]} \arrow[mapsto]{r} &
                \left[\frac{1}{x_1}:\cdots:\frac{1}{x_{M+1}}\right].
        \end{tikzcd}
\end{equation}
Then, this collection
is obtained by composing $\crem_3$ 
with projectivities of finite order $g\in \PGl(3,\C)$ acting as permutations on a
set of special points for the Cremona transformation. This set,
that we will denote by $\Calr$ (see \Cref{specialpoints}), is
the union of the set of fixed points and of the points over which $\crem_3$ performs a divisorial contraction. Precisely,
the fixed points are the solutions of the (projective) equation
$\crem_{3}([x_1:x_2:x_3:x_4])=[x_1:x_2:x_3:x_4]$, while the remaining
points are the coordinate points, that is, the images of the coordinate
planes of $\Pj^3$ under $\crem_3$. We call this group of projectivities
the \emph{Cremona-cubes group} and we will denote it by $\mathcal{C}$ (see \Cref{def:cubegroup}).

The motivation to study the {Cremona-cubes group} comes from
some recent results on the Kahan--Hirota--Kimura discretisation
\cite{Kahan1993,HirotaKimura2000} (KHK discretisation). Indeed, as
it was noted in \cite{Alonso_et_al2022}, the celebrated integrable
discretisation of the Euler top \cite{Arnold1997,Gold2002} produced
in \cite{HirotaKimura2000} via KHK discretisation is expressible as the
composition of the standard Cremona map with \emph{two projectivities} (see \Cref{lem:eulerdecomposition}). We
will show in \Cref{sec:motivations} that, up to birational equivalence, this is the
prototypical integrable birational map of the form $g\circ \crem_3$ for $g$ in the Cremona-cubes group. In particular,
in \Cref{sec:entropy} we will explain why this birational map is \emph{integrable} in
the sense of the low growth condition.

Before to clarify what do we mean when we speak about growth properties
of a birational map, we need a couple more of definitions.

It is a well known fact that the iteration of a birational map $\Phi\in\Bir(\Pj^{M})$ 
produces a (birational) discrete dynamical
system defined by the recurrence: \begin{equation}\label{eq:rec}
        \left[ x_{1}(n):\cdots:x_{M+1}(n) \right] = \Phi^{n}\left(\left[
        x_{1}(0):\cdots:x_{M+1}(0) \right]\right), 
\end{equation}
applied on some initial datum $\left[ x_{1}(0):\dots:x_{M+1}(0)\right]\in\Pj^{M}$
(see \cite{GJTV_class}). The
issue is then to characterise the asymptotic behaviour of the
dynamics with respect to generic initial conditions. The problem
of characterising the growth of complexity of the iterates
was first considered by Arnol'd in \cite{Arnold1990}
for the class of diffeomorphisms. Its analogue for birational
maps was first developed experimentally in a series of papers (see
\cite{Veselov1992,FalquiViallet1993,Diller1996,Russakovskii1997}), which
ended in the elaboration of the concept of the \emph{algebraic entropy}
\cite{BellonViallet1999}.
Following Arnol'd \cite{Arnold1990} and Veselov \cite{Veselov1992},
the ``good'' measure of the complexity of a birational map
$\Phi\in\Bir(\Pj^{M})$ is the intersection of the
iterated images of a straight line with a generic hyperplane in the
complex projective space. By the Bezout's theorem in projective and multi-projective spaces
(see \cite[\S IV.2]{Shafarevich1994}), this corresponds to the degree of the polynomials
in the entries of $\Phi^{n}$. Rigorously, we quote the following
definition.

\begin{definition}[\cite{BellonViallet1999}]
        \label{def:degree}
        Given a birational map $\Phi\in\Bir(\Pj^{M})$
        \begin{equation}
                \label{eq:birpol}
                [x_1:\cdots:x_{M+1}] \mapsto \left[P_{1}(x_1,\ldots,x_{M+1}):\cdots :P_{M+1}(x_1,\ldots,x_{M+1}) \right],
        \end{equation}
        such that its (homogeneous) polynomial entries $P_{i}\in\C[x_{1},\ldots,x_{M+1}]$ are devoid of common factors,
        that is $\gcd(P_{1},\dots,P_{M+1})=1$, we define its \emph{degree} to
        be:
        \begin{equation}
                \label{eq:dphi}
                d^\Phi = \deg P_{i},
                \quad \text{ for any $i=1,\ldots, M+1$}.
        \end{equation}
        In the same way, for all $n\in\N$ we define $d_n^\Phi $ as the \emph{degree} 
        of the $n$-th iterate to be
        \begin{equation}
                \label{eq:dn}
                d^\Phi_{n} = d^{\Phi^{n}}.
        \end{equation}
\end{definition}

\begin{remark}
        \label{rem:degree}
        We make the following observations.
        \begin{enumerate}
            \item The degree of a birational map is invariant under conjugation by 
                projectivities, but, in general, it is not invariant under change of 
                coordinates (see \cite{BellonViallet1999,HietarintaViallet1998}).
            \item \Cref{def:degree} is not the usual definition of degree
                in algebraic geometry. For instance, in \cite[Section II.6.3]{Shafarevich1994}
                the degree of a (finite) rational map is defined to be the cardinality
                of a generic fibre. Nevertheless, all the rational maps in
                this paper will actually be birational. Hence no ambiguity
                is present, and the numbers $d^{\Phi}$ and $d_{n}^{\Phi}$ are
                uniquely determined by the birational map $\Phi\in\Bir(\Pj^{M})$.
            \item It is crucial in \Cref{def:degree} to require that the
                polynomial entries have no common factors. For a given birational map $\Phi \in\Bir(\Pj^{M})$, 
                after some iterations common
                factors can appear and they must be removed. This process
                has geometric meaning which we will discuss later in this section.
        \end{enumerate}
\end{remark}

Having specified the notion of degree of a birational map, we give
the definition of algebraic entropy which measures the growth of the complexity of a birational map.

\begin{definition}[\cite{BellonViallet1999}]
        \label{def:algent} 
        The \emph{algebraic entropy} of a birational map $\Phi\in\Bir(\Pj^{M})$ is the 
        following limit:
        \begin{equation}
                \label{eq:algentdef}
                S_{\Phi} = \lim_{n\to\infty}\frac{1}{n}\log d_{n}^\Phi.
        \end{equation} 
\end{definition}

\begin{remark}[\cite{BellonViallet1999, GubbiottiASIDE16,
                GrammaticosHalburdRamaniViallet2009}]
        \label{rem:algent}
        The algebraic entropy has the following properties:
        \begin{itemize}
            \item by the properties of birational maps and the subadditivity
                of the logarithm, using Fekete's lemma \cite{Fekete1923},
                the algebraic entropy always exists;
            \item the algebraic entropy is non-negative and bounded from above: $0\le S_\Phi \le \log d^\Phi$;
            \item the algebraic entropy is invariant with respect to birational
                conjugation. That is, given two birational maps
                $\Phi,\Theta\in\Bir(\Pj^{M})$, we have $S_{\Phi}
                = S_{\Theta^{-1}\circ\Phi\circ\Theta}$;
            \item if $d_{n}^\Phi$ is subexponential as $n\to\infty$, e.g.       polynomial, then
                $S_{\Phi}=0$, while, if $d_{n}^\Phi \sim a^{n}$ for some $a\in\R$, 
                then $S_{\Phi} = \log a$.
        \end{itemize}
\end{remark}

Armed with the definition of algebraic entropy and the properties
described in \Cref{rem:algent}, we can define the integrability
according to the algebraic entropy.

\begin{definition}[\cite{BellonViallet1999,HietarintaBook}]
        \label{def:integrability}
        A birational map $\Phi\in\Bir(\Pj^{M})$ is
        \emph{integrable according to the algebraic entropy} if $S_{\Phi}
        = 0$. If $S_{\Phi}>0$ the map is said to be \emph{non-integrable}
        or \emph{chaotic}. Moreover, if $d_{n}^\Phi\sim n$ as $n\to \infty$ the
        map is said to be \emph{linearisable}. Finally, if $d_{n}^\Phi$ is periodic
        the map is said to be \emph{periodic}.
\end{definition}

\begin{remark}
       	\label{rem:integrability}
        Most of the known integrable maps are such that $d_{n}^{\Phi}\sim n^{2}$ as
        $n\to\infty$. From \cite{Bellon1999}, it is known that if the orbits
        of the system are elliptic curves, then the degree growth is quadratic.
        From~\cite{DillerFavre2001} it is known that in $\Pj^2$ the only sub-exponential
        beahaviour are quadratic, linear, and periodic. The first is associated
        with the preservation of an elliptic fibration, the second
        with the preservation of a rational fibration, the latter with a power of
        the map being isotopic to the identity.
        In $\Pj^M$ with $M>2$, it is possible that $d_{n}^{\Phi}\sim n^{k}$ as $n\to\infty$
        with $k>2$. For instance, in \cite{AnglesMaillarViallet2002,
                LaFortuneetal2001, GJTV_class, JoshiViallet2017} maps with cubic
        growth were presented. However, often maps with cubic growth
        arise from maps with quadratic growth through a procedure called
        \emph{inflation} (see \cite{JoshiViallet2017,GJTV_class,Viallet2019}).
\end{remark}

Let $\Phi\colon X\dashrightarrow Y$ be a birational map between
smooth projective varieties $X,Y$. Recall that the \emph{singularities}, also called \emph{indeterminacies}, of a birational
map consist of the loci where the map is not defined.
In the paper we will denote by $\Base\Phi$ the \emph{indeterminacy locus} of the map $\Phi$.

Consider the resolution of indeterminacies of the map $\Phi$ given by the Zariski closure of the graph $Z=\overline{\graph(\Phi)}\subset X\times Y$, i.e.  the following commutative diagram
\begin{equation}
\begin{tikzpicture}
    \node[right] at (0,0) {$\subset X\times Y$};
    \node at (0,0) {$Z$};
    \node at (1,-1) {$Y.$};
    \node at (-1,-1) {$X$};
    \node[above] at (0,-1) {\small$\Phi$};

    \draw[->] (-0.1,-0.2) -- (-0.9,-0.8);
    \draw[->] (0.1,-0.2) -- (0.9,-0.8);
    \draw[dashed,->] (-0.8,-1) -- (0.8,-1);
    \node[right] at (0.5,-0.5) {\tiny $\pi_Y|_{Z}$};
    \node[left] at (-0.5,-0.5) {\tiny $\pi_X|_{Z}$};
\end{tikzpicture}
\end{equation}
In this setting, one can define (see \cite{CarsteaTakenawa2019JPhysA}) a notion of pullback for birational maps
\begin{equation}
    \Phi^*\colon H^2(Y,\Z)\rightarrow H^2(X,\Z),
\end{equation}
defined as $\Phi^*=\left(\pi_X|_{Z}\right)_*\circ\left(\pi_Y|_{Z}\right)^*$, where the pullback and the pushforward on the right hand side are the usual inverse and direct images via morphisms (see \cite{CarsteaTakenawa2019JPhysA}).

\begin{remark}
    As explained in \cite{CarsteaTakenawa2019JPhysA}, it is possible to perform the actual computation via an auxiliary smooth variety $\widetilde{Z}$ instead of the possibly singular closure of the graph $Z$. This is possible thanks to the celebrated Hironaka's result on resolution of singularities (see \cite{Hironaka1964}).
\end{remark}

The following theorem, whose proof is divided in \Cref{sec:entropy}, \Cref{sec:covariants}, and \Cref{sec:invariants}, summarises the results of this paper.

\begin{theorem}
	\label{thm:main}
Let $\Cale=\Set{e_1,\ldots,e_4}$	be the set of coordinate points of $\Pj^3$. Consider a birational map of the form $\Phi=g\circ\crem_{3}\in\Bir(\Pj^{3})$ (or $\Phi=\crem_{3}\circ g$),  for some $g\in\Calc$. 
    Then, there are three possibilities depending on the cardinality of the orbit 
    $\langle g\rangle\cdot \Cale$ of the points in $\Cale$, under the action of $g$. That is:
	\begin{itemize}
		\item If $\abs{\langle g\rangle\cdot \Cale} = 8$ then the map is integrable
            in the sense of \Cref{def:integrability}, i.e. $d_n^\Phi\sim n^{2}$ as $n\to\infty$.
            Moreover, $\Phi$ possesses a covariant net of quadrics, and \emph{two}
            functionally independent invariants determined by its action on
            $\Calr \setminus(\langle g\rangle\cdot \Cale)$.
		\item If $\abs{\langle g\rangle\cdot \Cale} = 4$  then the map is periodic
            in the sense of \Cref{def:integrability}, i.e. $d_{n}^{\Phi}\in\Set{1,3}$.
            Moreover, $\Phi$ possesses a covariant five-dimensional linear system of 
            quad\-rics, and \emph{three} functionally independent invariants.
		\item If $\abs{\langle g\rangle\cdot \Cale} = 12$
            then the map is non-integrable in the sense of \Cref{def:integrability}, i.e. $d_{n}^{\Phi}\sim \varphi^{2n}$, where $\varphi$ is the golden ratio.
            Moreover, $\Phi$ exhibits late confinement, and it possess 
            a covariant pencil of desmic surfaces, and \emph{one} invariant.
	\end{itemize}
\end{theorem}

In \Cref{thm:main}, by \emph{covariant linear system} $\Sigma$ we mean that there exists a divisor $D\in\Div(\Pj^3)$ such that the correspondence
\begin{equation}
        \label{eautocov}
        \begin{tikzcd}[row sep=tiny]
                \Sigma\arrow{r}{} & \Sigma \\
                W \arrow[mapsto]{r} & {(\Phi^{-1})}^* W-D,
        \end{tikzcd}
\end{equation}
is a well defined group automorphism (see \Cref{def:convariance}). While, by \emph{invariant} we mean a meromorphic function $R:\Pj^M\dashrightarrow\C$ such that $R\equiv R\circ\Phi$. Finally, we say
that some meromorphic functions $R_1,\ldots,R_k$ are \emph{functionally independent} if
at all points of $\Pj^M$ the Jacobian matrix of the map $\Pj^M\dashrightarrow \C^k$, defined by $R_1,\ldots,R_k$, has maximal rank $k$.

In principle, the definition of algebraic entropy in equation
\eqref{eq:algentdef} requires one to compute all the iterates
of a birational map $\Phi$ and to take the limit as $n\to\infty$.
For practical purposes this is clearly impossible. So, during the years,
several heuristics methods to compute the entropy has been proposed, for
instance using the concept of \emph{generating function} \cite{Lando2003}
(see also \cite{GubbiottiASIDE16,GrammaticosHalburdRamaniViallet2009}).
On the other hand, several methods to compute the algebraic
entropy exactly has been proposed. Notably, most of the approaches
use the algebro-geometric structure of the projective spaces
\cite{DillerFavre2001,BedfordKim2004,BedfordKim2006,Takenawa2001JPhyA,Viallet2015},
with some notable exceptions \cite{Hasselblatt2007}. In this sense,
the computation of the algebraic entropy is more accessible if the
singularity are confined (see \cite{Grammaticosetal1991,Viallet2015}). For
instance, in this paper, we compute the algebraic entropy of integrable
and non-integrable maps both confining singularities. 

In the present paper, to compute the exact value of the algebraic entropy
of the maps of the form $g\circ\crem_3$, for $g$ in the Cremona-cubes
group, we take the viewpoint of the construction of the \emph{space
of initial values} of the given map $\Phi\in\Bir(\Pj^{M})$. This
concept is the discrete analogue of Okamoto's description
\cite{Okamoto1979,Okamoto1977} of the continuous Painlev\'e equations
\cite{InceBook}, and it was conceived in \cite{Sakai2001}. To introduce
this concept we need to introduce the following definition.

\begin{definition}[\cite{CarsteaTakenawa2019JPhysA}]
        \label{def:anstable}
        A rational map $\Phi$ from a smooth projective variety $X$  to itself is called \emph{algebraically stable} if $(\Phi^*)^n = (\Phi^n)^*$ holds.
\end{definition}

\begin{remark}
        \label{rem:singularities}
        The concept of algebraic stability is related to the one of
        singularity confinement. Indeed, heuristically algebraic stability means that the
        singularities of the map behave in a controlled way: they either form finite 
        or periodic patterns. Specifying to the case of
        interest, i.e. maps in $\Pj^{M}$, a singularity pattern will be of the
        following form:
        \begin{equation}
                \label{eq:singpatt}
                D \xrightarrow{\Phi} \gamma_{1} \xrightarrow{\Phi}
                \gamma_{2}\xrightarrow{\Phi} \cdots \xrightarrow{\Phi}
                \gamma_{K} \xrightarrow{\Phi} D',
        \end{equation}    
        where $D$, $D'$ are divisors and
        $\gamma_{i}$ are varieties of codimension greater than one. Finite concatenations of
        patterns of the form \eqref{eq:singpatt} can repeat periodically as
        long as the number of centres $\gamma_{i}$ stays finite (this
        last requirement can be false for linearisable equations
        \cite{Ablowitz_et_al2000,Takenawaatel2003,HayHoweseNakazonoShi2015}).
        Following \cite{BellonViallet1999,Viallet2015}, we can compute which
        are the divisors contracted by the map $\Phi$ and its inverse. Precisely, calling $\Psi\in\Bir(\Pj^{M})$
        the inverse of $\Phi$ the following relations hold:
        \begin{equation}
                \label{eq:kappadef}
                \Psi \circ \Phi\equiv \kappa \cdot \Id_{\Pj^M},
                \quad
                \Phi \circ \Psi\equiv \lambda \cdot\Id_{\Pj^M},
                \quad
                \kappa,\lambda\in\C\left[ x_{1},\cdots,x_{M+1} \right].
        \end{equation}
        The polynomials $\kappa$ and $\lambda$ admit a possibly trivial 
        factorisation  of the form:
        \begin{equation}
                \label{eq:kappafact}
                \kappa = \prod_{i=1}^{K_{\kappa}} \kappa_{i}^{d_{\kappa,i}},
                \quad
                \lambda = \prod_{i=1}^{K_{\lambda}} {\lambda}_{i}^{d_{\lambda,i}},
        \end{equation}
        where $\kappa_i \neq \kappa_j$ and $\lambda_i \neq \lambda_j$ for $i\neq j$.
        The only (prime) divisors that can be contracted to subvarieties of higher
        codimension by $\Phi$ are then the varieties:
        \begin{equation}
                \label{eq:singvarphi}
                \mathrm{K}_i =  \{\kappa_{i}=0\},
                \quad\mbox{for }i=1,\ldots,K_\kappa,
        \end{equation}
        while $\Psi$ can only contract the varieties:
        \begin{equation}
                \label{eq:singvarpsi}
                \Lambda_j = \{\lambda_{j}=0\}\quad\mbox{for }j=1,\ldots,K_\lambda.
        \end{equation}
        In \Cref{fig:singularities} we present a possible blow-up blow-down
        sequence in $\Pj^{3}$: the surface $D$ is mapped to a curve and then to a point, 
        but after four steps the singularity is confined and a new surface $D'$ is
        found. This is a graphical representation of equation \eqref{eq:singpatt}.
\end{remark}

\begin{figure}[htb]
        \centering
        \begin{tikzpicture}[scale=0.4]
                \draw[thick] (1,0)--(5,4)--(5,12)--(1,8)--(1,0);
                \node[below right] at (3,2) {$D$};
                
                \draw[->] (5.8,6) -- (8.2,6);
                \draw (9,0).. controls (7,5) and (11,9) ..(9.1,12);
                \draw[->] (9.2,6) -- (12-0.5,6);
                
                \node[below] at (9,0) {$\gamma$};
                \filldraw (12,6) circle (3pt);
                \node[above] at (12,6) {$p$};
                \draw[->] (12.4,6) -- (14.6,6);
                \filldraw (15,6) circle (3pt);
                \node[above] at (15,6) {$p'$};
                
                \draw[->,dashed] (15.5,6)--(18-0.6,6);
                \draw (18,0).. controls (16,5) and (20,9) ..(18,12);
                \node[below] at (18,0) {$\gamma'$};
                
                \draw[->, dashed] (18.3,6) -- (20.7,6);
                \draw[thick] (21.5,0)--(25.5,4)--(25.5,12)--(21.5,8)--(21.5,0);
                \node[below right] at (23.5,2) {$D'$};
        \end{tikzpicture}
        \caption{A possible blow-down blow-up sequence in $\Pj^{3}$.}
        \label{fig:singularities}
\end{figure}

The following result allows us to characterise algebraically stable maps
from the structure of their indeterminacy locus as described in \Cref{rem:singularities}.

\begin{proposition}[\cite{Bayraktar2013,CarsteaTakenawa2019JPhysA,BedfordKim2004,BedfordKim2008}]
        \label{prop:anstable}
        Let $X$ be a smooth projective variety and let
        $\Phi\in\Bir( X)$ be a birational map with indeterminacy locus
        $\Base\Phi$. Then, the map $\Phi$ is algebraically stable if
        and only if it does not exist a positive integer $k$ and a
        divisor $E$ on $X$  such that $\Phi (E \smallsetminus \Base\Phi )
        \subset \Base( \Phi^k)$.
\end{proposition}

Then, we define.

\begin{definition}
        \label{def:spaceofinitialcond}
        A \emph{space of initial values} of a map $\Phi\in\Bir(\Pj^{M})$
        is the datum of a birational projective morphism $\varepsilon\colon B\to\Pj^{M}$ such that the variety $B$ is smooth and
        the lifted (birational) maps $\widetilde{\Phi},\widetilde{\Phi}^{-1}\in\Bir (B)$ are algebraically stable. Sometimes we will also call space of initial values simply the variety $B$.
\end{definition}

\begin{remark}
        In what follows, using the canonical isomorphism $\Bir(\Pj^3)\cong\Bir(B)$, with abuse of notation, we will denote by $\Phi$ also the map $\widetilde{\Phi}$, specifying, at each instance, if we are working with the projective space or with the variety $B$.  
\end{remark}

Suppose now that the (prime) subvarieties $\gamma_i$, for $i=1,\ldots,K$, of codimension  greater than one encountered in the singularity pattern \eqref{eq:singpatt} of some map $\Phi$ are disjoint, i.e. $\gamma_i\cap\gamma_j=\emptyset$ for $i\not=j$ irreducible and smooth (we will just blow-up reduced points). The general case is more intricate and beyond our purpose. 
As a consequence of \Cref{rem:singularities} and of the properties of the blowup (see \cite[Proposition IV-22]{GEOFSCHEME}), we have that
\begin{equation}
        \label{eq:blowups}
        B = \Bl_{\underset{i=1}{\overset{K}{\cup}}\gamma_{i}}(\Pj^{M})
\end{equation}
is a space of initial values for $\Phi$. Denoting by $F_{i}$, for $i=1,\ldots,K$, the exceptional divisors of $\varepsilon$, we attach to $B$ its second cohomology group (see \cite[Section 4.6.2]{GRIHARR}):
\begin{equation}
        \label{eq:H2B}
        H^{2}(B,\Z) =
        \langle \varepsilon^{*}H , F_{1},\dots, F_{K}\rangle_{\Z}.
\end{equation}
Then, the action of $(\Phi^{-1})^{*}$ on $H^{2}(B,\Z)$ is linear and
the coefficient of the pullback of $\varepsilon^{*}H$ via $\Phi$ agrees with the \emph{degree of $\Phi$} in the sense of equation \eqref{eq:dn}.
So, following \cite{Takenawa2001JPhyA,BedfordKim2004}, from the algebraic stability 
condition we get that:
\begin{equation}
        \label{eq:dncoeff}
        d_{n}^\Phi  = \coeff\left( ((\Phi^{-1})^{*})^{n}\varepsilon^* H, \varepsilon^{*}H \right) = \coeff\left( (\Phi^{*})^{-n}\varepsilon^* H, \varepsilon^{*}H \right),
\end{equation}
that is we converted the problem of finding a closed form
expression for $d_{n}^\Phi$ to a problem in linear algebra over
the $\Z$-module $H^{2}(B,\Z)$.

The plan of the paper is the following. In \textsc{\Cref{sec:motivations}}
we present the motivations to consider the Cremona-cubes group,
taken from the recent literature on the KHK discretisation.  In
\textsc{\Cref{sec:cremona}} we recall some of the needed properties of
the standard Cremona transformations and we describe their resolutions of
indeterminacies in dimension 2 and 3. We will also remark that, in dimension
3, the associated variety is singular at twelve conifold points. In
\textsc{\Cref{sec:cubegroup}} we introduce rigorously the Cremona-cubes
group, a subgroup of $\Pj\Gl(4,\C)$.  In \textsc{\Cref{sec:entropy}}
we prove the growth properties described in \Cref{thm:main}.  Next, in
\textsc{\Cref{sec:covariants}}, we discuss the existence of covariant
linear systems of quadrics and quartics as stated in \Cref{thm:main}.
So, in the successive \textsc{\Cref{sec:invariants}}, we construct the
invariants via geometrical arguments ending the proof of \Cref{thm:main}. In particular, we
find results matching with those of \textsc{\Cref{sec:entropy}} because
we find two invariants for integrable maps, three for periodic maps, and
only one for non-integrable ones.  Finally, in \textsc{\Cref{sec:concl}}
we present some conclusion and some outlook for future works.

\section{The KHK discretisation of the Euler top}
\label{sec:motivations}

It is a well known fact that most of the problems in the theory
of dynamical systems cannot be solved in closed form. For instance, in \cite[\S 5,
        pag. 22]{Arnold1997},
V.~I. Arnol'd wrote: \begin{quote}
        \emph{``Analyzing a general potential system with two degrees of
                freedom is beyond the capability of modern science.''} 
\end{quote} This led many scientists to develop and study methods
to produce systems that could be tackled numerically
\cite{NumericalRecipes3rd}. In the case of ordinary differential
equations, this amounts to produce \emph{discretisations} which can be
solved iteratively. The problem that arises with the discretisation
approach is then to preserve the known properties of the continuous
systems. For instance, standard Hamiltonian systems are known to be
\emph{conservative}, meaning that the orbits of a Hamiltonian system
cannot spiral into points (or formally, stable equilibrium points
cannot be asymptotically stable). On the other hand, it is known
that this property is not preserved by all numerical methods, and for
instance a symplectic integrator cannot preserve exactly the energy
and vice versa an energy-preserving integrator cannot be symplectic 
\cite{ZhongMarsden1988}. This considerations led to the
introduction of a branch of numerical analysis called \emph{geometric
integration}, whose aim is to build discretisations preserving
as much as possible the properties of their continuous counterparts
\cite{BuddIserles1999,BuddPiggot2003,KrantzParks2008}.

In a series of unpublished lecture notes
(see \cite{Kahan1993}), W. Kahan devised a method to obtain good numerical approximations in the
sense outlined above: the orbits of some conservative systems did not seem
to be affected by the physically and mathematically incorrect spiraling
behaviour \cite{KahanLi1997}. Kahan's method was
rediscovered independently by Hirota and Kimura, who used it to produce
integrable discretisations of the Euler top \cite{HirotaKimura2000}
and the Lagrange top \cite{KimuraHirota2000}, followed by Suris
and his collaborators who produced many more integrable examples in
\cite{PetreraPfadlerSuris2009,PetreraSuris2010,PetreraPfadlerSuris2011}.
Before discussing our case of interest, that is the Euler top, we give
a brief account of the method.

\begin{definition}
        \label{def:khk}
        Assume we are given an \emph{$M$-dimensional system of first order differential equations}
        (also called a \emph{vector field}):
        \begin{equation}
                \label{eq:firstord}
                \vec{\dot{x}}=\vec{f}\left( \vec{x} \right),
                \quad
                \vec{x}\colon\R^{M}\to\R,
                \quad
                \vec{f}\colon\R^{M}\to\R^{M}.
        \end{equation}
        Then, its \emph{Kahan--Hirota--Kimura discretisation (KHK)} is:
        \begin{equation}
                \label{eq:kahan}
                \frac{\vec{x'}-\vec{x}}{h} = 
                2\vec{f}\left( \frac{\vec{x'}+\vec{x}}{2} \right)
                -\frac{\vec{f}\left( \vec{x'} \right)+\vec{f}\left( \vec{x}\right)}{2},
        \end{equation}
        where $\vec{x} = \vec{x}\left( n h \right)$,
        $\vec{x'}=\vec{x}\left( (n+1)h \right)$, and $h>0$ is an infinitesimal
        parameter.    
\end{definition}

\begin{remark}
        \label{rem:kahanproperties}
        In this remark we resume the most important known facts about the
        KHK discretisation.
        \begin{itemize}
            \item If the function $\vec{f}$ is \emph{quadratic}, then
                the associated map
                \begin{equation}
                        \label{eq:kahan2}
                        \boldsymbol{\Phi}_{h}\left( \vec{x} \right)=\vec{x'}= 
                        \vec{x} +
                        h \left( I_{M}-\frac{h}{2} \grad_{\vec{x}}\vec{f}\left( \vec{x} \right) \right)^{-1}\vec{f}\left(\vec{x}\right),
                \end{equation}
                 where $\grad_{\vec{x}}\vec{f}\left( \vec{x} \right)$ is the Jacobian of
                the function $\vec{f}$, is birational (see
                \cite{PetreraPfadlerSuris2011,CelledoniMcLachlanOwrenQuispel2013}).
                Its inverse is obtained through the substitution $h\mapsto-h$,
                i.e. $\boldsymbol{\Phi}^{-1}_{h}\left( \vec{x} \right)= 
                \boldsymbol{\Phi}_{-h}\left( \vec{x} \right)$.
                This considerations carry over in $\Pj^M$ considering first the complexification $\C^M$ of $\R^M$ and then its compactification to $\Pj^M$ with hyperplane at
                infinity $\Set{x_{M+1}=0}$. We denote the corresponding
                map by $\Phi_{h} \in\Bir(\Pj^{M})$ to underline the dependence
                on $h>0$.
            \item When applied to quadratic vector fields the KHK method is the restriction 
                of a Runge-Kutta method \cite{NumericalRecipes3rd}, so it is covariant with
                respect to affine transformations \cite{CelledoniMcLachlanOwrenQuispel2013}.
            \item Suppose that the vector field is Hamiltonian. That is, there exist
                a function $H\colon \R^{M}\to\R$ and a constant skew-symmetric
                matrix $J\in \Mat_{M,M}\left( \R \right)$ such that:
                \begin{equation}
                        \label{eq:firstordham}
                        \vec{\dot{x}}=J\grad H\left( \vec{x} \right).
                \end{equation}
               If $\deg H = 3$,  then the associated KHK discretisation admits
                an invariant $\tilde{H}_{h}$, such that $\lim_{h\to0^{+}} \tilde{H}_{h}=H$
                and a preserved measure which is a $h$-deformation of the standard
                Euclidean measure \cite{CelledoniMcLachlanOwrenQuispel2013, CelledoniMcLachlanOwrenQuispel2014}.
        \end{itemize}
\end{remark}

The Euler top is the following system of three first order quadratic
equations in the variables $(x_1,x_2,x_3)\in\R^3$:
\begin{equation}
        \label{eq:euler}
        \dot{x_1} = a_1 x_2 x_3,
        \quad 
        \dot{x_2} = a_2 x_1 x_3,
        \quad 
        \dot{x_3} = a_3 x_1x_2.
\end{equation}
This is a well known integrable system (see \cite[\S 29]{Arnold1997})
whose solution was given by Euler himself in terms of elliptic
functions. In fact, the Euler top belongs to a wider family
of continuous integrable systems known as the Manakov systems
\cite{Mishchenko1970,Manakov1976}.  Following \eqref{eq:kahan}, the KHK
discretisation of the Euler top is:
\begin{equation}
        \label{eq:eulerkhk}
        \frac{x_1'-x_1}{h} = \frac{a_1}{2} (x_2'x_3+x_2x_3'),
        \quad 
        \frac{x_2'-x_2}{h} = \frac{a_2}{2} (x_1'x_3+x_1x_3'),
        \quad 
        \frac{x_3'-x_3}{h} = \frac{a_3}{2} (x_1'x_2+x_1x_2').
\end{equation}
Explicitly, from \eqref{eq:kahan2}, after introducing homogeneous
coordinates $[x_1:x_2:x_3:x_4]$ on $\Pj^3$ we have the following
map of projective spaces:
\begin{equation}
        \label{eq:eulerkhk2}
        \begin{tikzcd}[row sep=tiny]
                \Pj^3 \arrow[dashed]{r}{\Phi_h} & \Pj^3 \\
                {[}x_1:x_2:x_3:x_4{]} \arrow[mapsto]{r} & \left[ x_1':x_2':x_3':x_4'\right],
        \end{tikzcd}
\end{equation}
where:
\begin{subequations}
        \label{eq:eulerxp}
        \begin{align}
                x_1' & =- \left( {  a_1} {  a_2} x_3^{2}+{  a_1} {  a_3} x_1^{2}-{a_2} {  a_3} x_1^{2} \right) {h}^{2}x_1-4 {  a_1} hx_2x_3x_4-4 x_1x_4^{2},
                \\
                x_2' & =- \left( {  a_1} {  a_2} x_3^{2}-{  a_1} {  a_3} x_2^{2}+{  a_2} {  a_3} x_1^{2} \right) {h}^{2}x_2-4 {  a_2} hx_1x_3x_4-4 x_2x_4^{2},
                \\
                x_3' & =  \left( {  a_1} {  a_2} x_3^{2}-{  a_1} {  a_3} x_2^{2}-{
                        a_2} {  a_3} x_1^{2} \right) {h}^{2} x_3 -4 {  a_3} hx_1x_2x_4-4 x_3x_4^{2},
                \\
                x_4' & = {  a_1} {  a_2} {  a_3} {h}^{3}x_1x_2x_3
                + \left( {  a_1} {  a_2} x_3^{2}+{  a_1} {  a_3} x_2^{2}+{  a_2} {  a_3} x_1^{2}\right) 
                {h}^{2}x_4-4 x_4^{3}.
        \end{align}
\end{subequations}

From \cite{HirotaKimura2000}, it is known that the above system is
integrable, with its Hamiltonian formulation given in \cite{PetreraSuris2010}.
Another remarkable property is the
existence of a Lax pair\footnote{Sometimes in the Russian literature
called a $L-A$ pair.} \cite{Kimura2017JPhysA,Kimura2017Lax,Sogo2017},
the only known case for a KHK discretisation along with the discrete
Nahm system \cite{GubNahm}. More recently, in \cite{Alonso_et_al2022},
the system \eqref{eq:eulerkhk} was derived using a three-dimensional
analogue of the QRT construction \cite{QRT1988,QRT1989}, that is as
action of involutions on two pencils of quadrics. In the same paper
\cite[Prop. 7.2]{Alonso_et_al2022} it is discussed the reduction of the
system to a three-dimensional standard Cremona transformation composed
with \emph{two} projectivities (see \Cref{lem:eulerdecomposition}). In the rest of this Section, we will interpret this statement from a different viewpoint  based on singularity confinement.
Moreover, we will explain why this naturally leads
to the definition of the Cremona-cubes group (see \Cref{def:cubegroup}).

From the heuristic point of view, if we compute the sequence of degrees
of the iterations of the map in \eqref{eq:eulerkhk2}, we obtain:
\begin{equation}
        \label{eq:dgrowth}
        1, 3, 9, 19, 33, 51, 73, 99, 129, 163, 201, 243, 289\dots.
\end{equation}
The following generating function fits the values in  \eqref{eq:dgrowth}:
\begin{equation}
        \label{eq:gfq}
        g(s) =\frac{3s^2 + 1}{(1 - s)^3}.
\end{equation}
Applying the inverse $\mathcal{Z}$-transform
(see \cite{GubbiottiASIDE16}) \eqref{eq:gfq}, we get $d_n^{\Phi_h}= 2 n^2 +1$. Furthermore, by
examining the sequence \eqref{eq:dgrowth}, we see that the deviation from
the standard growth $d_n^{(s)}=3^n$ happens at the third iterate. After having explored the singularity pattern (see \Cref{fig:singularitieseuler}), we
will see that this is not an accident, but it has a deep meaning.

Recall that $\Phi^{-1}_h = \Phi_{-h}$ (see \Cref{rem:kahanproperties}). Then,  following the idea of singularity confinement, we compute the polynomials
$\kappa$ and $\lambda$ from \eqref{eq:kappadef} for \eqref{eq:kahan2}:
\begin{equation}
        \label{eq:kappalambdaeuler}
        \kappa = \prod_{i=1}^{4} \kappa_{i}^{2},
        \quad
        \lambda = \prod_{i=1}^{4} \lambda_{i}^{2},
\end{equation}
with:
\begin{subequations}
        \label{eq:kappai}
        \begin{align}
                \label{eq:kappa_1}
                \kappa_{1} &= 
                \alpha_1 \alpha_2 h x_3-\alpha_1 \alpha_3 h x_2-\alpha_2 \alpha_3 h x_1-2 x_4,
                \\
                \label{eq:kappa_2}
                \kappa_{2}  &= 
                \alpha_1 \alpha_2 h x_3+\alpha_1 \alpha_3 h x_2-\alpha_2 \alpha_3 h x_1+2 x_4,
                \\
                \label{eq:kappa_3}
                \kappa_{3} &= 
                \alpha_1 \alpha_2 h x_3-\alpha_1 \alpha_3 h x_2+\alpha_2 \alpha_3 h x_1+2 x_4,
                \\
                \label{eq:kappa_4}
                \kappa_{4} &= 
                \alpha_1 \alpha_2 h x_3+\alpha_1 \alpha_3 h x_2+\alpha_2 \alpha_3 h x_1-2 x_4,
        \end{align}
\end{subequations}
and
\begin{subequations}
        \label{eq:lambdai}
        \begin{align}
                \lambda_{1} &= 
                \alpha_1 \alpha_2 h x_3 -\alpha_1 \alpha_3 h x_2 -\alpha_2 \alpha_3 h x_1+2 x_4,
                \\
                \lambda_{2} &= 
                \alpha_1 \alpha_2 h x_3+\alpha_1 \alpha_3 h x_2-\alpha_2 \alpha_3 h x_1-2 x_4,
                \\
                \lambda_{3} &= 
                \alpha_1 \alpha_2 h x_3-\alpha_1 \alpha_3 h x_2+\alpha_2 \alpha_3 h x_1-2 x_4,
                \\
                \lambda_{4} &= 
                \alpha_1 \alpha_2 h x_3+\alpha_1 \alpha_3 h x_2+\alpha_2 \alpha_3 h x_1+2 x_4,
        \end{align}
\end{subequations}
for some choice of square root $\alpha_j$ of $a_{j}$, for $j=1,2,3$.

Let us consider the varieties
\begin{equation}
        \label{piani}
        \mathrm{K}_i=\left\{ \kappa_{i}=0 \right\}\mbox{ and
        }\Lambda_{i} = \left\{ \lambda_{i}=0 \right\}.
\end{equation}
Then, for $i=1,\ldots,4$,  we have the following singularity pattern:
\begin{equation}
        \label{eq:eulersingpatt}
        \mathrm{K}_i 
        \longrightarrow 
        s_{i} 
        \longrightarrow
        s_{i}'
        \dashrightarrow
        \Lambda_{i},
\end{equation}
where
\begin{subequations}
        \label{sisip}
        \begin{align}
                \label{s1s1p}
                s_{1} = 
                \left[ 2\alpha_{1}:2\alpha_{2}:-2\alpha_{3}:-\alpha_{1}\alpha_{2}\alpha_{3}h \right],
                &\quad
                s_{1}'=
                \left[ 2\alpha_{1}:2\alpha_{2}:-2\alpha_{3}:\alpha_{1}\alpha_{2}\alpha_{3}h \right],
                \\
                \label{s2s2p}
                s_{2} =
                \left[ 2\alpha_{1}:-2\alpha_{2}:-2\alpha_{3}:\alpha_{1}\alpha_{2}\alpha_{3}h \right],
                &\quad
                s_{2}'=
                \left[ 2\alpha_{1}:-2\alpha_{2}:-2\alpha_{3}:-\alpha_{1}\alpha_{2}\alpha_{3}h \right],
                \\
                \label{s3s3p}
                s_{3} = 
                \left[ 2\alpha_{1}:-2\alpha_{2}:2\alpha_{3}:-\alpha_{1}\alpha_{2}\alpha_{3}h \right],
                &\quad
                s_{3}'=
                \left[ 2\alpha_{1}:-2\alpha_{2}:2\alpha_{3}:\alpha_{1}\alpha_{2}\alpha_{3}h \right],
                \\
                \label{s4s4p}
                s_{4} = 
                \left[ 2\alpha_{1}:2\alpha_{2}:2\alpha_{3}:\alpha_{1}\alpha_{2}\alpha_{3}h \right],
                &\quad
                s_{4}'=
                \left[ 2\alpha_{1}:2\alpha_{2}:2\alpha_{3}:-\alpha_{1}\alpha_{2}\alpha_{3}h \right].
        \end{align}
\end{subequations}
The singularity patterns in \eqref{eq:eulersingpatt} are depicted in
\Cref{fig:singularitieseuler}. This immediately explains the growth
in equation  \eqref{eq:dgrowth}: the deviation from the standard growth happens at
the third step because after three steps the map enters the singularities.
Notice that, the degree drop is exactly given by $\deg \kappa = 8$ because $\kappa$ is the common factor to be removed in the computation of the degree.
A similar analysis for $\Phi_h^{-1}$ shows that on the third iterate the common 
factor $\lambda$, whose degree is eight, is removed.

\begin{figure}[htb]
        \begin{tikzpicture}[scale=0.4]
                \draw[thick] (0,0)--(4,4)--(4,12)--(0,8)--(0,0);
                \node[below right] at (2,2) {$\mathrm{K}_{i}$};
                \draw[->] (5,6) -- node[above]{$\Phi_{h}$} (8.6,6) ;
                \filldraw (9,6) circle (3pt);
                \node[above] at (9,6) {$s_{i}$};
                \draw[->] (9.4,6) -- node[above]{$\Phi_{h}$} (13,6) ;
                \filldraw (13.4,6) circle (3pt);
                \node[above] at (13.4,6) {$s'_{i}$};
                \draw[->,dashed] (13.8,6) -- node[above]{$\Phi_{h}$} (17.6,6) ;
                \draw[thick] (18,0)--(22,4)--(22,12)--(18,8)--(18,0);
                \node[below right] at (20,2) {$\Lambda_{i}$};
        \end{tikzpicture}
        \caption{The blow-down blow-up sequence of the Euler top $\Pj^{3}$.}
        \label{fig:singularitieseuler}
\end{figure}

\begin{remark}
        \label{rem:limit}
        Note that, as $h\to 0^{+}$, all the points appearing in
        the singularity pattern are pushed on the plane at infinity
        $L_{\infty} = \left\{ x_4=0 \right\}$. This has to be expected
        because the invariants (first integrals) of the continuous
        Euler top \eqref{eq:euler} are polynomials \cite{Arnold1997}
        and polynomials are singular only at infinity (see
        also \cite{GrammaticosHalburdRamaniViallet2009,BellonViallet1999}
        for a similar discussion on the importance of singularities for
        polynomial maps).
\end{remark}

The geometry of of the singularity confinement is also enough
to build the invariants of the system, and hence to prove integrability
in the na\"ive sense (see \cite{GJTV_class}). Indeed, by considering
the net $\Sigma$ of quadrics passing through the points $s_{i}$ and $s_{i}'$, we find:
\begin{equation}
        \label{eq:Qeuler}
        \Sigma =
        \Set{Q_{\mu,\nu,\xi}\subset \Pj^3 | 
                \begin{array}{c}
                        \left[\mu:\nu:\xi\right]\in\Pj^2\\
                        Q_{\mu,\nu,\xi}=\left\{\mu \left( x_4^{2} - \frac{a_{1}a_{2} x_3^{2}}{4}  \right)
                        +\nu \left( x_2^{2} - \frac{a_{2}x_3^{2}}{a_3}  \right)
                        +\xi \left( x_1^{2} - \frac{a_1 x_3^{2}}{a_3}  \right)=0 \right\} 
                \end{array}
        } .
\end{equation}

Then, it is easy to show that $(\Phi_{h}^{-1})^{*}(Q) = \kappa_{1} \kappa_{2}
\kappa_{3} \kappa_{4} Q$, i.e. the net $\Sigma$  enjoys nice
\emph{covariance} properties (see \Cref{def:convariance}) with respect to the
KHK discretisation of the Euler top \eqref{eq:eulerxp}. 
This easily yields two functionally independent invariants:
\begin{equation}
        \label{eq:inveuler}
        I_{1} = 
        \frac{x_1^{2} - a_1 x_3^{2}/a_3}{
                x_4^{2} - a_1a_2 x_3^{2}/4},
        \quad
        I_{2} =
        \frac{x_2^{2} -  a_2x_3^{2}/a_3}{
                x_4^{2} - a_1a_2x_3^{2}/4}.
\end{equation}
The situation is a three-dimensional generalisation of the
results in \cite{PSS2019,GMcLQ_dummies}. Therein, it was proved how the
geometric structure of the KHK discretisation of a two-dimensional system
with cubic Hamiltonian determines its integrability. Moreover, a special r\^ole is played by the singular fibres of the associated
pencil. Note that the coordinate points of $\Sigma\cong\Pj^2_{[\mu: \nu: \xi]}$ 
are indeed singular members of the net.

\begin{remark}
    \label{rem:nongeneral} 
    We note that the eight points
    $\Set{s_{i}}_{i=1}^4 \cup \set{s_{i}'}_{i=1}^4$, are not in
    \emph{general position}. Indeed, in general, eight points of
    $\Pj^3$  generates only a pencil of quadrics.
    In \Cref{puntifissi} we will explain their relation with the fixed locus of the Cremona transformation. This led us to depict them as two tetrahedra in 
    \Cref{fig:tetra}.  To see that they are not in general position notice, for instance, that this set of vertices contains \emph{twelve 
    distinct co-planar quadruples}.
    We will see in \Cref{sec:cubegroup} how this this related to the
    \emph{Cremona-cubes group}.
\end{remark}

\begin{figure}[hbt]
        \centering
        
        \tdplotsetmaincoords{70}{130}
        \begin{tikzpicture}[scale=4,tdplot_main_coords]
                \draw[color=black,thick,->] (0,0,0) -- (1.5,0,0) node[anchor=north east]{$x_1/x_4$};
                \draw[color=black,thick,->] (0,0,0) -- (0,1.5,0) node[anchor=north west]{$x_2/x_4$};
                \draw[color=black,thick,->] (0,0,0) -- (0,0,1) node[anchor=north west]{$x_3/x_4$};
                
                \def\alpha{2}
                \def\beta{2}
                \def\gamma{2}
                \def\h{1}

                \coordinate (s1) at ({-2/(\beta*\gamma*\h)}, {2/(\alpha*\gamma*\h)}, {2/(\alpha*\beta*\h)});
                \coordinate (s2) at ({2/(\beta*\gamma*\h)}, {-2/(\alpha*\gamma*\h)}, {2/(\alpha*\beta*\h)});
                \coordinate (s3) at ({-2/(\beta*\gamma*\h)}, {-2/(\alpha*\gamma*\h)}, {-2/(\alpha*\beta*\h)});
                \coordinate (s4) at ({2/(\beta*\gamma*\h)}, {2/(\alpha*\gamma*\h)}, {-2/(\alpha*\beta*\h)});
                
                \coordinate (s1p) at ({2/(\beta*\gamma*\h)}, {2/(\alpha*\gamma*\h)}, {2/(\alpha*\beta*\h)} );
                \coordinate (s2p) at ({-2/(\beta*\gamma*\h)}, {-2/(\alpha*\gamma*\h)}, {2/(\alpha*\beta*\h)});
                \coordinate (s3p) at ({-2/(\beta*\gamma*\h)}, {2/(\alpha*\gamma*\h)}, {-2/(\alpha*\beta*\h)});
                \coordinate (s4p) at ({2/(\beta*\gamma*\h)}, {-2/(\alpha*\gamma*\h)}, {-2/(\alpha*\beta*\h)});

                \draw[red,thick] (s1)--(s2)--(s3)--(s1);
                \draw[red,thick] (s1)--(s2)--(s4)--(s1);
                \draw[red,thick] (s1)--(s3)--(s4)--(s1);

                \draw[blue,thick] (s1p)--(s2p)--(s3p)--(s1p);
                \draw[blue,thick] (s1p)--(s2p)--(s4p)--(s1p);
                \draw[blue,thick] (s1p)--(s3p)--(s4p)--(s1p);

                \node[red,anchor=south west] at (s1) {$s_1$};
                \draw[fill = red,draw=red] (s1) circle (0.5pt);
                \node[red,anchor=south east] at (s2) {$s_2$};
                \draw[fill = red,draw=red] (s2) circle (0.5pt);
                \node[red,anchor=west] at (s3) {$s_3$};
                \draw[fill = red,draw=red] (s3) circle (0.5pt);
                \node[red,anchor=north east] at (s4) {$s_4$};
                \draw[fill = red,draw=red] (s4) circle (0.5pt);

                \node[blue,anchor=south west] at (s1p) {$s_1'$};
                \draw[fill = blue,draw=blue] (s1p) circle (0.5pt);
                \node[blue,anchor=south east] at (s2p) {$s_2'$};
                \draw[fill = blue,draw=blue] (s2p) circle (0.5pt);
                \node[blue,anchor=west] at (s3p) {$s_3'$};
                \draw[fill = blue,draw=blue] (s3p) circle (0.5pt);
                \node[blue,anchor=north east] at (s4p) {$s_4'$};
                \draw[fill = blue,draw=blue] (s4p) circle (0.5pt);

        \end{tikzpicture}    
        \caption{The two tetrahedra in the finite chart 
                $\left\{ x_4\neq0 \right\}\subset \Pj^{3}$.}
        \label{fig:tetra}
\end{figure}

\begin{remark}
        \label{rem:invariants}
        The two invariants found in \Cref{eq:inveuler} are related
        to those known in the literature
        \begin{equation}
                \label{eq:inveuler_orig}
                F_{1} = 
                \frac{x_4^2-a_1 a_3 h^2 x_2^2/4}{x_4^2-a_1a_2 h^2 x_3^2/4},
                \quad
                F_{2} =
                \frac{x_4^2 - a_1a_2 h^2 x_3^2/4}{x_4^2-a_2a_3 h^2 x_1^2/4},
        \end{equation}
        (see \cite{Alonso_et_al2022, PetreraSuris2010,HirotaKimura2000}),
        through the transformation
        \begin{equation}
                \label{eq:funcdep}
                F_{1} = 1-\frac{a_1a_3 h^2}{4} I_{2},\quad 
                F_{2} = 
                \frac{1}{1-a_2a_3 h^2 I_1/4}.
        \end{equation}
\end{remark}

The following Lemma provides a motivation to study the Cremona-cubes group (see \cite[Proposition
7.2.]{Alonso_et_al2022} for more details).

\begin{lemma}
        \label{lem:eulerdecomposition}
        The KHK discretisation of the Euler top \eqref{eq:eulerxp}
        decomposes as $\Phi_{h} = \ell_{2}\circ \crem_{3} \circ \ell_{1}$
        where $\ell_{i}\colon \Pj^{3}\to\Pj^{3}$ are two projectivies
        whose representative matrices with respect to the standard homogeneous coordinate system in $\PGl(3,\C)$ are given by:
        \begin{subequations}
                \label{eq:ell12}
                \begin{align}
                        \label{eq:ell1}
                        M_{\ell_{1}} &= 
                        \begin{bmatrix}
                                \alpha_2 \alpha_3 h & \alpha_1 \alpha_3 h & -\alpha_1 \alpha_2 h & -2
                                \\ 
                                \alpha_2 \alpha_3 h & -\alpha_1 \alpha_3 h & -\alpha_1 \alpha_2 h & 2
                                \\
                                \alpha_2\alpha_3 h & -\alpha_1 \alpha_3 h & \alpha_1 \alpha_2 h & -2
                                \\ 
                                \alpha_2 \alpha_3 h & \alpha_1 \alpha_3 h & \alpha_1\alpha_2 h & 2
                        \end{bmatrix},
                        \\
                        \label{eq:ell2}
                        M_{\ell_{2}} &=
                        \begin{bmatrix}
                                2 \alpha_1 & 2 \alpha_1 & 2 \alpha_1 & 2 \alpha_1
                                \\ 
                                2 \alpha_2 & -2 \alpha_2 & -2 \alpha_2 & 2 \alpha_2
                                \\
                                -2 \alpha_3 & -2 \alpha_3 & 2 \alpha_3 & 2 \alpha_3
                                \\ 
                                \alpha_1 \alpha_2 \alpha_3 h & -\alpha_1 \alpha_2 \alpha_3 h & 
                                \alpha_1 \alpha_2 \alpha_3 h & -\alpha_1 \alpha_2 \alpha_3 h
                        \end{bmatrix}.
                \end{align}
        \end{subequations}
\end{lemma}
\begin{proof}
        The proof consists of a direct computation.
\end{proof}
\begin{remark}
        \label{rem:matrices}
        Sometimes, when no possible confusion occurs, we will identify any linear map $\ell \in \Bir(\Pj^{3})$ with the
        associated matrix $M_{\ell}\in\Pj\Gl(3,\C)$ and write $\ell$ in place of $M_{\ell}$.
\end{remark}

\begin{remark}
        The base loci of $\Phi_{h}$ and its inverse consist of the twelve lines on which lie the edges of two tetrahedra (see \Cref{fig:tetra}). Moreover, these lines are obtained as two by two intersections
        of the eight planes $\mathrm{K}_{i}$ and $\Lambda_{i}$ introduced in \eqref{piani}:
        \begin{equation}
                \label{eq:Lij}
                L_{i,j} = \mathrm{K}_{i}\cap\mathrm{K}_{j},
                \quad
                L_{i,j}' = \Lambda_{i}\cap\Lambda_{j},
                \quad 1 \leq i<j\leq 4.
        \end{equation}
        That is:
        \begin{equation}
                \label{eq:basephi}
                \Base \Phi_{h} = \bigcup_{1 \leq i<j\leq 4} L_{i,j},
                \quad
                \Base \Phi_{h}^{-1} = \bigcup_{1 \leq i<j\leq 4} L_{i,j}'.
        \end{equation}
\end{remark}

The importance of \Cref{lem:eulerdecomposition} is highlighted in
the following simple corollary.

\begin{corollary}
        \label{cor:biratequiv}
        The KHK discretisation of the Euler top \eqref{eq:eulerxp} is 
        projectively equivalent to the map $\Phi^{(0)} = g_0 \circ \crem_{3}$ where:
        \begin{equation}
                \label{eq:g0def}
                g_0 = 
                \begin{bmatrix}
                        1&-1&-1&-1
                        \\ 
                        -1&1&-1&-1
                        \\
                        -1&-1&1&-1
                        \\
                        -1&-1&-1&1
                \end{bmatrix}.
        \end{equation}
        Moreover, the projectivity $g_0$ is an involution.
\end{corollary}

\begin{proof} After noticing that $g_0=\ell_{1}\circ\ell_{2}$, the proof consists on a direct computation.
\end{proof}

\begin{remark}
    We remark that the conjugation of $\Phi_h$ by $\ell_1$ makes the
    discretisation parameter $h$ disappear. This implies that the 
    integrability properties of the system are \emph{independent from 
    the discretisation parameter} (see \cite[Introduction]{GubLat_sl2}).
\end{remark}

The great advantage in studying $\Phi^{(0)}$ with respect to the original
$\Phi_{h}$ is that $\Phi^{(0)}$ is the composition of two involutions, making it
strikingly similar to the QRT map construction \cite{QRT1989,QRT1988}.
Moreover, the rows and the columns of the matrix $g_0$ are made of
\emph{fixed points} of the standard Cremona transformation $\crem_{3}$.
This construction unveils the geometric structure underlying the KHK
discretisation of the Euler top \eqref{eq:eulerxp}. In this paper we consider a class of maps sharing similar behaviour (see \Cref{possibilemma}). The most important finding
is to consider birational maps decomposing as $\Phi = g \circ
\crem_{3}$, where $g\in\Pj\Gl(4,\C)$ is a projectivity having the following properties:
\begin{enumerate}
    \item\label{cond:0} up to the multiplicative action of $\C^*$ on the space of matrices, the entries of the matrix representing $g$ with respect the canonical projective coordinates belong to $\Set{-1,0,+1}$;
    \item \label{cond:1} the projectivity $g$ is represented by a matrix whose rows and columns are coordinate points or fixed points of the standard Cremona transformation $\crem_{3}$;
    \item\label{cond:2}  the order of  $g$ is finite (not necessarily an involution);
    \item\label{cond:3} for any $k\ge0$,  $g^{k}$ verifies (1),(2) and (3).
\end{enumerate}

The above requirements will be made mathematically rigorous in the next
section (see \Cref{possibilemma,def:cubegroup}), with the definition of
the \emph{Cremona-cubes group}.

\section{The Standard Cremona Transformation in dimension three}
\label{sec:cremona}

It is a classical fact known as the Noether--Castelnuovo theorem
\cite{Noether1870,Castelnuovo1901} that the Cremona group $\Bir (\Pj^2)$
of the projective plane is generated by $\Pj\Gl (3,\C)$ and the standard
Cremona transformation $\crem_2$, that is the birational map defined as:
\begin{equation}
        \label{eq:C2}
        \begin{tikzcd}[row sep=tiny]
                \Pj^2 \arrow[dashed]{r}{\crem_2} & \Pj^2 \\
                {[}x_1:x_2:x_3{]} \arrow[mapsto]{r} & \left[\frac{1}{x_1}:\frac{1}{x_2}:\frac{1}{x_3}\right].
        \end{tikzcd}
\end{equation}

The resolution of the indeterminacy locus of the map $\crem_2$ is
described in \Cref{Cremonaresolution2}. The two involved $\Pj^2$'s are
represented as triangles whose edges correspond to the coordinate lines,
while the exagon represents a del Pezzo surface of degree six, i.e. the blow-up of $\Pj^2$ at three non-collinear points, and its edges
correspond to $(-1)$-lines. The maps $\pi_1$ and $\pi_2$ are the blowups
of $\Pj^2$ with center the three coordinate points. In particular, the
dotted (resp. dashed) edges correspond to exceptional lines for $\pi_1$
(resp. $\pi_2$) (see \cite[Example 7.1.9]{Dolgachev2012book} for more
details).

\begin{figure}[H]
        \centering
        \begin{tikzpicture}
                \node at (-1.25,-1) {$\underset{\mbox{\normalsize $\Pj^2$}}{\begin{matrix}
                                        \begin{tikzpicture}
                                                \draw [dashed] (-0.3,-0.5)--(0.3,-0.5)--(0,0)--(-0.3,-0.5);
                                        \end{tikzpicture} 
                        \end{matrix}}$};
                \node at (1.25,-1) {$\underset{\mbox{\normalsize $\Pj^2$}}{\begin{matrix}
                                        \begin{tikzpicture}
                                                \draw [dotted] (-0.3,0.5)--(0.3,0.5)--(0,0)--(-0.3,0.5);
                                        \end{tikzpicture} 
                        \end{matrix}}$};
                \node at (0,0.5) {\begin{tikzpicture}\draw [dashed] (-0.2,-0.33)--(0.2,-0.33)(0.433,0)--(0.2,0.33)(-0.433,0)--(-0.2,0.33);
                                \node at (0,0) {$\dP{6}$};
                                \draw [dotted] (-0.2,0.33)--(0.2,0.33)(0.433,0)--(0.2,-0.33)(-0.433,0)--(-0.2,-0.33);
                \end{tikzpicture}};
                \draw[dashed,->] (-0.8,-0.7)--(0.8,-0.7);
                \draw[->] (0.4,0.2)--(0.9,-0.4);
                \node[right] at (0.55,0) {\footnotesize $\pi_2$};
                
                \draw[->] (-0.4,0.2)--(-0.9,-0.4);
                \node[left] at (-0.55,0) {\footnotesize $\pi_1$};
                \node[below] at (0,-0.7) {\footnotesize $\crem_2$};
        \end{tikzpicture}
        \caption{The resolution of the indeterminacies of the standard 
        Cremona transformation in dimension 2.}
        \label{Cremonaresolution2}
\end{figure}

The situation in higher dimension happens to be much more intricate. For
instance, it is no longer true that the Cremona group $\Bir(\Pj^n)$ of
the $n$-dimensional projective space is generated by $\Pj\Gl (n+1,\C)$
and the higher dimensional analogue of standard Cremona transformation.

This section focuses on the standard Cremona transformation $\crem_3$ of $\Pj^3$. We will describe the resolution of the
indeterminacies of $\crem_3$, and then we will discuss the configuration
of the fixed points of $\crem_3$.

\subsection{The standard Cremona transformation in dimension three}\label{sec:3dimcremona}

Let $\crem_3\in\Bir(\Pj^3)$ be the standard Cremona transformation (see \eqref{eq:C3}).  The map $\crem_3$
is well defined outside the union of the six coordinate axes of $\Pj^3$. We can solve the
indeterminacy locus of the standard Cremona transformation as follows
\begin{equation}
\begin{tikzcd}
        &B \arrow{dr}{f}\arrow[swap]{dl}{g}   \\
        \Pj^3 \arrow[dashed]{rr}{\crem_3} && \Pj^3
\end{tikzcd}
    \label{eq:cremres}
\end{equation}
where $B\subset ({\Pj^3})^{\times 2}$ is the closure of the graph
of $\crem_3$,  $B=\overline{\graph({\crem_3})} $, and $f$ and $g$
are the restrictions to $B$ of the canonical projections. Thus,
$B$ is the complete intersection defined as follows (see also \cite[Example
7.2.5]{Dolgachev2012book})
\begin{equation}
    B=\Set{([x_1:x_2:x_3:x_4],[y_1:y_2:y_3:y_4])\in ({\Pj^3})^{\times 2}| \begin{array}{c}
                    x_1y_1=x_2y_2\\
                    x_1y_1=x_3y_3\\
                    x_1y_1=x_4y_4
    \end{array}}.
    \label{eq:defB}
\end{equation}

Let us denote by $B_{ij}$ the affine chart on $B$ defined by
\begin{equation}
    B_{ij}=\Set{([x_1:x_2:x_3:x_4],[y_1:y_2:y_3:y_4])|x_i\not=0\mbox{ and }y_j\not=0}.
    \label{eq:Bijdef}
\end{equation}
Then, the chart $B_{ij}$ is smooth if and only if $i=j$, otherwise
$B_{ij}$ has an isolated conifold singularity. Thus, $B$ has twelve singular
points. Moreover, the exceptional divisors of $g$ and $f$ agree. They consist of the union of six surfaces isomorphic to $({\Pj^1})^{\times 2}$, each mapping, via $g$ and $f$, onto a coordinate line of $\Pj^3$, and three projective planes which are contracted by $g$ to coordinate points. In particular, the three coordinate points of each of the four
exceptional $\Pj^2$'s are the conifold singularities of the ambient
threefold $B$.

The variety $B$, which is the blowup of $\Pj^3$ along the union of the coordinate axes, can be alternatively constructed as follows (see \Cref{COSTRUZIONE}):
\begin{enumerate}
        \item blowup the four vertices of the standard tetrahedron of $\Pj^3$, i.e. the coordinate points of $\Pj^3$. Since we are blowing up (reduced) smooth points, the exceptional locus is the disjoint union of 4 copies of $\Pj^2$;
        \item blowup the strict transforms of the six edges of the standard tetrahedron, i.e. the six coordinate lines of $\Pj^3$. The exceptional locus of this blowup is given by six copies of $({\Pj^1})^{\times 2}$. The four exceptional divisors of the previous step happens to be blown up at three distinct non-collinear points. As a consequence, their strict transforms are all isomorphic to a del Pezzo surface of degree six $\dP{6}$;
        \item the last step consists in contracting twelve $(-1,-1)$-lines, 
            i.e. twelve lines $L_i\cong \Pj^1$ for $i=1,\ldots,12$ with 
            normal bundle
            \begin{equation}
                \normalb{L_i}{\widetilde{B}}\cong\mathcal O_{\Pj^1}(-1)^{\oplus 2}\mbox{ for }i=1,\ldots,12, 
                \label{eq:normalbundle}
            \end{equation}
        where $\widetilde{B}$ is the variety constructed in the previous
        step.  To see this, notice that each line is the complete
        intersection of one of the del Pezzo and the strict transform
        of some coordinate hyperplane of $\Pj^3$.

        The lines $L_i$, for $i=1,\ldots,12$, are the strict transforms
        of the coordinate lines of the exceptional divisors in
        step (1). Notice (see \cite[Section 11.3]{HUYB}) that all the $L_i$'s
        are contracted to conifold points.
\end{enumerate}

\begin{figure}[H]
        \centering
        \begin{tikzpicture}
                \draw (0,0)--(1.26,0)--(0.63,-0.77)--(0,0);
                \draw (0,0)--(-0.63,0.77);
                \draw (1.26,0)--(1.89,0.77);
                \draw (0.63,-0.77)--(0.63,-1.77);
                \node at (0.63,-2.3) {(1)};
        \end{tikzpicture}
        \begin{tikzpicture}
                \node at (-2,0) {$\ $};
                \node at (3,0) {$\ $};
                \node at (0.5,-1.3) {(2)};
                \draw (0,0)--(1,0)--(1.63,0.77)--(1,1.54)--(0,1.54)--(-0.63,0.77)--(0,0);
                \draw (1.63,0.77)--(2.26,1.54)--(1.63,2.31)--(1,1.54);
                \draw (0,1.54)--(-0.63,2.31)--(-1.26,1.54)--(-0.63,0.77);
                \draw (0,0)--(0,-1)--(1,-1)--(1,0);
        \end{tikzpicture}
        \begin{tikzpicture}
        
                \draw (1.63,0.77)--(2.26,1.54)--(1.63,2.31)--(1,1.54);
                \draw (1,1.54)--(0.37,2.31)--(-0.26,1.54)--(0.37,0.77);
                \draw (1.63,0.77)--(1.63,-0.23)--(0.37,-0.23)--(0.37,0.77);
                \draw (1.63,0.77)--(0.37,0.77)--(1,1.54)--(1.63,0.77);
                \node at (1,-0.8) {(3)};
        \end{tikzpicture}
        \caption{The toric description of the three steps in the constuction of $B$ near a coordinate point of $\Pj^3$.}
        \label{COSTRUZIONE}
\end{figure}

\begin{remark}
        When restricted to each exceptional $\Pj^2$ in step (1), the concatenation of step (2) and (3) is a two-dimensional standard Cremona transformation.
\end{remark}

\begin{remark}
    \label{puntifissi} We have:
    \begin{equation}
        \label{eq:fixc3}
        \Fix\crem_3 = \Calp \cup \Calq,
        \quad 
         \Calp=\Set{p_1,p_2,p_3,p_4},
        \,
        \Calq=\Set{q_1,q_2,q_3,q_4},
    \end{equation}
    where:
    \begin{equation}
        \begin{aligned}
                p_1&=[1:-1:-1:-1],  \\
                p_2&=[-1:1:-1:-1], \\
                p_3&=[-1:-1:1:-1],  \\
                p_4&=[-1:-1:-1:1], 
        \end{aligned}
        \quad
        \begin{aligned}
                q_1&=[1:-1:-1:1],  \\
                q_{2}&=[-1:1:-1:1],  \\
                q_{3}&=[1:1:-1:-1],  \\
                q_{4}&=[1:1:1:1].
        \end{aligned}
    \end{equation}
    These eight points correspond to two four-tuples of lines of 
    $\C^4$ orthogonal with respect to the standard scalar product. In particular, these are four-tuples of points in general position. We highlight that the sets $\Calp,\Calq$ correspond to the sets $\Set{s_i}_{i=1}^4,\Set{s_i'}_{i=1}^4$ via the projectivity $\ell_1^{-1}$ introduced in \eqref{eq:ell1}. Moreover, the points in $\Fix \crem_3$ can be interpreted as the vertices of a cube in the affine space, as depicted in \Cref{CUBE}. In this paper, by vertices of a cube we mean the base locus of a general net of quadrics of $\Pj^3$ (see  \cite[App. B.5.2]{HuntGeomQuot1996} and \cite[Section 1.5.2]{Dolgachev2012book}). Note that we are considering only general nets in order to have a 0-dimensional\footnote{The dimension of the base locus may jump in some special cases, an example being the twisted cubic.} reduced base locus.
\end{remark}

\begin{remark}
        \label{punticonfinement} 
        If we are interested in spaces of initial values for $\crem_3$, 
        we do not need to work with a resolution of singularities of the variety $B$ in \Cref{sec:3dimcremona},
        and it is enough to consider the variety
        \begin{equation}
        \overline{B}=\Bl_{\Cale}\Pj^3, \quad
                \Cale=\Set{e_1,e_2,e_3,e_4}    
        \end{equation}
        where
        \begin{equation}
            \label{eq:edef}
            \begin{aligned}
                e_1&=[1:0:0:0],  
                \\
                e_2&=[0:1:0:0],  
                \\
                e_3&=[0:0:1:0],  
                \\
                e_4&=[0:0:0:1].     
            \end{aligned}
        \end{equation}
        Indeed, the only divisorial contractions of $\crem_3$ consist of contractions over 
        one of the $e_i$'s and the map induced by $\crem_3$ on $\overline{B}$ is algebraically stable.
\end{remark}

\section{The Cremona-cubes group \texorpdfstring{$\mathcal C$}{}}
\label{sec:cubegroup}

In this section we introduce the subgroup of $\PGl(3,\C)$
we are interested in. The subgroup $\mathcal C$ is defined in terms
of the ``special'' points introduced in the previous section (see
\Cref{puntifissi,punticonfinement}).

\begin{definition}\label{specialpoints}
        We will denote\footnote{The letter $\Calr$ stands for Reye (see \Cref{thecubes}).} by $\mathcal{R}\subset \Pj^3$ the
        finite subset containing all the points appearing in
        \Cref{puntifissi,punticonfinement}, i.e.
        \begin{equation}
                \label{eq:setSdef}
                \mathcal{R}=\Cale \cup\Calp\cup\Calq.
        \end{equation}
\end{definition}

As explained in \Cref{sec:motivations} we are interested in maps of the
form $\Phi=g\circ \crem_3$ where $g\in \Pj\Gl(4,\C)$ is a projectivity of finite order
that acts on the set $\mathcal{R}$.

\begin{definition}\label{def:cubegroup}
        We will call the \textit{Cremona-cubes group} the subgroup 
        $\mathcal{C}$ of $\Pj\Gl(4,\C)$ defined by:
        \begin{equation}
                \label{eq:CremonaGroup}
                \mathcal{C}=
                \Set{g\in\Pj\Gl(4,\C)|g\cdot \mathcal{R}\subseteq \mathcal{R}}.
        \end{equation}
\end{definition}

\begin{remark}
        We remark that, since $\Calr $ contains five-tuples of points in general position, we have $\stab_{\langle g \rangle} 
        (\mathcal{R})=\langle\Id_{\Pj\Gl
        (4,\C)}\rangle$,  for any $g\in\Calc$ (see \cite[Section 1.3]{PARDINI}).  Here, $\langle g \rangle $ denotes the cyclic subgroup of $\Pj\Gl
        (4,\C)$ generated by $g$, while $\stab_{\langle g \rangle}(\mathcal{R})$ is the following subgroup of $\langle g \rangle$:
        \begin{equation}
            \stab_{\langle g \rangle}(\mathcal{R})=\Set{h\in \langle g \rangle | h|_{\Calr }\equiv \Id_{\Calr}}.
        \end{equation} This implies that all the elements
        $g\in\Calc$ have finite order. Indeed, suppose that there exists a $g\in\Calc$ of infinite order. In particular, for any integer $k>1$, $g^k$ is not the inverse of $g$. Now, since $g$ acts on the finite set $\Calr$, there is an integer $k>1$ such that $g|_{\Calr}\equiv g^k|_{\Calr}$. This implies that $g^{1-k}$ would be a non trivial element in $\stab_{\langle g \rangle} 
        (\mathcal{R})$. 
\end{remark}

The following result tells us that, within $\Calr$, the three subsets $\Cale$,
$\Calp$, and $\Calq$ are mapped between themselves {as a whole}.

\begin{lemma}\label{azioneindotta}
        The action of $\mathcal{C}$ on $\mathcal{R}$ induces an action of
        $\mathcal{C}$ on the set $\Set{\Cale,\Calp,\Calq}$.
\end{lemma}

\begin{proof}
        First notice that, if a line $L$ of $\Pj^3$ contains at least two points of $\Calr$, then it contains either three aligned points each
        belonging to one of the sets $\Cale,\Calp$ and $\Calq$ or two points
        from the same collection $\Cale,\Calp$ or $\Calq$.
        
        We now proceed by contradiction. Suppose, without loss of generality,
        that the projectivity $g$ sends the point $e_1$ to the point $p_1$
        and the point  $e_2$ to the point $q_2$, i.e.:
        \begin{equation}
                \label{eq:p1p2proof}
                \begin{tikzcd}[row sep=tiny]
                        e_1 \arrow[mapsto]{r}{g} & p_1 \\
                        e_2\arrow[mapsto]{r}{g} & q_2 .
                \end{tikzcd}
        \end{equation}
        Let $L_{12}$ be the line trough $p_1$ and $q_2$ and let $e_j$ be
        the third intersection point in $L_{12}\cap\Calr$, i.e.:
        \begin{equation}
                \label{eq:L59intersect}
                L_{12}\cap\Calr=\Set{e_j,p_1,q_2}.
        \end{equation}
        Then, we get 
        \begin{equation}
                g^{-1 }\cdot e_j \in (g^{-1}(L_{12})\smallsetminus\{e_1,e_2\})\cap\Calr.
                \label{eq:gpapplied}
        \end{equation} 
        Which is a contradiction.
\end{proof}

Now, we characterise the elements of $\Calc$ as belonging to three
different classes depending on their action on the set $\Set{\Cale,\Calp,\Calq}$. 
The following Lemma is crucial in this characterisation.

\begin{lemma}
        \label{possibilemma}
        Let $g\in\mathcal{C}\subset \Pj\Gl(4,\C)$ be an element
        of the Cremona-cubes group. Then, there is a matrix
        $\widetilde{g}\in\Gl(4,\C)$ representing $g$ whose entries
        belong to $\{-1,0,1\}$. Moreover, $g$ falls in one of the
        following cases.
        \begin{itemize}
                \item[\textit{(A)}] Both,  the columns and rows of $g$ represent the points in $\Calp$ (or in $\Calq$).
                \item[\textit{(B)}] The matrix $g$ is a permutation matrix with signs.
                \item[\textit{(C)}] The columns of $g$ represent the points in $\Calp$ and the rows represent the points in $\Calq$ (or viceversa).
        \end{itemize}
\end{lemma}

\begin{proof}
    The first part of the statement follows from the second, while the second
    part is  a direct consequence of \Cref{azioneindotta}. Indeed, as per
    \Cref{azioneindotta}, $g$ and $g^{-1}$ act on $\Set{\Cale,\Calp,\Calq}$
    and, depending on the action on this set, we get \textit{(A)},
    \textit{(B)} or \textit{(C)}.
\end{proof}

\begin{remark}
    Notis that, on the contrary, even if a projectivity of finite order $g\in\Pj\Gl(4,\C)$ verifies \textit{(A)}, \textit{(B)}, or \textit{(C)} it is not guaranteed that $g$ belongs to $\Calc$. Indeed, in general property  (\ref{cond:3}) in \Cref{sec:motivations} would not be satisfied.
\end{remark}

\begin{remark}\label{orbits}
    As a consequence of \Cref{possibilemma}, we can divide the elements of
    the Cremona-cubes group according to the orbit $\langle g \rangle\cdot
    \Cale$ of $\Cale$ via $g$:
    \begin{equation}\label{eq:orbit}
            \langle g \rangle\cdot \Cale=\Set{g^k\cdot e_i | k\in\N, \ 1\le i\le 4}.
    \end{equation}
    We have the following characterisation.
    \begin{itemize}
        \item An element $g \in \Calc$ belongs to case \textit{(A)} in
            \Cref{possibilemma} if and only if 
            $\langle g \rangle\cdot \Cale=\Cale\cup \Calp$ or 
            $\langle g \rangle\cdot \Cale=\Cale\cup \Calq$.
        \item An element $g \in \Calc$ belongs to case \textit{(B)} in 
            \Cref{possibilemma} if and only if 
            $\langle g \rangle\cdot \Cale=\Cale$.
        \item An element $g \in \Calc$ belongs to case \textit{(C)} in 
            \Cref{possibilemma} if and only if 
            $\langle g \rangle\cdot \Cale=\Cale\cup \Calp\cup \Calq$.        
    \end{itemize}
    We will see in \Cref{sec:entropy} that the orbit $\langle g \rangle
    \cdot \Cale$ \eqref{eq:orbit} plays a fundamental r\^ole in the
    confinement of singularities of the maps of the form $\Phi=g\circ
    \crem_{3}$ for $g\in\Calc$ (see also \Cref{thm:main}).
\end{remark}

\begin{definition}
    We will say that an element $g\in \Calc$ is of type \textit{(A)}
    (resp. of type \textit{(B)} or \textit{(C)}) if it belongs to
    the case \textit{(A)} (resp. \textit{(B)} or \textit{(C)}) in
    \Cref{possibilemma}.
\end{definition}

The following lemma investigates the relation between elements of the
Cremona-cubes group of different type.

\begin{lemma}\label{proprietagruppi}
        The following properties hold for the elements in $\Calc$ (see \Cref{possibilemma}).
        \begin{itemize}
                \item Two elements of type \textit{(A)} (resp.  \textit{(C)}) differ by multiplication by a permutation matrix with sign having an even number of -1 (which is an element of $\Calc$ of type \textit{(B)}).
                \item Two elements of type  \textit{(B)} differ by multiplication by a permutation matrix with signs, i.e. by an element of type \textit{(B)}. In particular, the elements of type \textit{(B)} form a subgroup of $\Calc$ that we will denote by $\Calc_{\mbox{\tiny\textit{(B)}}}$.
                \item An element of type \textit{(A)} differs by an element of type \textit{(C)} by multiplication by a permutation matrix with signs having an odd number of -1 (which is an element of $\Calc$ of type \textit{(B)}). 
                \item The inverse of an element of type \textit{(A)} (resp. \textit{(B)} or \textit{(C)}) is of  the same type.
        \end{itemize}
\end{lemma}
\begin{proof}
        The proof of the first three points consists on a direct check, while the fourth point is a direct consequence of \Cref{orbits}.
\end{proof}

\begin{remark}\label{rem:cardinality}
        Notice that the subgroup 
        \begin{equation}
            \Calc_{\mbox{\tiny\textit{(B)}}}=\Set{g\in\Calc|g\mbox{ is of type \textit{(B)}}}.
            \label{eq:Cb}
        \end{equation}
        has cardinality 192 and it contains four copies
        of the \textit{Full Octahedral Group} (see \cite[Section
        11.4]{NORMAN}) acting as the group of symmetries, non necessarily
        preserving the orientation, of the cube with vertices $\Calp\cup\Calq$. Precisely, chosen any coordinate $x_i$, for $i=1,\ldots,4$, the
        Full Ochtahedral Group can be identified with the elements in
        $\Calc_{\mbox{\tiny\textit{(B)}}}$ which preserve the hyperplane
        $\Set{x_i=0}$. In particular, we have
        \begin{equation}
       \Calc_{\mbox{\tiny\textit{(B)}}}=\left\langle
        \begin{bmatrix}
                1&0&0&0\\
                0&1&0&0\\
                0&0&1&0\\
                0&0&0&-1    
        \end{bmatrix},
        \begin{bmatrix}
                1&0&0&0\\
                0&1&0&0\\
                0&0&0&1\\
                0&0&1&0    
        \end{bmatrix},
        \begin{bmatrix}
                1&0&0&0\\
                0&0&1&0\\
                0&1&0&0\\
                0&0&0&1    
        \end{bmatrix},
        \begin{bmatrix}
                0&1&0&0\\
                1&0&0&0\\
                0&0&1&0\\
                0&0&0&1    
        \end{bmatrix}\right\rangle.
            \label{eq:Cbgen}
        \end{equation}
         Notice also that the only possible orders for the elements in $
         \Calc_{\mbox{\tiny\textit{(B)}}}$ are 2,3,4,6.
        Now, as a consequence of \Cref{proprietagruppi}, the
        Cremona-cubes group is the subgroup of $\Pj\Gl(4,\C)$ generated
        by $\Calc_{\mbox{\tiny\textit{(B)}}}$ and the projetivity $g_{0}$ of type
        \textit{(A)}  given in
        \eqref{eq:g0def}.
\end{remark}

As an immediate consequence, we get the following result.

\begin{corollary}[of \Cref{proprietagruppi}] \label{card} There are exactly 192 elements of each type \textit{(A)}, \textit{(B)} and \textit{(C)}.
\end{corollary}

\begin{theorem}
    \label{thm:cardinality}
    The cardinality of $\Calc $ is
    \begin{equation}
        |\Calc|=576.
        \label{eq:cdim}
    \end{equation}
\end{theorem}

\begin{proof}
    The statement is a direct consequence of \Cref{card}.  Alternatively,
    one can directly compute the cardinality of $\Calc$ using the
    computer software \texttt{Macaulay2} \cite{M2} with the package
    \texttt{InvariantRing} \cite{INVRING}.
\end{proof}

\begin{remark}
    \label{rem:identification}
    We observe that the Cremona-cubes group $\Calc$ is
    isomorphic to $((A_4\times A_4)\rtimes \Z/2\Z)\rtimes
    \Z/2\Z$, where $A_4<S_4$ is the alternating subgroup in the symmetric group of four elements, i.e. the subgroup consisting of permutations with even order.  This identification is obtained using the function
    \texttt{StructureDescription} of the system for computational
    discrete algebra GAP \cite{GAP4}\footnote{Using the function
    \texttt{IdGroup} we see that  $\Calc$ is the 8654-th finite
    group of order 576 of the finite groups database provided by
    GAP (see  $\texttt{SmallGroupInformation}$ \cite{GAP4}).}.
    On the other hand, in the same way, we have that the subgroup
    $\Calc_{\mbox{\tiny\textit{(B)}}}<\Calc$ is isomorphic $(((\Z/2\Z)^{\times
    4}\rtimes \Z/3\Z)\rtimes\Z/2\Z)\rtimes\Z/2\Z$\footnote{Analogously,
    $\CB$ is the 955-th finite group of
    order 192 of the finite groups database provided by GAP (see
    $\texttt{SmallGroupInformation}$ \cite{GAP4}).}.
\end{remark}

\subsection{The cubes of $\crem_3$}
\label{thecubes}

Let us now explain the origin of the name Cremona-cubes group.
The configuration of the points in $\Calr$ is known in the literature as
the \emph{Reye configuration} or also the \textit{$D_4$ configuration}
(see \cite{DOLGACHEV2,CONFIGURATION}). Alternatively, one can say that the
tetrahedra defined by $\Cale,\Calp$ and $\Calq$ constitute a desmic
triple (see \cite[App. B.5.2]{HuntGeomQuot1996}). 
Explicitly, the
elements of $\Calr$ are the points of a $(12_4\ 16_3)$ configuration,
i.e. 12 points and 16 lines with the property that there are four lines
trough each point and three points on each line.

Let us consider the following set of hyperplanes of $\Pj^3$
\begin{equation}
    \mathcal H =\Set{H_{ij}^\bullet| 1\le i<j\le 4,\ \bullet\in\{+,-\}},
    \label{eq:Hplanes}
\end{equation}
where $H^\bullet_{ij}=\Set{[x_1:x_2:x_3:x_4]\in\Pj^3|x_i=\bullet
x_j}$. Let us now focus on some affine chart
$U_{\overline{\imath}}=\set{x_{\overline{\imath}}\not=0}$, for some
$\overline{\imath}=1,\ldots,4$, with the usual affine coordinates
\begin{equation}
    X_j=\frac{x_j}{x_{\overline{\imath}}}, \quad j\in\Set{1,2,3,4}\smallsetminus\Set{\overline{\imath}}
    \label{eq:Xjdef}
\end{equation}
and denote by $\overline{\mathcal H}$ the following set of affine hyperplanes: 
\begin{equation}
    \overline{\mathcal H}=\Set{H\cap U_{\overline{\imath}} | H\in\mathcal{H}}.
    \label{eq:Hbardef}
\end{equation}

\begin{figure}[htb]
	\centering
    \tdplotsetmaincoords{70}{120}
    \begin{tikzpicture}[scale=1.5,tdplot_main_coords]
        \draw[color=black,thick,dashed] (0,0,0) -- (1,0,0);
        \draw[color=black,thick,->] (1,0,0) -- (4,0,0) node[anchor=north east,red]{$e_2$};
        \draw[color=black,thick,dashed] (0,0,0) -- (0,1,0);
        \draw[color=black,thick,->] (0,1,0) -- (0,4,0) node[anchor=north west,red]{$e_3$};
        \draw[color=black,thick,dashed] (0,0,0) -- (0,0,1);
        \draw[color=black,thick,->] (0,0,1) -- (0,0,3) node[anchor=south,red]{$e_4$};

        \coordinate (p1) at (-1,-1,-1);
        \coordinate (p2) at (-1,1,1);
        \coordinate (p3) at (1,-1,1);
        \coordinate (p4) at (1,1,-1);
        
        \coordinate (q1) at (-1,-1,1);
        \coordinate (q2) at (-1,1,-1);
        \coordinate (q3) at (1,-1,-1);
        \coordinate (q4) at (1,1,1);
                
        \draw[blue,thick] (q1)--(p2)--(q4)--(p3)--cycle;
        \draw[blue,thick] (q4)--(p4)--(q2)--(p2);
        \draw[blue,thick] (p3)--(q3)--(p4);
        \draw[blue,thick,dashed] (q3)--(p1)--(q2);
        \draw[blue,thick,dashed] (p1)--(q1);
        
        \node[blue,anchor=south west] at (p1) {$p_1$};
        \draw[fill = blue,draw=blue] (p1) circle (1pt);
        \node[blue,anchor=south west] at (p2) {$p_2$};
        \draw[fill = blue,draw=blue] (p2) circle (1pt);
        \node[blue,anchor=south] at (p3) {$p_3$};
        \draw[fill = blue,draw=blue] (p3) circle (1pt);
        \node[blue,anchor=north] at (p4) {$p_4$};
        \draw[fill = blue,draw=blue] (p4) circle (1pt);
        
        \node[blue,anchor=south west] at (q1) {$q_1$};
        \draw[fill = blue,draw=blue] (q1) circle (1pt);
        \node[blue,anchor=south west] at (q2) {$q_2$};
        \draw[fill = blue,draw=blue] (q2) circle (1pt);
        \node[blue,anchor=north] at (q3) {$q_3$};
        \draw[fill = blue,draw=blue] (q3) circle (1pt);
        \node[blue,anchor=south] at (q4) {$q_4$};
        \draw[fill = blue,draw=blue] (q4) circle (1pt);
        
        \node[red,anchor=south west] at (0,0,0) {$e_1$};
        \draw[fill = red,draw=red] (0,0,0) circle (1pt);
        \draw[fill = red,draw=red] (4,0,0) circle (1pt);
        \draw[fill = red,draw=red] (0,4,0) circle (1pt);
        \draw[fill = red,draw=red] (0,0,3) circle (1pt);
    \end{tikzpicture}  
    
	\caption{The configuration, in the chart $\{x_1\not=0\}\subset \Pj^3$, of the points in $\Calr$.}
	\label{CUBE}
\end{figure}

There are twelve planes in $\overline{\mathcal{H}}$. Six of them are
faces of the  cube with vertices $\Calp\cup \Calq$ and the remaining
six are symmetry planes of the cube containing pairs of parallel lines.
Using these planes, an alternative description of the Reye configuration is given
by Hilbert in terms in terms of sextuples of co-planar points \cite{HilbertCV1952}.

Now, a direct check shows that the Cremona-cubes group acts (with trivial
stabilizer) on the set $\overline{\mathcal H}$. This observation provide
a different point of view which allows one to work in an affine setting.

Another way to understand the action of $\Calc$ is to look at the three
cubes $C_{\Cale\Calp}$, $C_{\Cale\Calq}$ and $C_{\Calp\Calq}$ which
have vertices respectively $\Cale\cup \Calp$, $\Cale\cup \Calq$ and
$\Calp \cup \Calq$. While the group $\Calc_{\mbox{\tiny\textit{(B)}}}$
acts by just swapping $C_{\Cale\Calp}$ and $C_{\Cale\Calq}$ (and keeping
$C_{\Calp\Calq}$ fixed), the Cremona-cubes group $\Calc$ permutes the
three cubes.

In what follows, the two net of quadrics respectively generated by $\Cale \cup \Calp $ and $\Cale\cup \Calq$ will play a fundamental r\^ole.
\begin{definition}\label{def:nets}
    We will call $\Sigma_{\Calp}$ and $\Sigma_{\Calq}$ the nets of
    quadrics respectively generated by ${\Cale\cup\Calp}$, ${\Cale\cup\Calq}$. Precisely, we have:
    \begin{subequations}
        \begin{align}
			\label{eq:web1} 
			\Sigma_{\Calp} &=
			\Set{S_{\alpha,\beta,\gamma}\subset \Pj^3 | [\alpha:\beta:\gamma]\in\Pj^2,\  S_{\alpha,\beta,\gamma}=\Set{\alpha S_1^{(+)}+ \beta S_2^{(+)} +\gamma S_3^{(+)}=0 }},
			\\
			\label{eq:web2} 
			\Sigma_{\Calq} &=
			\Set{S_{\alpha,\beta,\gamma}\subset \Pj^3 | 
				[\alpha:\beta:\gamma]\in\Pj^2,\ S_{\alpha,\beta,\gamma}=\Set{\alpha S_1^{(-)}+ \beta S_2^{(-)} +\gamma S_3^{(-)}=0 }}
			,
        \end{align}    
    \end{subequations}
    where:
    \begin{equation}
            \label{eq:quadricpm}
            S_1^{(\pm)} = x_1 x_2\pm x_4 x_3,
            \quad
            S_2^{(\pm)} = x_1 x_3 \pm x_4 x_2,
            \quad 
            S_3^{(\pm)} = x_1 x_4 \pm x_2 x_3.
    \end{equation}
\end{definition}

\begin{remark}\label{nopointperS}
    Note that the projective subspaces (planes) $\Sigma_{\Calp},
    \Sigma_{\Calq}, \subset\Pj H^0(\Pj^3,\Calo_{\Pj^3}(2))\cong\Pj^9$
    do not intersect, i.e.
    \begin{equation}
        \label{eq:intersection}
        \Sigma_{\Calp} \cap \Sigma_{\Calq}= \emptyset.
    \end{equation}
    In other words there is no quadric passing trough all the points
    in $\Calr$.
\end{remark} 

We will see in \Cref{sec:covariants,sec:invariants} that these nets
are key tools in the computation of the invariants of the maps of the
form $g\circ\crem_3$ for $g\in \Calc$.

\begin{remark}\label{desmicclassic}
    In what follows we shall consider the pencil of
    quartics having nodes precisely at $\Calr$ (see in
    \Cref{sss:caseBinvariants,sss:casseCinva}). Its existence is a
    classical fact (see \cite[App. B.5.2]{HuntGeomQuot1996}) and it is
    called \emph{desmic pencil}. In particular, a desmic pencil  contains
    exactly three reducible members called \emph{desmic quartics}.
\end{remark}

\section{The spaces of initial values and the algebraic entropy}
\label{sec:entropy}

In this section we study the confinement of the indeterminacies
of the maps $\Phi\in\Bir(\Pj^3)$ of the form $\Phi=g\circ \crem_3$, for
$g\in\Calc$, via some space of initial values. We will prove in
\Cref{prop:caseAgrowth,prop:caseC,prop:caseBgrowth} that the growth
behaviour of $d_n^\Phi$ only depends on the type
of $g\in\Calc$. The same result is also true for the maps of the form
$\Psi=\crem_3\circ g$. Indeed, the birational transformations $\Phi =
g\circ \crem_3$ and $\Psi=\crem_3\circ g$ have the same algebraic entropy
because they are conjugated via $g^{-1}$ which is of the same type of $g$ as per \Cref{proprietagruppi} (see \Cref{rem:algent}).

The results contained in this section can be seen either directly,
or using the approach described in \cite{BedfordKim2004} for maps $\Phi$
coming from the composition of the Cremona map $\crem_M$  \eqref{eq:C3}
with projectivities using
the method of \emph{singular orbits}, i.e. sequences of the form 
\begin{equation} \label{eq:patternkim}
    \begin{tikzcd}
    \Set{x_i=0} \arrow{r}{\Phi} & {m}_1 \arrow{r}{\Phi} &
    \cdots\arrow{r}{\Phi}& m_k\arrow{r}{\Phi}& e_j
\end{tikzcd}
\end{equation}
where $1\le i,j \le 4$, $k\ge 0$, $m_1,\ldots,m_k \in \Pj^M$. Notice
that \eqref{eq:patternkim} is a special instance of \eqref{eq:singpatt}.
We remark that the authors in \cite[Theorem A.1]{BedfordKim2004} gave the
general expression of the matrix representing the action in cohomology
induced by a map of the form $\ell \circ \crem_M$ with with $\ell \in
\PGl(M+1,\C)$, with singular orbits of prescribed form. In general,
this gives an estimate of the degree growth, and for some choices
of the matrix $\ell$ it allows direct computations. In this paper,
we compute explicitly the sequence of degrees $\Set{d_{n}}_{n\in\N}$
using induction on the action induced in cohomology over a generic plane
$H\subset \Pj^{3}$ of map $\Phi = g\circ \crem_3$, $g\in\Calc$. For
instance, this approach provides us formulas \eqref{eq:dnen} leading to
equation \eqref{eq:entropyB} in \Cref{prop:caseBgrowth}.

\setcounter{subsection}{-1}
\subsection{The case of the standard Cremona transformation} 

We start by reviewing the case $\Phi=\crem_3$
(see. \cite{CarsteaTakenawa2019JPhysA} for more details).  Let
$\overline{\varepsilon} \colon \overline{B}\to \Pj^{3}$ be the blowup of $\Pj^3$ with center
the set of coordinate points. Let us denote by $E_{i}$
the exceptional divisor over the coordinate point $e_{i}$, for
$i=1,\ldots,4$. Then, one can choose (see \cite[Section 4.6.2]{GRIHARR})
the following basis of the second singular cohomology group of $\overline{B}$:
\begin{equation}
	\label{eq:H2cremona} H^{2}( \overline{B},\Z ) = \langle
	\varepsilon^*H,E_1,E_2,E_3,E_4\rangle_{\Z},
\end{equation} 
where $H$ is the class of an hyperplane in $\Pj^3$ and, with abuse
of notation, we have denoted by the same symbols the exceptional
divisors $E_i$, for $i=1,\ldots,4$, and their cohomology classes.
Moreover, the action of the standard Cremona transformation on the
second cohomology group $H^{2}(\overline{B},\Z)$ is expressed, in terms of the
basis \eqref{eq:H2cremona}, by the following matrix
\begin{equation}
    \label{CREMONACTION} 
(\crem_3^{-1})^* =   \crem_3^*=
    \begin{pmatrix} 
        3 & 1 & 1 & 1 & 1
        \\
	-2 & 0 & -1 & -1 & -1
        \\
        -2 & -1 & 0 & -1 & -1
        \\
        -2 & -1 & -1 & 0 & -1
        \\ 
        -2 & -1 & -1 &- 1 & 0
    \end{pmatrix}
\end{equation} 
see for instance \cite{CarsteaTakenawa2019JPhysA} 
or \cite[Eq. (3.1)]{BedfordKim2004} evaluated at $d=3$.

\begin{remark}
    Since $\Calr\smallsetminus\Cale = \Fix \crem_3$, also $\Bl_{\Calr}\Pj^3$ 
    is a space of initial values for $\crem_3$, but it is superfluous 
    to blowup all $\Calr$.
\end{remark}

We divide the study of the confinement of the indeterminacies accordingly
to the possible types of $g$, namely \textit{(A)}, \textit{(B)} or
\textit{(C)}.

From now on, given any $g\in\Calc$, we will denote by
$\varepsilon_g:B_g\rightarrow \Pj^3$ the blowup of $\Pj^3$ with center
the finite set $\langle g \rangle\cdot \Cale$. Notice that, since $B_g$
is the blowup of $\Pj^3$ with center an $n$-tuple of distinct reduced
points, the class of its canonical divisor is
\begin{equation}
    \label{anticanonical}
    -K_{B_g}=4\varepsilon_{g}^*H-2 \underset{a\in\langle g\rangle\cdot \Cale}{\sum}D_a,
\end{equation}
where $D_a$ is the exceptional divisor over the point $a\in\langle g \rangle\cdot \Cale$ (see \cite[Section 1.4.2]{GRIHARR}).

Sometimes, with abuse of notation we will denote by
$\Phi:B_g\dashrightarrow B_g$ the natural lift of a projective map
$\Phi\in\Bir(\Pj^3)$ to the variety $B_g$. Moreover, when no ambiguity
is present, we will denote by the same symbol a divisor and its
cohomology class.

\begin{notation}
    In what follows, given a birational map $\Phi\in\Bir(X)$, for $X$
    smooth projective variety, we will denote by $\Phi_{*}$ the linear
    operator
    \begin{equation}
        \Phi _*=(\Phi^{-1})^*\in\Gl(H^2(X,\Z)).
        \label{eq:pushforward}
    \end{equation}
\end{notation}

Before proceeding, we recall that, if $\varepsilon\colon B\rightarrow \Pj^3$
is a space of initial values for a birational map $\Phi\in\Bir(\Pj^3)$,
then the degree $d_{n}^{\Phi}$ of the $n$-th iterate  is given by formula
\eqref{eq:dncoeff} (see \Cref{sec:intro}).

\subsection{Case \textit{(A)}}\label{CASEA}

Let $g\in\Calc$ be a projectivity of type \textit{(A)}. Without loss of generality,
we can suppose that $g$ acts  on the set
$\{\Cale,\Calp,\Calq\}$ by swapping $\Cale $ and $\Calp$ (see \Cref{azioneindotta}). Recall  that, in this case, 
the blowup $B_g$ of $\Pj^3$ with center $\langle g \rangle\cdot \Cale=
\Cale\cup \Calp$ is a space of initial values (see
\Cref{orbits}). Let us fix the following basis of the second cohomology
group of $B_{g}$: \begin{equation}
	\label{eq:H2caseA} H_{\mbox{\tiny\textit{(A)}}}^2 := H^{2}\left(
	B_{g},\Z \right) = \langle \varepsilon_{g}^*H,E_1,E_2,E_3,E_4,P_1,P_2,P_3,P_4\rangle_{\Z},
\end{equation} where, as above, $H$ is the class of an hyperplane in
$\Pj^3$, $E_i$ is the cohomology class of the exceptional divisor over
the point $e_i$, for $i=1,\ldots,4$, and $P_i$ is the cohomology class
of the exceptional divisor over the point $p_i$, for $i=1,\ldots,4$. We
want to compute the action induced by $\Phi=g\circ \crem_3$ on the
cohomology group $ H_{\mbox{\tiny\textit{(A)}}}^2$. Equivalently, we
want to compute the matrix representing $\Phi_*$ with respect to the
basis \eqref{eq:H2caseA}.

First, notice that the action of the standard Cremona transformation on
$\varepsilon_{g}^*H, E_1, \ldots, E_4$ is given by the matrix
\eqref{CREMONACTION} while the elements $P_i $ are fixed by ${\crem_3}_*$
because they lie over the points $p_i$ which are fixed by $\crem_3$.

Since $g$ sends hyperplanes to hyperplanes and it swaps the sets $\Cale$
and $\Calp$, the matrix representing $g_*$ with respect to the basis
\eqref{eq:H2caseA} has the following block decomposition: \begin{equation}
	\label{blocks}
	g_*=\begin{pmatrix}
		\begin{tikzpicture}
			\draw (0,0)--(0,-3); \draw (1.1,0)--(1.1,-3); \draw
			(-0.5,-0.5)--(2.25,-0.5); \draw (-0.5,-1.75)--(2.25,-1.75);
			\node at (-0.25,-0.25) {\small 1};
			
			\node at (-0.25,-0.75) {\small 0}; \node at (-0.25,-1) {\tiny
				$\vdots$}; \node at (-0.25,-1.5) {\small 0};
			
			\node at (-0.25,-2) {\small 0}; \node at (-0.25,-2.25) {\tiny
				$\vdots$}; \node at (-0.25,-2.75) {\small 0};
			
			\node at (0.25,-0.25) {\small 0}; \node at (0.6,-0.25) {\small
				$\cdots$}; \node at (0.95,-0.25) {\small 0};
			
			\node at (1.3,-0.25) {\small 0}; \node at (1.65,-0.25) {\small
				$\cdots$}; \node at (2,-0.25) {\small 0};
			
			\node at (0.5,-1.1) {\Huge $0$};
   \node at (0.6,-2.4) {\Huge $M_1$};
			\node at (1.7,-1.1) {\Huge $M_2$}; \node at (1.5,-2.4) {\Huge $0$};
			
		\end{tikzpicture}
\end{pmatrix} \end{equation} where $M_1$ and $M_2$ are $4\times 4$ permutation
matrices.
In terms of singular orbits \cite{BedfordKim2004} this implies that 
there are four closed singular orbits of length two.

We finally obtain the action of $\Phi_{*}$ on
$H_{\mbox{\tiny\textit{(A)}}}^2$ as the composition $\Phi_*=g_*\circ
{\crem_3}_*$. Precisely, there exist two permutations $\sigma_1,\sigma_2$
in the symmetric group $\SSym_4$ of four elements $\{1,2,3,4\}$ such
that \begin{equation}
	\label{eq:actionA} \begin{tikzcd}[row sep=tiny] \varepsilon_{g}^{*}H
		\arrow[mapsto]{r}{\Phi_*} & 3 \varepsilon_{g}^{*}H -
		2\ssum{j=1}{4} P_{j},\\
  E_{i} \arrow[mapsto]{r}{\Phi_*}
		&\varepsilon_{g}^* H-\ssum{j\neq i}{}P_{\sigma_{1}(j)}&\mbox{for }i=1,\ldots,4,\\
		P_{i} \arrow[mapsto]{r}{\Phi_*}& E_{\sigma_{2}(i)}&\mbox{for }i=1,\ldots,4.
	\end{tikzcd}
\end{equation} Notice that $\sigma_1$ (resp. $\sigma_2$) corresponds to
the block $M_1$ (resp. $M_2$).
The same action can be recovered from the four singular orbits using \cite[Eqs. (4.1,4.3)]{BedfordKim2004}.

Then, we obtain the main result of this subsection.
\begin{proposition}
    \label{prop:caseAgrowth} Consider the birational map  $\Phi=g\circ \crem_{3} \in\Bir (\Pj^{3})$ for $g\in\Calc$ of type \textit{(A)}. Then, the following formula is true for all $n\in\N$:
    \begin{equation}
        \label{eq:caseAphinH} 
        (\Phi_{*})^{n}( \varepsilon_{g}^{*}H ) 
        = 
        \left( 2n^{2}+1 \right) \varepsilon_{g}^{*}H  - n\left( n-1
        \right)\sum_{j=1}^{4} E_{j}-
        n\left( n+1 \right)\sum_{j=1}^{4} P_{j}.
    \end{equation} 
    As a consequence, we have $d_n^\Phi=2n^2+1$, that is the map $\Phi$
    is \emph{integrable} according to the algebraic entropy.
\end{proposition}

\begin{proof} 
    First notice that the following formulas are a direct
    consequence of \Cref{eq:actionA}: 
    \begin{subequations}
        \label{evalA} \begin{align}
        \Phi_{*}\left(\ssum{j=1}{4}E_j\right)
                    & =4\varepsilon_{g}^*H-3\ssum{j=1}{4}P_j,\\
                    \Phi_{*}\left(\ssum{j=1}{4}P_j\right) &=
                    \ssum{j=1}{4}E_{\sigma_2(j)}=\ssum{j=1}{4}E_{j}
                    .
            \end{align}
    \end{subequations} Now the proof goes by induction on $n\in\N$. The case
    $n=0$ is trivial, while the case $n=1$ follows from \Cref{eq:actionA}.
    We move now to the proof of the inductive step. Suppose that
    formula \eqref{eq:caseAphinH} holds true for some $n\in\N$, we want to
    prove that it is also true for $n+1$. This can be shown as follows:
    \begin{equation}
        \begin{aligned}
            (\Phi_{*})^{n+1}( \varepsilon_{g}^{*}H )
            &=\Phi_*\left[\left(\Phi_*\right)^n\left(\varepsilon_{g}^*H\right)\right]
            \\ 
            &=
            \Phi_* \left[\left( 2n^{2}+1 \right) \varepsilon_{g}^{*}H  - n\left( n-1
        \right)\sum_{j=1}^{4} E_{j}-
        n\left( n+1 \right)\sum_{j=1}^{4} P_{j}\right]
            \\ 
            &=
            \left( 2n^{2}+1\right)\left(3\varepsilon_{g}^*H-2\ssum{j=1}{4}P_j\right) - n\left( n-1 \right)\ssum{j=5}{8}
            \left(4\varepsilon_{g}^*H-3\ssum{j=1}{4}P_j\right)
            \\
            &
            \phantom{\left( 2n^{2}+1\right)\left(3\varepsilon_{g}^*H-2\ssum{j=1}{4}E_j\right)}
            \quad
            - n\left( n+1 \right)\ssum{j=1}{4} E_{j}            
            \\ 
            &= \left[
            2\left( n+1 \right)^{2}+1 \right] \varepsilon_{g}^{*}H - n\left(
            n+1 \right)\ssum{j=1}{4} E_{j} - \left( n+1
            \right)\left( n+2 \right)\ssum{j=1}{4} P_{j},
        \end{aligned}
        \label{eq:caseAcomputations}
    \end{equation}
    where the third equality is a consequence of \Cref{evalA}.
    From \Cref{rem:algent}, since $d_{n}^{\Phi}$ is subexponential,
    we have $S_{\Phi}=0$, and hence the statement follows.
\end{proof}

\subsection{Case \textit{(B)}}\label{subsec:case3}

In this case, a space of initial values is $B_g=\Bl_{\Cale}\Pj^3$
and we can chose as basis of $H^2(B_g,\Z)$ the one given in
\Cref{eq:H2cremona}. Recall that there is a matrix representing $g$ which is a
permutation matrix with signs. Since we will focus on the action induced
on the second singular cohomology group of $B_g$ by the map $\Phi =
g\circ \crem_3$, it is enough to study the action on set $\Cale$. Thus,
it is enough to suppose that $g$ is a permutation matrix in the usual
sense. As a consequence, the matrix representing $g_*$ in our basis
is a permutation matrix with $\varepsilon_{g}^*H$ as eigenvector, i.e.
\begin{equation}
	g_*=\begin{pmatrix}
		\begin{tikzpicture}
			\draw (0,0)--(0,-1.75); \draw (-0.5,-0.5)--(1.1,-0.5);
			
			\node at (-0.25,-0.25) {\small 1};
			
			\node at (-0.25,-0.75) {\small 0}; \node at (-0.25,-1) {\tiny
				$\vdots$}; \node at (-0.25,-1.5) {\small 0};
			
			\node at (0.25,-0.25) {\small 0}; \node at (0.6,-0.25) {\small
				  $\cdots$}; \node at (0.95,-0.25) {\small 0};
			
			\node at (0.5,-1.1) {\Huge $M$};
		\end{tikzpicture}
	\end{pmatrix}
\end{equation}
where $M$ is a $4\times 4$ permutation matrix.  
In terms of singular orbits \cite{BedfordKim2004} this implies that 
there are four closed singular orbits of length one.

After composing $g_*$ with \Cref{CREMONACTION}, we obtain the action of
$\Phi_{*}$ on $H_{\mbox{\tiny\textit{(B)}}}^2$: \begin{equation}
	\label{eq:actionC} \begin{tikzcd}[row sep=tiny] \varepsilon_{g}^{*}H
		\arrow[mapsto]{r}{\Phi_*} & 3 \varepsilon_{g}^{*}H - 2\ssum{j=1}{4}
		E_{j},\\
		E_i \arrow[mapsto]{r}{\Phi_*}& \varepsilon_{g}^{*} H - \ssum{j\neq
			i}{} E_{\sigma(j)}&\mbox{for }i=1,\ldots,4,
	\end{tikzcd}
\end{equation} where $\sigma $ is the element in $\SSym_4$ corresponding
to the matrix $M$.
The same action can be recovered from the four singular orbits using \cite[Eqs. (4.1,4.3)]{BedfordKim2004}.

As a consequence, we have the main result of this subsection.

\begin{proposition}
    \label{prop:caseC}
    Let $g\in \Calc$ be an element of type \textit{(B)}. Then, for the map $\Phi=g\circ\crem_3 \in\Bir
    (\Pj^{3})$, the following formulas hold true
    for all $n\in\N$:
    \begin{subequations}
            \label{eq:caseCnthiter} \begin{align}
                    \label{eq:caseCHnthiter}
                    \left(\Phi_{*}\right)^{n}(\varepsilon_{g}^{*}H )
                    &= \begin{cases} 3 \varepsilon_{g}^{*}H - 2\ssum{j=1}{4}
                            E_{j} & \text{for $n$ odd}, \\ \varepsilon_{g}^{*}H &
                            \text{for $n$ even}, \end{cases} \\ \label{eq:caseCEinthiter}
                    \left(\Phi_{*}\right)^{n}\left(E_{i}\right) &= \begin{cases}
                            \varepsilon_{g}^{*} H - \ssum{j\neq \sigma^{n-1}(i)}{}
                            E_{\sigma^{n}(j)} & \text{for $n$ odd}, \\ E_{\sigma^{n}(i)}
                            & \text{for $n$ even},
                    \end{cases}&\mbox{for }i=1,\ldots,4.
            \end{align}
    \end{subequations}
    So, the map $\Phi$ is \emph{periodic} (see \Cref{def:integrability}), and hence integrable, with two-periodic degrees: 
    \begin{equation}
	\label{eq:dncaseC} 
        d_{n}^\Phi =
	\begin{cases} 3 & \text{for $n$ odd}, \\ 1 & \text{for $n$ even}.
	\end{cases}
    \end{equation}
\end{proposition} 

\begin{proof}
    As in the proof of \Cref{prop:caseAgrowth} 
    we proceed by evaluating the action of $\Phi_{*}$ on the combination
    $\ssum{j=1}{4}E_j$. 
    From \eqref{eq:actionC}, we have:
    \begin{equation}
        \Phi_{*}\left( \ssum{i=1}{4}E_j \right) =
        4 \varepsilon_{g}^{*} H - 3\ssum{j=1}{4}E_{j},
        \label{eq:caseCsumEj}
    \end{equation}
   which, again from \eqref{eq:actionC} implies:
    \begin{equation}
        \begin{aligned}
            \left(\Phi_{*}\right)^{2}(  \varepsilon_{g}^{*} H )
            &=
            \Phi_{*}\left(  3 \varepsilon_{g}^{*}H - 2\ssum{j=1}{4}
            E_{j}\right)
            \\
            &= 3 
            \left( 3 \varepsilon_{g}^{*} H - 2\ssum{j=1}{4}E_{j}\right)
            -2\left[
            4 \varepsilon_{g}^{*} H - 3\ssum{j=1}{4}E_{j}\right]
            \\
            &=\varepsilon_{g}^{*} H.
        \end{aligned}
        \label{eq:caseCcompH}
    \end{equation}
    As a consequence, we get formula \eqref{eq:caseCHnthiter}.
    On the other hand, still from \eqref{eq:actionC} and \eqref{eq:caseCsumEj},
    we have:
    \begin{equation}
        \begin{aligned}
            \left(\Phi_{*}\right)^{2}\left(  E_{i}\right)
            &=
            \Phi_{*}\left(  \varepsilon_{g}^{*}H - \ssum{j\neq i}{}
            E_{\sigma(j)}\right)
            \\
            &= 3 \varepsilon_{g}^{*} H - 2\ssum{j=1}{4}E_{j}
            -
            \Phi_{*}\left( \ssum{j=1}{4}E_{j} - E_{\sigma(i)}\right)
            \\
            &= 3 \varepsilon_{g}^{*} H - 2\ssum{j=1}{4}E_{j}
            - \left[4 \varepsilon_{g}^{*} H - 3\ssum{j=1}{4}E_{j}\right]
            + \varepsilon_{g}^{*}H - \ssum{j\neq \sigma^{2}(i)}{} E_{j}
            \\
            &= E_{\sigma^{2}(j)}.
        \end{aligned}
        \label{eq:caseCcompEi}
    \end{equation}
    Again , we obtain formula \eqref{eq:caseCEinthiter}.
    The periodicity of the degrees of $\Phi$ follows immediately from 
    formulas \eqref{eq:caseCnthiter} and \eqref{eq:dncoeff}.
    The algebraic entropy of a limited sequence is clearly zero,
    and this ends the proof.
\end{proof}

An immediate corollary of \Cref{prop:caseC} concerns the periodicity
of $\Phi_*$ for $\Phi\in\Bir(\Pj^3)$ of the form $g\circ \crem_3$, for $g\in\Calc$ of type \emph{(B)}.

\begin{corollary} 
	\label{rem:periodic}  In the hypotheses of \Cref{prop:caseC}, we have
        \begin{equation}
	    \ord\left(\Phi_{*}\right) = \lcm\left( 2,\ord \sigma \right)\in \Set{2,4,6}.
            \label{eq:Phimper}
        \end{equation}
\end{corollary}

\begin{remark}
    We remark that from \Cref{prop:caseC,rem:periodic} it follows that,
    while the degrees of the maps $\Phi=g\circ\crem_3$, with  $g\in
    \Calc$  of type \textit{(B)}, are two-periodic, the map itself is not
    necessarily two-periodic. Indeed, from \Cref{rem:periodic} it follows
    that the order of the map $\Phi$ can be 2,4, or 6.  The situation
    here is similar to the periodic QRT maps found in \cite{Tsuda2004},
    with the difference that, for our maps, odd orders are not possible.
    \label{rem:tsuda}
\end{remark}

\subsection{Case \textit{(C)}} 
\label{CASEB}

Let $g\in\Calc$ be a projectivity of type \textit{(C)}. Then, a space of initial values for $\Phi=g\circ \crem_3$
is (see \Cref{orbits}) the variety 
\begin{equation}
	B_g=\Bl_{\Calr}\Pj^3.
\end{equation}
The cyclic subgroup of $\Calc$ generated by $g$ acts transitively on the set
$\{\Cale,\Calp,\Calq\}$ (see \Cref{possibilemma}). Without loss of
generality, we can suppose that $g$ acts as follows: 
\begin{equation}
    \begin{tikzcd}[row
	sep=tiny] \Cale \arrow[mapsto]{r}{g} & \Calp \arrow[mapsto]{r}{g} & \Calq
	\arrow[mapsto]{r}{g} & \Cale.  
    \end{tikzcd} 
    \label{eq:caseBsequence}
\end{equation}
Let us fix the following basis of the second cohomology group of $B_{g}$:
\begin{equation}
    \label{eq:H2caseB} 
    H_{\mbox{\tiny\textit{(C)}}}^2 := 
    H^{2}\left(B_{g},\Z \right) = 
    \langle \varepsilon_{g}^*H,E_1,E_2,E_3,E_4, P_1,P_2,P_3,P_4,Q_1,Q_2,Q_3,Q_4\rangle_{\Z},
\end{equation}
where $H$, $E_i,P_i$, for $i=1,\ldots,4$, are as in \Cref{CASEA},
and $Q_{i}$, for $i=1,\ldots,4$, is the cohomology class of the
exceptional divisor contracted by $\varepsilon_{g}$ to $q_i$.

Analogously to \Cref{CASEA}, also in this case, the action of the
standard Cremona transformation on the elements $\varepsilon_{g}^*H,E_1,E_2,E_3,E_4$
agrees with \Cref{CREMONACTION} while, the elements $P_i,Q_i$,
for $i=1,\ldots,4$, are fixed by ${\crem_3}_*$ because they lie over the
fixed points of $\crem_3$.
In terms of singular orbits \cite{BedfordKim2004} this implies that 
there are four closed singular orbits of length three.

So, as in the previous section, the linear map $g_*=(g^{-1})^*$ fixes
$\varepsilon_{g}^*H$ and it permutes the remaining elements of
the basis of the cohomology we have chosen. As a consequence,
the matrix that represents $g_*$ with respect to the basis
\eqref{eq:H2caseB} has a block decomposition similar to the block
decomposition given in \eqref{blocks}. In particular, the cyclic subgroup of
$\Gl(H_{\mbox{\tiny\textit{(C)}}}^2,{\Z})$ generated by $g_*$ induces
a transitive action on the set $\Set{\Cale,\Calp,\Calq}$.

Finally, the action of $\Phi_{*}=g_*\circ {\crem_3}_*$ on
$H_{\mbox{\tiny\textit{(C)}}}^2$ is: \begin{equation}
	\label{eq:actionB} \begin{tikzcd}[row sep=tiny] \varepsilon_{g}^{*}H
		\arrow[mapsto]{r}{\Phi_*} & 3 \varepsilon_{g}^{*}H -
		2\ssum{j=1}{4} P_{j},\\
  E_{i} \arrow[mapsto]{r}{\Phi_*}
		&\varepsilon_{g}^{*} H -
		\ssum{j\neq i}{} P_{\sigma_{1}(j)}&\mbox{for }i=1,\ldots,4,\\ P_i
		\arrow[mapsto]{r}{\Phi_*} &Q_{\sigma_{2}(i)}&\mbox{for
		}i=1,\ldots,4,\\
		Q_{i} \arrow[mapsto]{r}{\Phi_*}& E_{\sigma_{3}(i)}&\mbox{for }i=1,\ldots,4.
	\end{tikzcd}
\end{equation} where $\sigma_1,\sigma_2,\sigma_3$ are elements of
$\SSym_4$ corresponding to the nonzero $4\times4$ blocks of the matrix
representing $g_*$ with respect to the basis in \eqref{eq:H2caseB}.
The same action can be recovered from the four singular orbits using \cite[Eqs. (4.1,4.3)]{BedfordKim2004}.

As a consequence of the above description, we have the main result of this
subsection.

\begin{proposition}
    \label{prop:caseBgrowth} 
    Let $g\in \Calc$ be an element of type \textit{(C)}.
    Then, for all $n\ge 1$, the following formula holds true for the map $\Phi=g\circ \crem_{3}\in\Bir( \Pj^{3})$: 
    \begin{equation}
        \label{eq:caseBphinH} 
        (\Phi_{*})^{n}( \varepsilon_{g}^{*}H) 
            = d_{n} \varepsilon_{g}^{*}H - f_{n}\sum_{j=1}^{4} E_{j}
            - b_{n}\sum_{j=1}^{4} P_{j} - c_{n}\sum_{j=1}^{4} Q_{j},
    \end{equation} where the coefficients solve the following system of
    difference equations: 
    \begin{equation}
        \label{eq:dnen}
        \begin{gathered}
            d_{n}=d_n^\Phi = 3 d_{n-1} - 4 f_{n-1}, 
            \quad
            f_{n} = c_{n-1},
            \\
            b_{n} = 2 d_{n-1} - 3 f_{n-1},
            \quad
            c_{n} = b_{n-1},
        \end{gathered}
    \end{equation} 
    with initial conditions: 
    \begin{equation}
        \label{eq:ini} 
        d_{0}= 1, \,
        f_{0} = 0, \, 
        b_{0} = 0, \, 
        c_{0} = 0. \,
    \end{equation}
    This implies that the map $\Phi$ has positive algebraic entropy given by:
    \begin{equation}
        \label{eq:entropyB} 
        S_{\mbox{\tiny\textit{(C)}}} = 2 \log \varphi,
    \end{equation} 
    where $\varphi$ is the \emph{golden ratio}, i.e. the only positive solution 
    of the algebraic equation $\varphi^{2}=\varphi+1$.
    That is, the map $\Phi$ is \emph{non-integrable} according to the algebraic entropy.
\end{proposition}

\begin{remark}
    We remark that \Cref{prop:caseBgrowth} is coherent with the upper bound on the 
    algebraic entropy presented in \cite[Theorem 4.2]{BedfordKim2004}. Explicitly, 
    for maps of the form $\ell\circ\crem_M$, where $\ell\in\PGl(M+1,\C)$, with at least 
    one singular orbit we have that $S_{\ell\circ\crem_M}< \log M$. Indeed, in the 
    case of \Cref{prop:caseBgrowth} we have $M=3$ and
    \begin{equation}
        S_{\mbox{\tiny\textit{(C)}}} = \log \left(\frac{3+\sqrt{5}}{2}\right)<\log 3.
    \end{equation}
    We will see in \Cref{sec:covariants,sec:invariants} that this imply a certain
    regularity which we do not have for ``generic'' maps.
    \label{rem:upperbound}
\end{remark}

\begin{proof}
    As in the proof of \Cref{prop:caseAgrowth}, we start by
    evaluating $\Phi_*$ on the sums $\sum_{j=1}^4E_j$, $\sum_{j=1}^4P_j$
    and $\sum_{j=1}^4Q_j$. Thanks to \Cref{eq:actionB}, we have:
    \begin{subequations}
        \label{evalB} 
        \begin{align}
            \Phi_{*}\left(\ssum{j=1}{4}E_j\right) &=4\varepsilon_{g}^*H-3\ssum{j=1}{4}P_j,&
            \\ \Phi_{*}\left(\ssum{j=1}{4}P_j\right) &=
            \ssum{j=1}{4}Q_{\sigma_2(j)}=\ssum{j=1}{4}Q_{j},&
            \\ \Phi_{*}\left(\ssum{j=1}{4}Q_j\right)
            &=
            \ssum{j=1}{4}E_{\sigma_2(j)}=\ssum{j=1}{4}E_{j}.
        \end{align}
    \end{subequations} 
    We proceed now by induction on $n\ge 1$. The case $n=1$ is a direct
    computation. We suppose now that \Cref{prop:caseBgrowth} 
    is true for some $n\ge 1$ and we prove it for $n+1$. We have:
    \begin{equation}
            \label{eq:caseBphiHpcomput} \begin{aligned}
                    (\Phi_{*})^{n+1}( \varepsilon_{g}^{*}H )
                    &=\Phi_*\left[\left(\Phi_*\right)^n\left(\varepsilon_{g}^*H\right)\right]=
                    \\ &= \Phi_* \left[d_{n} \varepsilon_{g}^{*}H -
                    f_{n}\ssum{j=1}{4} E_{j} - b_n\ssum{j=1}{4}
                    P_{j} - c_n\ssum{j=1}{4}
                    Q_{j}\right]=\\ &=\left(3d_n-4f_{n}
                    \right)\varepsilon_{g}^*H-c_n\ssum{j=1}{4}E_j-\left( 2d_{n}-3f_{n}
                    \right)\ssum{j=1}{4}P_j-b_n\ssum{j=1}{4}Q_j,
            \end{aligned}
    \end{equation} where the third equality is a consequence of
    \Cref{evalB}.  On the other hand we must have: \begin{equation}
            \label{eq:caseBphinHp} 
            (\Phi_{*})^{n+1}( \varepsilon_{g}^{*}H ) = d_{n+1} \varepsilon_{g}^{*}H -
            f_{n+1}\sum_{j=1}^{4} E_{j} - b_{n+1}\sum_{j=1}^{4} P_{j} -
            c_{n+1}\sum_{j=1}^{4} Q_{j}.
    \end{equation} So, the condition is satisfied by equating
    with the right hand side of \eqref{eq:caseBphiHpcomput} and
    \eqref{eq:caseBphinHp} and invoking the linear independence of
    the generators of $H_2(B_g,\Z)$.  This implies that $d_n,f_n,b_n$ and
    $c_n$ satisfy the system \eqref{eq:dnen} with initial conditions
    \eqref{eq:ini}.
    
    In order to compute the algebraic entropy from \Cref{def:algent}, we
    need to evaluate the asymptotic behaviour of $d_{n}^{\Phi}$ in
    equation \eqref{eq:dnen}. Since the system \eqref{eq:dnen} is linear
    we use the technique explained in \cite[Chap. 3]{Elaydi2005}.
    Writing the system as:
    \begin{equation}
        \begin{pmatrix}
            d_n
            \\
            f_n
            \\
            b_n
            \\
            c_n
        \end{pmatrix}
        =
        M_g
        \begin{pmatrix}
            d_{n-1}
            \\
            f_{n-1}
            \\
            b_{n-1}
            \\
            c_{n-1}
        \end{pmatrix},
        \quad \mbox{ where }
        M_g =
        \begin{pmatrix}
            3&-4&0&0
            \\ 
            0&0&0&1
            \\
            2&-3&0&0
            \\
            0&0&1&0
        \end{pmatrix},
    \end{equation}
    then the solution is:
    \begin{equation}
        \begin{pmatrix}
            d_n
            \\
            f_n
            \\
            b_n
            \\
            c_n
        \end{pmatrix}
        =
        M_g^n
        \begin{pmatrix}
            d_{0}
            \\
            f_{0}
            \\
            b_{0}
            \\
            c_{0}
        \end{pmatrix}
        =
        M_g^n
        \begin{pmatrix}
            1
            \\
            0
            \\
            0
            \\
            0
        \end{pmatrix}.
    \end{equation}
    Computing $M_g^n$, e.g. using Putzer algorithm \cite[Sect. 3.1.1]{Elaydi2005},
    we obtain the following solution for all $n\in\N$:
    \begin{equation}
        \begin{pmatrix}
            d_n
            \\
            f_n
            \\
            b_n
            \\
            c_n
        \end{pmatrix}
        =
        \begin{pmatrix}
            \displaystyle
            \frac{8}{5}\left(\varphi^{2n}+\varphi^{-2n}\right)
            -\frac{1}{5}\left( -1 \right) ^{n}-2
            \\ 
            \noalign{\medskip}
            \displaystyle
            \frac{2}{5}\left(\varphi^{2n-2}+\varphi^{-2n+2}\right)
            -\frac{1}{5}\left( -1\right) ^{n}-1
            \\ 
            \noalign{\medskip}
            \displaystyle
            \frac{2}{5}\left(\varphi^{2n+2}+\varphi^{-2n-2}\right)
            -\frac{1}{5}\left( -1\right) ^{n}-1
            \\
            \noalign{\medskip}
            \displaystyle
            \frac{2}{5}\left(\varphi^{2n}+\varphi^{-2n}\right)
            -\frac{1}{5}\left( -1 \right) ^{n}-1
        \end{pmatrix},
    \end{equation}
    where $\varphi$ is the golden ratio. 
    Since $d_n=d_n^\Phi$, we have:
    \begin{equation}
        d_n^\Phi \sim \varphi^{2n}, \quad
        n\to\infty,
    \end{equation}
    and, from \eqref{eq:algentdef},
    formula \eqref{eq:entropyB} follows.
\end{proof}

\section{Covariant linear systems}
\label{sec:covariants}

In this section we will introduce a notion of covariant linear system with respect a fixed divisor. We will see in next section that these objects are key tools in the computation of the invariants of the maps of the form $\Phi=g\circ \crem_3$, for $g\in \Calc$. Moreover, we will compute covariant linear systems of quadrics for $g$ of type \textit{(A)} and \textit{(B)} (\Cref{prop:quadricsA,covcasec,covac2}) and covariant linear systems of quartics for $g$ of type \textit{(C)} (\Cref{prop:quarticsC}). As in  \Cref{sec:entropy}, the results in this section works also for the birational maps of the form $\crem_3\circ g$ (see \Cref{inverso}).

In what follows we shall need the following definition.
\begin{definition}
    \label{def:convariance}

   Let $X$ be a smooth projective variety and let $\Phi\in\Bir(X)$
   be a birational map. Let also $D\in\Div(X)$ be a divisor. A linear
   system $\Sigma\subset |\mathcal L|$, for some line bundle $\mathcal
   L$, is $D$-covariant if, for any $E\in \Sigma$, we have $\Phi_*E\in
   \Sigma+D$, where
   \begin{equation}
        \Sigma +D=\Set{A+D|A\in \Sigma}.
        \label{eq:Sigma}
   \end{equation}
   In this context, if $\Phi_*E=E+D$, we will say that $E$  is $D$-invariant.
 \end{definition}

\begin{remark}
    \label{inverso} We remark that any result about the existence of $D$-covariant linear systems for maps of the form $g\circ \crem_3$, for $g\in\Calc,$ also implies existence of $\widetilde D$-covariant linear systems for maps of the form $\crem_3\circ g$ for some divisor $\widetilde D$. Indeed, $g$ and $g^{-1}$ are of the same type as per \Cref{proprietagruppi} and the equality
    \begin{equation}
        (g\circ\crem_3)_*S=S+D  
    \end{equation}
    implies
    \begin{equation}
        (\crem_3\circ g^{-1})_*S=S-(\crem_3\circ g^{-1})_*D.
    \end{equation}
 \end{remark}

\begin{remark}\label{CONDIIZIONEQUIVALENTE}
Let $\Phi=g\circ\crem_3\in\Bir(\Pj^3)$, for some $g\in\Calc$, be a birational map, and let $\Sigma$ be a $D$-covariant linear system for some $D\in\Div(\Pj^3)$. Then, since $g\circ\crem_3$ only contracts the coordinate hyperplanes $H_i=\Set{x_i=0}$, for $i=1,\ldots,4$, the divisor $D$ has the form
\begin{equation}
    \label{sommaiperpiani}
    D=\underset{i=1}{\overset{4}{\sum}}n_iH_i,
\end{equation}
for some $n_i\in\Z$, for $i=1,\ldots,4$. Notice also that, if $B $ is a space of initial values for $\Phi$,  $\Sigma$ is $D$-covariant  if and only if the linear system of the strict transform of a general member of $\Sigma$, is covariant on $B$, i.e. called $\widetilde{E}$ the strict transform of a general member $E\in \Sigma $, we have ${\Phi}_*\widetilde{E}\in|\widetilde{E}|$.
\end{remark}

We now start a case-by-case analysis of the $D$-covariant linear systems of quadrics for some divisor $D$, i.e. projective subspaces, $\Sigma\subset |\Calo_{\Pj^3}(2)|$ with the property that, $\Phi_*\Sigma=\Sigma +D\subset |\Calo_{\Pj^3}(2+\deg (D))|$. This will be enough to compute the invariants for maps of the form $g\circ\crem_3$ for $g\in\Calc$ of type \textit{(A)} and \textit{(B)} as we will show in \Cref{sec:invariants}. For type \textit{(C)}, we will need to consider $D$-covariant linear systems of quartics.

\begin{notation}
    In what follows, we will denote by $H_i$ the the divisor associated
    to the hyperplane $\Set{x_i=0}\subset\Pj^3$. Moreover, we will denote
    by $\overline{H}$ the divisor defined as the sum of the coordinate
    hyperplane divisors, i.e.
    \begin{equation}
        \overline{H}=\ssum{i=1}{4}H_i.
        \label{eq:Hi}
    \end{equation}
    Sometimes, with abuse of notation, we will denote by $H_i$ also the 
    $i$-th coordinate hyperplane.
\end{notation}
\subsection{Case \textit{(A)}} 
\label{sss:caseAinvariants}

As in \Cref{CASEA}, without loss of generality, we keep the assumption
that $\Phi =g\circ \crem_3$, where $g\in\Calc$ is of type \textit{(A)} and that it swaps $\Cale$ and $\Calp$.
Then, we have the following.

\begin{proposition}
    \label{prop:quadricsA}
    Let $g\in\Calc$ be an element of type \textit{(A)} and let $\Phi\in\Bir(\Pj^3)$ be the projective map defined as $\Phi =g\circ\crem_3$. Then, the net of quadrics $\Sigma_{\Calp}$ in \Cref{def:nets} is $\overline{H}$-covariant. Moreover, if $\Sigma $ is a positive-dimensional $D$-covariant linear system of quadrics, for some $D\in \Div(\Pj^3)$, we have $D=\overline{H}$ and $\Sigma\subset \Sigma_{\Calp}$.
\end{proposition}
\begin{proof} Let $S\in\Sigma_{ \Calp}$ be a general element, and let $\widetilde{S} $ be its strict
	transform via $\varepsilon_{g}$, where $\varepsilon_g$ is the same as in  \Cref{CASEA}. We will prove that $|\widetilde{S}|$ is covariant, i.e. $\Phi_*\widetilde{S}\in|\widetilde{S}|$ (see \Cref{CONDIIZIONEQUIVALENTE}). We have:
	\begin{equation}\label{strictcovA}
	    \Phi_*\widetilde{S}\sim\Phi_*\left(2\varepsilon_{g}^*H-\ssum{j=1}{4}E_j-\ssum{j=1}{4}P_j\right)\sim 2\varepsilon_{g}^*H-\ssum{j=1}{4}E_j-\ssum{j=1}{4}P_j\sim \widetilde{S},
	\end{equation}
	where the second equality is a consequence of formulas
	\eqref{evalA}. Therefore, the net $\Sigma_{ \Calp}$ is $\overline{H}$-covariant.

Now, the strict transform $\widetilde{S}$ of a quadric $S\in\Pj^3$ is linear equivalent to the following divisor
\begin{equation}
    \widetilde{S}\sim 2\varepsilon_{g}^*H-\ssum{j=1}{4}n_jE_j-\ssum{j=1}{4}m_jP_j,
\end{equation}
for some non-negative integers $n_j,m_j$, for $j=1,\ldots,4$. If we impose the covariance, we get the conditions
\begin{equation}
\label{eq:conditionequiv}
    \begin{cases}
    \ssum{j=1}{4}n_j=4,\\
    n_i=m_{\sigma_1(i)} & \mbox{ for} i=1,\ldots,4. 
\end{cases}
\end{equation}
where $\sigma_1 $ is taken from \eqref{eq:actionA}. Now, up to relabeling the coordinates, we have three possibilities, namely
\begin{subequations}
    \begin{gather}
    n_1=m_1=1,\ n_2=m_2=1,\ n_3=m_3=1,\ n_4=m_4=1, 
    \\
    n_1=m_1=2,\ n_2=m_2=1,\ n_3=m_3=1,\ n_4=m_4=0,
    \\
    n_1=m_1=2,\ n_2=m_2=2,\ n_3=m_3=0,\ n_4=m_4=0. 
    \end{gather}
\end{subequations}
Notice that, the first choice works for any permutation of four elements
$\sigma_1$. Moreover, the second and the third choice of coefficients  cut
zero-dimensional linear systems.  As a consequence, the  only possible
$D$-covariant and positive-dimensional linear system is obtained for
$n_i=m_{i}=1$, for $i=1,\ldots,4$.
\end{proof}

\begin{remark}
    If $g$ swaps $\Cale$ and $\Calq$ the statement of \Cref{prop:quadricsA} is true with $\Sigma_{\Calq}$ instead of $\Sigma_{\Calp}$ (see \Cref{def:nets}).
\end{remark}

\begin{corollary}
    With the same notation and hypotheses in \Cref{prop:quadricsA},
    the anticanonical system of $B_g$ is covariant. Moreover, we have
    $|-K_{B_g}|=|2\widetilde{E}|$  where $\widetilde{E}$ is the strict
    transform of a general member of $\Sigma_{\Calp}$.
\end{corollary}

\begin{proof}
    
 The statement follows by applying the first equivalence in \eqref{strictcovA} to formula \eqref{anticanonical}.
\end{proof}

\begin{remark}\label{cosarestadafar}
    Notice that, every pair of points $q_i,q_j$ in $\Calq$ cuts a unique
    quadric $S_{ij}$ in $\Sigma_{\Calp}$. Explicitly, with the
    notation in \Cref{thecubes}, the unique quadric of $\Sigma_{\Calp}$
    passing trough $q_i$ and $q_j$ is reducible and it is
    \begin{equation}
        S_{ij}=H_{ij}^\bullet + H_{hk}^\bullet,
    \end{equation}
    where
    \begin{equation}
        \Set{i,j,h,k}=\Set{1,2,3,4}
        \quad 
        \text{and}
        \quad 
        \bullet=
        \begin{cases}
            +&\text{ if } 4 \in \Set{i,j},\\
            -& \text{ otherwise.}
        \end{cases}
    \end{equation}
    The strict transform of this quadric on some space of initial
    values, is fixed by the action of $\Phi_*=(g\circ \crem_3)_*$,
    i.e. $S_{ij}$ is $\overline{H}$-invariant if and only if
    $\Set{q_i,q_j}\subset \Calr$ is invariant under $g$. Therefore, the
    existence of $\overline{H}$-invariant quadrics depends on the action
    of $g$ on the third tetrahedron, namely the tetrahedron with $\Calq$
    as set of vertices.
\end{remark}

\subsection{Case \textit{(B)}}
\label{sss:caseCinvariants}

We have the following result.

\begin{proposition}\label{covcasec}
 Let $g\in\Calc$ be an element of type \textit{(B)} and let $\Phi\in\Bir(\Pj^3)$ be the birational map defined as $\Phi=g\circ\crem_3$. Then, the  5-dimensional linear system $\Sigma_{\mbox{\tiny \textit{(B)}}}\subset |\Calo_{\Pj^3}(2)|$ consisting of the quadrics containing $\Cale$ is $\overline{H}$-covariant.
\end{proposition}

\begin{proof}
    First notice that $ \Sigma_{\mbox{\tiny \textit{(B)}}}$ is
    5-dimensional because the points in $\Cale$ are in general
    position. Let $\widetilde{S}$ be the strict transform of a general
    quadric $S\subset \Pj^3$. As in the proof of \Cref{prop:quadricsA},
    we will prove that $|\widetilde{S}|$ is covariant. We have,
    \begin{equation}
        \widetilde{S}\sim 2\varepsilon_{g}^*H-\ssum{i=1}{4}n_i E_i
    \end{equation}
    for some $n_i\ge 0$, for $i=1,\ldots,4$. If we now impose the
    covariance and we apply formulas \eqref{eq:actionC}, we find the conditions
    \begin{equation}
        \begin{cases}
        \ssum{i=1}{4}n_i=4\\
        n_i=n_{\sigma(i)} &\mbox{ for }i=1,\ldots,4,
        \end{cases}
    \end{equation}
    where $\sigma\in\SSym_4$ is the element corresponding to $g$
    (see \Cref{possibilemma}). If we impose the above conditions for any $g$ of type \textit{(B)}, and hence for any
    $\sigma\in\SSym_4$, we get \[n_i=1\mbox{ for }i=1,\ldots,4,\] which
    implies the thesis.
\end{proof}

\begin{corollary}
   With the same notation and hypotheses in \Cref{covcasec}, the anticanonical system of $B_g$ is covariant. Moreover, we have $|-K_{B_g}|=|2\widetilde{E}|$  where $\widetilde{E}$ is the strict transform of a general member of $\Sigma_{\mbox{\tiny \textit{(B)}}}$.
\end{corollary}

The linear system $\Sigma_{\mbox{\tiny \textit{(B)}}}$ in \Cref{covcasec} is $\overline{H}$-covariant for all the $\Phi=g\circ\crem_3$ with $g$ of type \textit{(B)}, but, in some instances, there are positive dimensional linear systems which are $D$-covariant for some other divisor $D$. However, we will show in \Cref{bastasigmab} that $\Sigma_{\mbox{\tiny \textit{(B)}}}$ suffices for the construction of the invariants. Rigorously, we have the following.

\begin{proposition}\label{covac2}
    Let us fix the same notation and hypotheses in \Cref{covcasec}. We have the following possibilities.
    \begin{itemize}
        \item Suppose $g$ fixes exactly two points of $\Cale$, say $e_1$ and $e_2$. Then, the webs $\Sigma_{1}$ and $\Sigma_{2}$ of quadrics passing trough $e_3,e_4$ and having respectively a node at $e_1$ and $e_2$ are respectively $(2H_1+H_3+H_4)$-covariant and $(2H_2+H_3+H_4)$-covariant.
        \item Suppose $g$ fixes $\Cale$ pointwise. Then, for any divisor of the form $D=2H_i+H_j+H_k$, where $i,j$ and $k$ are different, there is a web of $D$-covariant quadrics and, for any divisor of the form $F=2H_i+2H_j$, where $i\not=j$, there is a net of $F$-covariant quadrics. 
     \end{itemize}
\end{proposition}

\begin{remark}
    \label{rem:CREMONACASEC}
    Notice that the identity $\Id\in\Calc$ belongs to case
    \textit{(B)}. Therefore, the standard Cremona transformation $\crem_3=\Id \circ \crem_3$
    has the linear system $\Sigma_{\mbox{\tiny \textit{(B)}}}$ in
    \Cref{covcasec} as a $\overline{H}$-covariant linear system. Moreover,
    the set of $\overline{H}$-invariant divisors for ${\crem_3}$
    consists of the disjoint union, inside $\Sigma_{\mbox{\tiny
    \textit{(B)}}}$, of the  two nets of quadrics $\Sigma_{\Calp}$
    and $\Sigma_{\Calq}$  in \Cref{def:nets}.
\end{remark}

\subsection{Case \textit{(C)}}
\label{sss:caseBinvariants}

We start this last subsection by proving that in the case
\emph{(C)} we do not have $D$-covariant linear systems of quadric anymore.

\begin{proposition}
    \label{prop:noquadricsB}
    For any divisor $D\in\Div (\Pj^3)$, there is no $D$-covariant linear 
    system of quadrics for the map $\Phi$.
\end{proposition}

\begin{proof}
    The strict transform $\widetilde{S}$ of a quadric $S\in\Pj^3$
    is linear equivalent to the following divisor
    \begin{equation}
        \widetilde{S}\sim 2\varepsilon_{g}^*H-\ssum{j=1}{4}n_jE_j-\ssum{j=1}{4}m_jP_j-\ssum{j=1}{4}k_jQ_j
    \end{equation}
    for some non-negative integers $n_i,m_{i},k_{i}$ for $i=1,\ldots,4$. 
    If we impose the covariance, we get the conditions
    \begin{equation}
    \label{eq:conditionequivB}
        \begin{cases}
        \ssum{j=1}{4}n_j=4\\
        n_i=m_{\sigma_1(i)} & i=1,\ldots,4,\\
        m_{i}=k_{\sigma_2(i)} & i=1,\ldots,4. 
    \end{cases}
    \end{equation}
    where $\sigma_1  $ and $\sigma_2$ are taken from \eqref{eq:actionB}. A
    quadric $S$ satisfying \eqref{eq:conditionequivB} must pass trough
    all the points in $\Calr$, and a dimension count shows that a similar
    quadric does not exist.
\end{proof}

Although, in this case, there are no $D$-covariant linear systems of
quadrics, for $D\in\Div(\Pj^3)$, there is a $2\overline{H}$-covariant
linear system of quartics. The proof of the following proposition is
similar to the proof of \Cref{prop:quadricsA}.

\begin{proposition}
    \label{prop:quarticsC}
    Let $g\in\Calc$ be an element of type \textit{(C)} and let
    $\Phi\in\Bir(\Pj^3)$ be the projective map defined as $\Phi
    =g\circ\crem_3$.  Then, the pencil of quartics $\Xi_{\Calr}$ nodal at
    all the points of $\Calr$ is $2\overline{H}$-invariant. Moreover,
    if $\Sigma $ is a positive-dimensional $D$-covariant linear system
    of quartics for some $D\in \Div(\Pj^3)$, we have $D=2\overline{H}$
    and $\Sigma= \Xi_{\Calr}$.
\end{proposition}

\begin{corollary}
    With the same notation and hypotheses in \Cref{prop:quarticsC},
    the anticanonical system of $B_g$ is covariant. Moreover, we have
    $|-K_{B_g}|=|\widetilde{E}|$  where $\widetilde{E}$ is the strict
    transform of a general member of $\Xi_\Calr$.
\end{corollary}

\begin{remark}
    \label{rem:desmic}
    We remark that the quartic pencil
    \begin{equation}
        \label{eq:desm}
        \begin{tikzcd}[column sep=huge]
            \Pj^3 \arrow[ dashed,"\Xi_{\Calr}"]{r} & \Pj^1
        \end{tikzcd}
    \end{equation}
    is exactly the pencil mentioned in \Cref{desmicclassic}. Its base
    locus consists of 16 lines, namely the 16 lines of the $(12_4\ 16_3)$ configuration (see \Cref{thecubes}). The resolution of its indeterminacies
    induces the crepant resolution\footnote{Recall that the nodes
    are special instances of rational double points.} of each fibre
    (see \cite{GraffeoConstellation2022}) but three, namely the desmic
    surfaces. Therefore, we get a one-dimensional family of surfaces
    \begin{equation}
        Y \rightarrow \Pj^1
    \end{equation}
    whose generic member is a $K3$ surface with 16 disjoint rational
    curves. This is the highest possible number of disjoint rational
    curves on a $K3$ surface and it is known to be achieved by Kummer
    surfaces of an abelian surface of the form $A=E\times E$ for $E$
    elliptic curve (see \cite[Theorem B.5.6]{HuntGeomQuot1996}).
\end{remark}

\section{Construction of the invariants}
\label{sec:invariants}

In this section we explain how to determine the existence of invariants
in the three cases listed in \Cref{possibilemma} and how to compute all
of them. We recall that, by invariant of a map $\Phi\in\Bir(\Pj^M)$,
we mean a rational function, i.e. a degree zero element $R=P/Q\in
\C(x_1,\ldots,x_{M+1})$ for $P,Q\in\C[x_1,\ldots,x_{M+1}]$ homogeneous,
such that:
\begin{equation}
    \label{eq:invcond}
    \Phi_*(R)(x_1,\ldots,x_{M+1})=R(x_1,\ldots,x_{M+1}).
\end{equation}

As for the results in \Cref{sec:entropy,sec:covariants}, the results
in this section hold true for the maps of the form $\crem_3\circ g$ as
well. Indeed, $g$ and $g^{-1}$ are of the same type and, after composing
the equality $R\circ \Phi \equiv R$ with $\Phi ^{-1}$ we find that,
if $R$ is invariant for $g\circ\crem_3 $ then, it is also invariant for
$\crem_3\circ g^{-1}$.

We recall the following known fact on the construction of invariants
of birational maps adapted to our setting. 

\begin{lemma}
    \label{rem:ellipticandnoint} Let $\Phi\in\Bir(\Pj^3)$ be a projective
    map and let $D\in\Div(\Pj^3)$ be a divisor. Let us consider
    a $D$-covariant linear system $\Sigma\subset |\mathcal{L}|= \Pj
    H^0(\Pj^3,\mathcal L)$, for some line bundle $\mathcal L$.  Suppose
    that there exists a pencil $\Xi\subset \Sigma$ of $D$-invariant
    divisors. In particular $\Xi= \Pj V_\Xi$ for some vector subspace
    $V_\Xi\subset H^0(\Pj^3,\Call)$. Then, one can construct invariants
    of the map $\Phi$ by considering the meromorphic functions of the
    form $f=s_1/s_2$ for any given choice of  $s_1,s_2\in V_\Xi$.
\end{lemma}

For a proof of \Cref{rem:ellipticandnoint} we refer to \cite{FalquiViallet1993} (see also \cite{CarsteaTakenawa2019JPhysA, Gubbiotti_Levi70} and reference therein).

\begin{remark}
    \label{rem:3points}
    Let $\Sigma$ be a $D$-covariant linear system and let $F_1,F_2\in \Sigma$ be two $D$-invariant elements. Then, the elements of the pencil generated by $F_1$ and $F_2$ are not necessarily $D$-invariant. This is a consequence of the fact that a projectivity of $\Pj^1$ is uniquely determined by its value on three points (see \cite[Section 1.3]{PARDINI}).
\end{remark}

The results in \Cref{sec:entropy} lead us to expect the following behaviour of the maps of the form $\Phi=g\circ \crem_3$:
\begin{description}
    \item[Case \emph{(A)}] the map admits two invariants\footnote{We recall 
    that, for three-dimensional systems, the notion of Liouville--Poisson integrability and
    algebraic integrability agree (see the discussion in \cite{GJTV_class,GLT_coalgebra} for more details).},
    \item[Case \emph{(B)}] the map admits more than two invariants,
    \item[Case \emph{(C)}] the map admits at most one invariant.
\end{description}

In what follows we will prove that the actual number of invariants agrees
with the expected one by showing how to construct them explicitly.

\begin{remark}
    We remark that the invariants will be of degree up to twelve.
    This is because a $g\in\Calc$ is is not guaranteed to preseve
    the fibres of the associated linear system, but rather they are permuted
    periodically. As a consequence, they will be invariants for an appropriate
    power of the map itself. This behaviour is linked to the notion of $k$-invariants introduced in 
    \cite{Haggaretal1996}, i.e. invaiants for the $k$-th iterate $\Phi^k$ of $\Phi$. In particular, it is possible to compute the invariants of $\Phi$ starting from its $k$-invariants (see \cite{HietarintaBook,RJ15}). For a geometric discussion on the origin of
    this kind of maps in $\Bir(\Pj^2)$ we refer to \cite{CarsteaTakenawa2012}.
    We mention that a classification and constructions of these plane maps 
    were presented in \cite{JoshiKassotakis2019,KassotakisJoshi2010,RJ15}.
    We finally recall that a three-dimensional example, obtained as \emph{deflation} 
    (see \cite{JoshiViallet2017}) of a four-dimensional map admitting an 
    anti-invariant was presented in \cite{GJTV_sanya}. The appearance
    of more general fibre exchanges in dimension three is,
    up to our knowledge, new.
\end{remark}

\subsection{Case \textit{(A)}} \label{subsec:casea}

Let $g\in\Calc$ be an element of type \textit{(A)}. As usual, without loss
of generality (see \Cref{CASEA}), we can suppose that $g$ swaps $\Cale $
and $ \Calp$.

As explained in \Cref{rem:ellipticandnoint}, in order to build
invariants for the map $\Phi=g\circ \crem_3$, we need to find pencils
of $D$-invariant divisors inside some $D$-covariant linear system. We
start looking for quadrics.

As a consequence of \Cref{prop:quadricsA}, we have at most
$\overline{H}$-invariant pencils quadrics in the $\overline{H}$-covariant net
$\Sigma_{\Calp}$.

Recall that the elements in $\Sigma_{\Calp}$ are
$\overline{H}$-invariant with respect to the map $\crem_3$
(see \Cref{rem:CREMONACASEC}). Therefore, since $\Phi_*= g_*\circ
{\crem_3}_*$, $\overline{H}$-invariant divisors of $\Phi$ correspond to
divisors which are invariant under $g$. As a consequence, the existence
of $\overline{H}$-invariant elements in $\Sigma_{\Calp}$ is ruled by
the action of $g$ on the collection of pairs of elements in $\Calq $
as explained in \Cref{cosarestadafar}.

Each pair of points $\Set{q_i,q_j}\subset \Calq $, for $1\le i<j \le 4$,
corresponds to one of the quadrics $S_{ij}$ in \Cref{cosarestadafar}. Now,
the action of $g$ on
\begin{equation}
S_\Calq=\Set{S_{ij}|1\le i<j \le 4}
\end{equation}
is uniquely determined by the action of  $g$ on $\Calq$. 

\Cref{tab:excaseA}  compares the possible cycle decompositions of the
action  of $g$ on $\Calq$, on $ S_{\Calq}$ and on $\Cale\cup\Calp$. In
particular, for each possible cycle decomposition, an element $g$
that realizes it and its cycle decomposition on $\Calp\cup\Calq$ are
given. Length one cycles are omitted.

\begin{table}[ht]
    \centering
      \begin{tabular}{c|ccccc}
    \hline
&\# & Action on $\Calq$ & \hspace{-10pt} Action on ${S}_\Calq$& \hspace{-5pt}$g$& Action on ${C}_{\Cale\Calp}$ \\
    \hline
 \hspace{-10pt} \mbox{\textit{(i)}} &   4 &$\Id_\Calq$ & $\Id_{S_\Calq}$ &\hspace{-5pt}{\footnotesize $
        \begin{bmatrix}
            1&\hspace{-7pt}1&\hspace{-7pt}1&\hspace{-7pt}-1
            \\ 
            1&\hspace{-7pt}1&\hspace{-7pt}-1&\hspace{-7pt}1
            \\ 
            1&\hspace{-7pt}-1&\hspace{-7pt}1&\hspace{-7pt}1
            \\ 
            -1&\hspace{-7pt}1&\hspace{-7pt}1&\hspace{-7pt}1
    \end{bmatrix}$} &\hspace{-10pt} {\footnotesize $(e_1\ p_4)(e_2\ p_3)(e_3\ p_2)(e_4\ p_1)$} \\\hline
   \hspace{-10pt} \mbox{\textit{(ii)}} &   24 & {\footnotesize $(q_1\ q_2)$} &{\footnotesize $(S_{13}\ S_{23})(S_{14}\ S_{24})$} &\hspace{-5pt}{\footnotesize $\begin{matrix}
          \begin{bmatrix}
            1&\hspace{-7pt}1&\hspace{-7pt}-1&\hspace{-7pt}1
            \\ 
            1&\hspace{-7pt}1&\hspace{-7pt}1&\hspace{-7pt}-1
            \\ 
            1&\hspace{-7pt}-1&\hspace{-7pt}1&\hspace{-7pt}1
            \\ 
            -1&\hspace{-7pt}1&\hspace{-7pt}1&\hspace{-7pt}1
    \end{bmatrix}\\
    \begin{bmatrix}
            -1&\hspace{-7pt}1&\hspace{-7pt}1&\hspace{-7pt}1
            \\ 
            1&\hspace{-7pt}-1&\hspace{-7pt}1&\hspace{-7pt}1
            \\ 
            1&\hspace{-7pt}1&\hspace{-7pt}1&\hspace{-7pt}-1
            \\ 
            1&\hspace{-7pt}1&\hspace{-7pt}-1&\hspace{-7pt}1
    \end{bmatrix}
      \end{matrix}$}& \hspace{-10pt} {\footnotesize $\begin{matrix}(e_1\ p_4\ e_2\ p_3)(e_3\ p_1\ e_4\ p_2)  \\
     \  \\ \  \\
      (e_1\ p_1)(e_2\ p_2)(e_3\ p_4)(e_4\ p_3)
      \end{matrix}$  } \\ \hline
      
   \hspace{-10pt}    \mbox{\textit{(iii)}} & 12 & {\footnotesize $(q_1\ q_2) (q_3\ q_4)$} &{\footnotesize $(S_{13}\ S_{24})(S_{14}\ S_{23})$}&\hspace{-5pt}{\footnotesize $\begin{bmatrix}
            -1&\hspace{-7pt}1&\hspace{-7pt}-1&\hspace{-7pt}-1
            \\ 
            1&\hspace{-7pt}-1&\hspace{-7pt}-1&\hspace{-7pt}-1
            \\ 
            1&\hspace{-7pt}1&\hspace{-7pt}1&\hspace{-7pt}-1
            \\ 
            1&\hspace{-7pt}1&\hspace{-7pt}-1&\hspace{-7pt}1
    \end{bmatrix}$}&\hspace{-10pt} {\footnotesize $(e_1\ p_1\ e_2\ p_2)(e_3\ p_3\ e_4\ p_4)$} \\ \hline
 \hspace{-10pt}   \mbox{\textit{(iv)}} &   32 & {\footnotesize $(q_1\ q_2\  q_3)$} &{\footnotesize $(S_{12}\ S_{23}\  S_{13})(S_{14}\ S_{24}\  S_{34})$} &\hspace{-5pt}{\footnotesize $
   \begin{bmatrix}
            -1&\hspace{-7pt}1&\hspace{-7pt}1&\hspace{-7pt}1
            \\ 
            1&\hspace{-7pt}1&\hspace{-7pt}-1&\hspace{-7pt}1
            \\ 
            1&\hspace{-7pt}1&\hspace{-7pt}1&\hspace{-7pt}-1
            \\ 
            1&\hspace{-7pt}-1&\hspace{-7pt}1&\hspace{-7pt}1
    \end{bmatrix} $}&\hspace{-10pt} {\footnotesize $ (e_1\ p_1)(e_2\ p_4\ e_3\ p_2\ e_4\ p_3)$}
      \\ \hline
   \hspace{-10pt}  \mbox{\textit{(v)}} & 24 & {\footnotesize $(q_1\ q_2\  q_3\ q_4)$} & {\footnotesize $(S_{12}\ S_{23}\  S_{34}\ S_{14}) (S_{13}\ S_{24})$ }& \hspace{-5pt}{\footnotesize $\begin{bmatrix}
            -1&\hspace{-7pt}1&\hspace{-7pt}-1&\hspace{-7pt}-1
            \\ 
            1&\hspace{-7pt}1&\hspace{-7pt}-1&\hspace{-7pt}1
            \\ 
            1&\hspace{-7pt}1&\hspace{-7pt}1&\hspace{-7pt}-1
            \\ 
            1&\hspace{-7pt}-1&\hspace{-7pt}-1&\hspace{-7pt}-1
    \end{bmatrix}$}& \hspace{-10pt} {\footnotesize $(e_1\ p_1\  e_4\ p_2)(e_2\ p_4\ e_3\ p_3)$}  \\
      \hline 
    \end{tabular}
    \caption{Comparison of the possible cycle decompositions of the action  of $g$ on $\Calq$, on $ S_{\Calq}$ and on $C_{\Cale\Calp}$}
    \label{tab:excaseA}
\end{table}

\begin{remark}
    \label{fixpoints}

    Let $s_{ij}\in\C[x_1,\ldots,x_4]$ be a polynomial defining the quadric
    $S_{ij}\in S_\Calq$, for $1\le i<j\le 4$. Let also $\sigma=(S_{i_1j_1}\
    \cdots\ S_{i_kj_k})$, for some $k\in\Set{1,2,3,4}$, be a cycle appearing in the cycle decomposition
    of the action of $g $ on $S_\Calq$.

    One can try to find coefficients $ \alpha_{i_sj_s} \in \C $, for
    $s\in 1,\ldots,k$, such that the quadric
    \begin{equation}
        S_\sigma=\Set{\alpha_{i_1j_1 }s_{i_1j_1 }+\alpha_{i_2j_2 } s_{i_2j_2 }
        +\cdots +\alpha_{i_kj_k } s_{i_kj_k }=0}
    \end{equation}
   is $\overline{H}$-invariant. Notice that, in general, $S_\sigma\cap
   \Calq=\emptyset$.

   Moreover, if we chose appropriately a representative for
   $g\in\Calc\subset \Pj\Gl(4,\C) $, in the actual computation,  we can
   restrict to $\alpha_{i_aj_a}\in \Set{\pm 1}$ for all $a=1,\ldots,k$.

   This observation allows one to construct a finite number (possibly
   zero) of $\overline{H}$-invariant quadrics. At most one quadric
   for each cycle in the cycle decomposition of the action of $g$
   on $S_\Calq$. Now, we want to understand if they generate pointwise
   $\overline{H}$-invariant pencils, i.e. pencils whose points correspond
   to $\overline{H}$-invariant quadrics. This will help us in constructing
   the invariants of $\Phi$ as described in \Cref{rem:ellipticandnoint}.
\end{remark}

\begin{proposition}\label{fixlines}
   Let $S_{\sigma_i}=\Set{s_i=0}\in \Sigma_\Calp$, for $i=1,2$,
   be two quadrics obtained  as described in \Cref{fixpoints}
   and let $\Xi\subset \Sigma_{\Calp }$ be the pencil generated by
   $S_{\sigma_1}$ and $S_{\sigma_2}$. Then, the pencil $\Xi$ is pointwise
   $\overline{H}$-invariant if and only if the quadric
   \begin{equation}
        \Set{s_1+s_2=0}
   \end{equation}
   is $\overline{H}$-invariant.
\end{proposition}
\begin{proof}
    We have
    \begin{equation}
        \Xi=\Set{\Set{\mu s_1+\lambda{s_2}=0}\subset\Pj^3|[\mu:\lambda]\in\Pj^1}.
    \end{equation}
    Since $S_{\sigma_1}$ and $S_{\sigma_2}$ are $\overline{H}$-invariant, the pencil $\Xi$ is a $\overline{H}$-covariant subspace of $\Sigma_{\Calp}$. Therefore, we obtain an automorphism
    \begin{equation}
     \begin{tikzcd}[row sep=tiny]
        \Xi \arrow{r} & \Xi\\
        S\arrow[mapsto]{r} &\Phi_* S- \overline{H} .
    \end{tikzcd}
    \end{equation}
    Now, the thesis follows from the fact that there is only one automorphism of $\Pj^1$ that fixes 3 points, city (see \Cref{rem:3points}).
\end{proof}

\begin{remark}
    \label{rem:notenough}
    Unfortunately, \Cref{fixlines} together with \Cref{fixpoints} are not enough to compute all the invariants as predicted at the beginning
    of this section. Indeed, two issues can occur:
    \begin{itemize}
        \item there are less than two $\overline{H}$-covariant pencils of quadrics,
        \item there are at least two $\overline{H}$-covariant pencils, 
        but the invariants produced are not functionally independent.
    \end{itemize}
    However, implementing the argument given in \Cref{fixpoints} with combinatoric techniques, it is possible to construct the invariants
of the maps appearing in \Cref{tab:excaseA} as a nonlinear combinations of $\overline{H}$-invariant quadrics. Their explicit
form is is presented in \Cref{tab:ivariantcaseA}. 
\end{remark}

We now comment some facts about the invariants listed in \Cref{tab:ivariantcaseA}.
\begin{itemize}
    \item[\textit{(i)}] The net $\Sigma_{\Calp} $ is pointwise 
    $\overline{H}$-invariant and the quadrics are enough to construct 
    the two expected invariants. This happens only in case \textit{(i)}.
    Moreover, from \Cref{lem:eulerdecomposition}, we have that the KHK discretisation of the Euler top fits in this class.
    \item[\textit{(ii)}] Although two cycle decompositions on $\Cale\cup\Calp$ are possible, 
    the rational functions in \Cref{tab:ivariantcaseA} are invariant in both cases.
    The two invariants are a ratio of quadrics and of quartics respectively.
    Since the ratio of quartics can be chosen as a square, it is easy to see that the
    function:
    \begin{equation}
    J^{\text{\textit{(ii)}}} = \frac{x_1x_2+x_3x_4}{x_1x_4+x_2x_3},
    \label{eq:J2ii}
    \end{equation}
    is an anti-invariant. That is, the function \eqref{eq:J2ii} is such that:
    \begin{equation}
        \begin{tikzcd}[column sep=huge]
            J \arrow[mapsto, "\Phi_{*}"]{r} & -J.
        \end{tikzcd}
        \label{eq:antinvariant}
    \end{equation}
    \item[\textit{(iii)}] The two invariants are a ratio of quadrics and of quartics respectively. Again, the ratio of quartics can be chosen as a square, so 
        that the function:
    \begin{equation}
        J^{\text{\textit{(iii)}}} = \frac{x_1 x_2+x_3 x_4}{(x_3-x_4)(x_1-x_2)},
    \label{eq:J2casea_1423}
    \end{equation}
    is an anti-invariant (see equation \eqref{eq:antinvariant}).
    \item[\textit{(iv)}] The two invariants are a ratio of quartics and 
        a ratio of sextics. We note that there is no anti-invariant, because the cycle 
        $(S_{14}\ S_{24}\ S_{34})$ does not produce $\overline{H}$-invariant quadrics (see
        \Cref{rem:3points}).
        On the other hand, taking the ratios of two $\overline{H}$-invariants coming respectively from the two 3-cycles $(S_{12}\ S_{23}\ S_{13})$ and
        $(S_{14}\ S_{24}\ S_{34})$,
        we get (see \Cref{rem:ellipticandnoint}) the following \emph{triple} of 3-invariants given by the following ratios
        of quadrics:
        \begin{equation}
            J^{\text{\textit{(iv)}}}_{1} = 
            \frac{(x_1+x_2) (x_3+x_4)}{(x_1-x_4)(x_2-x_3)},
            \quad
            J^{\text{\textit{(iv)}}}_{2} =
            -\frac{(x_1+x_4) (x_2+x_3)}{(x_1-x_3)(x_2-x_4)},
            \quad
            J^{\text{\textit{(iv)}}}_{3} =
            \frac{(x_1+x_3) (x_2+x_4)}{(x_1-x_2)(x_3-x_4)}.
            \label{eq:Jiv}
        \end{equation}
        Indeed, these three functions are cyclically permuted by the map.
        Clearly only two of them are functionally independent.
    \item[\textit{(v)}] The two invariants are both ratios of quartics.
        We note that there is no anti-invariant, because the cycle 
        $(S_{13}\ S_{24})$ does not produce $\overline{H}$-invariant quadrics (see
        \Cref{rem:3points}).
        On the other hand, from the 4-cycle $(S_{12}\ S_{23}\ S_{34}\ S_{14})$
        we get the pair of 4-invariants given by the following ratios
        of quadrics:
        \begin{equation}
            J^{\text{\textit{(v)}}}_{1} = 
            \frac{(x_1+x_4) (x_2+x_3)}{(x_1-x_4)(x_2-x_3)},
            \quad
            J^{\text{\textit{(v)}}}_{2} =
            \frac{(x_1+x_2) (x_3+x_4)}{(x_1-x_2)(x_3-x_4)}.
            \label{eq:Jv}
        \end{equation}
        Indeed, $J^{\text{\textit{(v)}}}_{1}$ and $J^{\text{\textit{(v)}}}_{2}$ are such that:
        \begin{equation}
            \begin{tikzcd}[column sep=huge]
            J_{i}^{\text{\textit{(v)}}} \arrow[mapsto, "\Phi_{*}"]{r} & J_{i+1} ^{\text{\textit{(v)}}} \arrow[mapsto, "\Phi_{*}"]{r} & \frac{1}{J_{i}^{\text{\textit{(v)}}}};
        \end{tikzcd}
        \quad\mbox{for } i=1,2,
        \end{equation}
        where the indices are taken modulo 2.
\end{itemize}

\begin{table}[ht]
    \begin{tabular}{c|c}
        \toprule
        &Invariants
        \\
        \midrule[\heavyrulewidth]
         \mbox{\textit{(i)}} &
        {\small $\displaystyle\frac{(x_3+x_4)(x_1+x_2)}{(x_2-x_3)(x_1-x_4)}$,
         $\displaystyle \frac{ (x_2+x_3)(x_1+x_4)}{(x_2-x_3)(x_1-x_4)} $}
         \\[2ex]
        \midrule
         
         \mbox{\textit{(ii)}} &
        {\small $\displaystyle\frac{x_1x_2+x_3x_4}{(x_3+x_4)(x_1+x_2)}$,
         $\displaystyle\left(\frac{(x_3-x_4)(x_1-x_2)}{(x_3+x_4)(x_1+x_2)}\right)^{2}$}
         \\[2ex]
        \midrule
        
        \mbox{\textit{(iii)}} &  
        
        {\small$\displaystyle
        \begin{gathered}
        \frac{(x_3+x_4) (x_1+x_2)}{(x_3-x_4) (x_1-x_2)}, 
         \displaystyle\left(\frac{x_1 x_2+x_3 x_4}{(x_3-x_4)(x_1-x_2)}\right)^2
        \end{gathered}$}
         
         \\[2ex]
        \midrule
        
         \mbox{\textit{(iv)}} &
         
        {\small$\displaystyle
        \begin{gathered}
        \frac{(x_1 x_2+x_1 x_3+x_1 x_4+x_2 x_3+x_2 x_4+x_3 x_4)^2}{(x_1 x_2+x_3 x_4)^2+(x_1 x_3+x_2 x_4)^2+(x_1 x_4+x_2 x_3)^2},
        \\
         \displaystyle\frac{(x_3+x_4) (x_1+x_2)  (x_2+x_4) (x_1+x_3) (x_2+x_3) (x_1+x_4)}{(x_3-x_4) (x_2-x_4) (x_2-x_3) (x_1-x_4) (x_1-x_3) (x_1-x_2)}
        \end{gathered}$}
         \\[2ex]
        \midrule
        
         \mbox{\textit{(v)}} &
         
        {\small$\displaystyle
        \begin{gathered}
        \frac{(x_2+x_4)^2 (x_1+x_3)^2+(x_2-x_4)^2 (x_1-x_3)^2}{(x_3+x_4) (x_1+x_2) (x_3-x_4) (x_1-x_2)-(x_2+x_3) (x_1+x_4) (x_2-x_3) (x_1-x_4)},
        \\
         \displaystyle\frac{(x_2+x_4) (x_1+x_3)(x_2-x_4)(x_1-x_3)}{\ssum{\circ,\bullet \in\Set{\pm 1}}{}\circ(x_1 + \circ x_2)(x_3 + \circ x_4)(x_1 + \bullet x_4)(x_2 + \bullet x_3) }
        \end{gathered}$}
         \\[3ex]
         \bottomrule
    \end{tabular}
    \caption{A resuming table of the invariants of the maps 
    $\Phi=g \circ \crem_{3}$ for $g$ chosen from \Cref{tab:excaseA}.}
    \label{tab:ivariantcaseA}
\end{table}

Finally, we highlight that, even if there are 96 elements of type \textit{(A)} swapping $\Cale$ and $\Calp$, only the invariants in \Cref{tab:ivariantcaseA} can occur. {We saw this via direct check on all the 96 elements of type \textit{(A)}} swapping $\Cale$ and $\Calp$. Thus, for a rational function, being invariant only depends upon the action of $g$ on $S_\Calq$.

\begin{remark}
    \label{ex:celledoni}
    The results of this section suggest a relation 
    with the KHK discretisation of another physically relevant class of systems,
    namely the \emph{quadratic three-dimensional Nambu systems} 
    \cite{Nambu1973,Vaisman1999}. That is, the system of ODEs
    \begin{equation}
        \vec{\dot{x}} = \grad_{\vec{x}} H_{1}(\vec{x}) \cross \grad_{\vec{x}} H_{2} (\vec{x}),
        \label{eq:nambuq}
    \end{equation}
    where, for $i=1,2$, $H_{i}(\vec{x})=\vec{x}^{T} A_{i} \vec{x}$ is a
    Nambu--Hamiltonian function. The system \eqref{eq:nambuq} is 
    clearly integrable in the sense of Liouville, since the two Nambu--Hamiltonians are first integrals. Moreover, the system 
    \eqref{eq:nambuq} is a generalisation of the Euler top 
    \eqref{eq:euler}, obtained when $A_{1}$ and $A_2$ are diagonal
    (see \cite{CelledoniMcLachlanOwrenQuispel2014}).
    It was proven in \cite{CelledoniMcLachlanOwrenQuispel2014} that the 
    KHK discretisation of the system \eqref{eq:nambuq} is Liouville
    integrable with the two following modified Nambu--Hamiltonians:
    \begin{equation}
        \label{eq:Hinambu}
        \widetilde{H}_i(\vec{x}) = \frac{H_i(\vec{x})}{1+4h^2 H_3(\vec{x})},\quad \mbox{ for }i=1,2,
        \mbox{ and }
        H_{3} (\vec{x})=\vec{x}^{T} A_1 \adj (A_2) A_1\vec{x}.
    \end{equation}
    The modified Nambu--Hamiltonians \eqref{eq:Hinambu} are ratio of
    quadratic polynomials.     
    Let $\Phi_{h}^{(N)},({\Phi_{h}^{(N)}})^{-1}\in\Bir( \Pj^{3})$ be the homogeneous maps 
    associated with the KHK discretisation of \Cref{eq:nambuq} (see \Cref{sec:motivations}).
    We conjecture that, up to conjugation by a projectivity, the map $\Phi_{h}^{(N)}$
    is of type \emph{(A)-(i)}, i.e.\ it exists an analogue of \Cref{lem:eulerdecomposition} 
    for the map $\Phi_{h}^{(N)}$. Unfortunately, in this case
    the computations are more intricate than in the case of the Euler top. 
    So, we conclude this subsection by giving evidences supporting thist last claim,
    but we leave a complete proof as a subject of further research.

    We start noticing that since $\deg\Phi_{h}^{(N)}=\deg({\Phi_{h}^{(N)}})^{-1}=3$ the degree
    of the polynomials $\kappa$ and $\lambda$ in \eqref{eq:kappadef} is eight.
    From a direct computation we see that these polynomials are squares,
    i.e. $\kappa=(\tilde{\kappa})^2$ and $\lambda=(\tilde{\lambda})^2$.
    Another direct check shows that $(\tilde{\kappa})^2$ and $(\tilde{\lambda})^2$ admit a factorisation of the following form:
    \begin{equation}
        \label{eq:nambufact}
        \kappa = \left(\kappa_{1}\kappa_{2}\kappa_{3}\kappa_{4} \right)^{2},
        \quad 
        \lambda=\left(\lambda_{1}\lambda_{2}\lambda_{3}\lambda_{4} \right)^{2},
    \end{equation}
    where $\deg\kappa_{i}=\deg\lambda_{i}=1$, for $ i=1,2,3,4$.
    For the sake of readability, we omit the explicit  form of the polynomials $\kappa_i$ and $\lambda_i$, $i=1,2,3,4$, because they are too cumbersome.
    
    Let us consider the net generated by the invariants $\tilde{H}_i$ in
    \Cref{eq:Hinambu}:
    \begin{equation}
        \label{eq:Qnambu}
        \Sigma^{(N)} =
        \Set{Q_{\mu,\nu,\xi}\subset \Pj^3 | 
            \left[\mu:\nu:\xi\right]\in\Pj^2, \;
            Q_{\mu,\nu,\xi}=\left\{\mu H_1 +\nu H_2 +\xi (x_4^2+4h^2H_3)=0 \right\}
        } .
    \end{equation}
It is generic in the sense that its base locus consists of eight distinct points. This check can be done via Macaulay2 \cite{M2}.

    The matrices $A_1$ and $A_2$ are positive definite symmetric quadratic forms. As a consequence, up to orthogonal linear maps (see \cite[Section 2.2]{CelledoniMcLachlanOwrenQuispel2014}), we can assume $A_1=\Id$ and
    \begin{equation}
        A_2 = 
        \begin{pmatrix}
            a_1 & a_2/2 & a_3/2
            \\
            a_2/2 & a_4 & a_5/2
            \\
            a_3/2 & a_5/2 & a_6
        \end{pmatrix},
    \end{equation}
    for some $a_i\in \C$, $i=1,\ldots,6$. 
    Then, the quadratic form $Q_{\mu,\nu,\xi}$ takes the
    form:
    \begin{equation}
        Q_{\mu,\nu,\xi}=
        [x_1:x_2:x_3:x_4]
        \mathcal{M}(\mu,\nu,\xi)
        \begin{bmatrix}
            x_1
            \\
            x_2
            \\
            x_3
            \\
            x_4
        \end{bmatrix},
    \end{equation}
    where $\mathcal{M}(\mu,\nu,\xi)=\mathcal{M}_{0}(\mu,\nu,\xi) +\xi h^{2}\mathcal{M}_{2}(\mu,\nu,\xi)$
    and the matrices $\mathcal{M}_{0}(\mu,\nu,\xi)$ and $\mathcal{M}_{2}(\mu,\nu,\xi)$ are
    given by:
    \begin{subequations}
        \begin{align}
            \mathcal{M}_{0}(\mu,\nu,\xi) &= 
            \begin{pmatrix}
                \mu+\nu a_1 & a_2\nu /2 & \nu a_3/2  & 0
                \\ 
                a_2 \nu/2 & \mu+\nu a_4 & a_5 \nu/2 &0
                \\
                a_3 \nu /2 & a_5 \nu/2 & \mu+\nu a_6 & 0
                \\
                0&0&0&\xi
            \end{pmatrix},
            \label{eq:M0}
            \\
            \mathcal{M}_{2}(\mu,\nu,\xi) &= 
            \begin{pmatrix}
                4a_4 a_6 - a_5^{2} & 
                -2a_2 a_6+ a_3 a_5 &
                a_2 a_5- 2a_3 a_4 & 
                0
                \\ 
                -2a_2a_6+a_3 a_5& 
                4 a_1a_6- a_3^{2}&
                -2 a_1 a_5+ a_3 a_2&
                0
                \\ 
                a_2 a_5-2 a_3 a_4&
                -2 a_1 a_5+ a_3 a_2 &
                4  a_1 a_4- a_2^{2} &
                0
                \\ 
                \noalign{\medskip}0&0&0&0
            \end{pmatrix}.
            \label{eq:M2}
        \end{align}
        \label{eq:Mnambu}%
    \end{subequations}

    We study now the base locus of the linear system $\Sigma^{(N)}$
    \eqref{eq:Qnambu}. To do so, we first show that there exist six
    reducible members of the net $\Sigma^{(N)}$. These members will
    consist of twelve distinct planes. Recall that the quadratic form
    associated to a symmetric matrix is factorisable if and only if its
    $3\times3$ minors vanish, i.e. its rank is two. We omit the explicit
    computations since they are rather cumbersome, yet straightforward.
    However, we present the solutions and comment them.

    There exist six solutions, which we divide in two families, namely:
    \begin{subequations}
        \begin{align}
            [\mu_{i}^{(1)}:\nu_{i}^{(1)}:\xi_{i}^{(1)}]
            &= [\gamma_{i}:1:0], \quad i=1,2,3, \\
            [\mu_{i}^{(2)}:\nu_{i}^{(2)}:\xi_{i}^{(2)}] &=
            [\pi_{2}(\delta_{i}):2\delta_{i}\pi_{1}(\delta_{i}):2\pi_{1}(\delta_{i})],
            \quad i=1,2,3, \label{eq:solmin1}
        \end{align} \label{eq:solsmin}%
    \end{subequations} where $\gamma_{i}$ are the roots of the polynomial equation:
    \begin{equation}
        \begin{aligned}
            \gamma^3&+(a_1+a_4+a_6) \gamma^2 
            \\ 
            +\left(a_1 a_4+a_1 a_6+a_4 a_6- \frac{a_2^2+a_3^2+ a_5^2}{4}\right) \gamma 
            \\
            &+a_1 a_4 a_6+\frac{a_2 a_3 a_5-a_1 a_5^2-a_2^2 a_6-a_3^2 a_4}{4}=0,
        \end{aligned} \label{eq:polgamma}
    \end{equation} 
    $\delta_{i}$ are the roots of the polynomial equation:
    \begin{equation}
        \begin{aligned}
            \delta^3 &-4 h^2 (a_1+a_4+a_6) \delta^2 +4 h^4 \left[4 (a_1
            a_4+ a_1 a_6+ a_4 a_6)-a_2^2-a_3^2-a_5^2\right] \delta \\
            &-16 h^6 (4 a_1 a_4 a_6-a_1 a_5^2-a_2^2 a_6+a_2 a_3 a_5-a_3^2
            a_4) =0,
        \end{aligned} \label{eq:poldelta}
    \end{equation} 
    and \begin{subequations}
    \begin{align}
        \pi_{1}(\delta_{i}) &= \delta_i a_5-2 h^2 (2 a_1 a_5-a_2 a_3),
        \label{eq:pi1} \\ \pi_{2}(\delta_{i}) &\begin{aligned}[t]
            =& -(2 a_1 a_5-a_2 a_3) \delta_i^2 \\ &+2 h^2 \left[
                \begin{gathered} (4 a_1^2 +a_2^2 +a_3^2 +a_5^2)a_5 \\
                -2 (a_1 a_2 a_3 +a_2 a_3 a_4 +a_2 a_3 a_6 + 2 a_4
                a_5 a_6) \end{gathered}
            \right] \delta_i \\ &+8 a_5 h^4 (4 a_1 a_4 a_6-a_1
            a_5^2-a_2^2 a_6+a_2 a_3 a_5-a_3^2 a_4).
        \end{aligned} \label{eq:pi2}
    \end{align} \label{eq:pols}%
    \end{subequations} 
    As a consequence, the net $\Sigma^{(N)}$
    contains exactly six reducible members. So, we end up again with a
    configuration of twelve planes. 
    A direct symbolic computation to see how the eight base points are 
    arranged on the twelve planes was not possible because of the high
    complexitly of the involved computations.

    So, to prove that up to conjugation the map $\Phi_{h}^{(N)}$ is of
    type \emph{(A)-(i)}, it remains to understand how the eight points
    are arranged with respect to the twelve planes.
\end{remark}

\subsection{Case \textit{(B)}} 

In this section, we will adopt the same notation as in
\Cref{subsec:casea}. Moreover, we will denote by $T_{ij}$, for $1\le
i< j\le 4$, the unique quadric of $\Sigma_\Calq$ passing trough $p_i$
and $p_j$, and by $\Cals$ the set
\begin{equation}
    \Cals=\Set{S_{ij}}_{1\le i< j\le 4}\cup \Set{T_{ij}}_{1\le i< j\le 4}.
\end{equation}

To simplify the description of the invariants we use the following
general result.

\begin{lemma}
    \label{lem:commutation}
    Let $h\in\Calc$ be an element of type \textit{(B)}. Then, $h$
    commutes with $\crem_3$, i.e.  $h\circ\crem_3 =\crem_3\circ h$.
\end{lemma}

\begin{proof}
    The projectivity $h$ can be represented by a permutation matrix
    with signs (see \Cref{possibilemma}). Let us denote by $\sigma$
    the permutation of the coordinates $\Set{1,\ldots,4}$ induced by
    $\sigma$. then, we have
    \begin{subequations}
        \begin{gather}
            \begin{tikzcd}[row sep=tiny, ampersand replacement=\&]   
                x_i\arrow[mapsto,"h"]{r} \& \pm x_{\sigma(i)} \arrow[mapsto,"\crem_3"]{r} \&  \pm \frac{1}{x_{\sigma(i)} }
            \end{tikzcd}    
            \\
            \begin{tikzcd}[row sep=tiny, ampersand replacement=\&]   
            x_i\arrow[mapsto,"\crem_3 "]{r} \& \frac{1}{x_i}\arrow[mapsto,"h"]{r} \&\pm \frac{1}{x_{\sigma(i)} }
            \end{tikzcd}
        \end{gather}
    \end{subequations}
    which completes the proof.
\end{proof}

\begin{lemma}
    \label{prop:conjB}
    Let $g_1,g_2\in\CB\subset \Calc$ be two elements of type \textit{(B)}
    which are conjugated in $\CB$, i.e. there exists $h\in\CB$ such that
    $g_2=h\circ g_1 \circ h^{-1}$. Let also $R_1\in\C(x_1,\ldots,x_4)$
    be an invariant of $\Phi_1=g_1\circ \crem_3$. Then, $R_2=R_1\circ h^{-1}$
    is an invariant for $\Phi_2=g_2\circ \crem_3$.
\end{lemma}
\begin{proof} The proof consists in the following chain of equalities
\begin{equation}
    \begin{aligned}[b]
        R_2\circ \Phi_2 &\equiv R_1\circ h^{-1}\circ g_2\circ\crem_3\equiv R_1\circ h^{-1}\circ h\circ g_1\circ h^{-1}\circ\crem_3
        \\
        &\equiv R_1\circ g_1\circ\crem _3\circ h^{-1}\equiv R_1\circ\Phi_1\circ h^{-1}\equiv R_2.
    \end{aligned}
\end{equation}
Notice that the last equality is a consequence of \Cref{lem:commutation}.
\end{proof}

To simplify the description of the invariants in case \emph{(B)}
we also use the following result.

\begin{lemma}
    \label{cor:conjBk} 
    Let $g\in\CB\subset \Calc$ be an
    element of type \textit{(B)} and let $R\in\C(x_1,\ldots,x_4)$ be an invariant of 
    $\Phi=g\circ \crem_3$. Consider the birational map $\Phi^{(k)}=g^k \circ \crem_3$. Then, the following rational function
    \begin{equation}
        \label{eq:Rt}
        \widetilde R = 
            {\crem_3}_{*}(R) + R,
    \end{equation}
    is an invariant of $\Phi^{(k)}$ for all $k\ge 0$. Moreover, if $k$ is odd also $R$ is an invariant for $\Phi^{(k)}$. While, if $k$ is even and  $\widetilde  R=0$, then $R^2$ is an invariant of $\Phi^{(k)}$. 
\end{lemma}

\begin{proof}
    The proof follows from \Cref{lem:commutation} which states that
    $g^k$ and $\crem_3$ commute.
\end{proof}

\begin{remark}
    \Cref{lem:commutation,prop:conjB,cor:conjBk} hold in general on any $\Pj^M$ by considering permutation matrices with signs of size $(M+1)\times(M+1)$ and the $M$-dimensional standard Cremona transformation $\crem_M$.
\end{remark}

\begin{remark}
\label{rem:nuoviinvarianti}
    We remark that, in the case in which the invariant $\widetilde  R$ given by \Cref{cor:conjBk} vanishes, we have that $R$ is
    an anti-invariant for $\Phi^{(k)}$ when $k$ is even. Moreover, somentimes, it can happen that $\widetilde{R}= R$, i.e. $R$ is invariant, also for $k$ even and not just for the odd case. 
    
    Finally, we observe that, given a set of $m$ functionally independent
    invariants $\Set{R_i}_{i=1}^m$ of $\Phi$,  it is not guaranteed
    that the set of invariants constructed in \Cref{cor:conjBk} are still functionally independent.
\end{remark}

As a consequence of \Cref{prop:conjB,cor:conjBk}, if one knows the invariants of some $g\circ\crem_3
$ then one can recover the invariants of all the maps of the form
$h\circ\crem_3$ for $h\in \CB$ in the same conjugacy class of $g$ or, for $h\in \CB$ of the form $g^k$ for some $k\ge 0$.
A direct computation  tells us that there are exactly 14 conjugacy
classes in $\CB$. One representative for each of them is given in
\Cref{tab:excaseB}, while the explicit form of the invariants is
shown in \Cref{tab:ivariantcaseB}. To shorten the latter, we will give a set of invariants that suffices to compute all the others via \Cref{cor:conjBk}. In \Cref{tab:relationB} we will show to which conjugacy classes belong the second power of the projectivities involved in \Cref{tab:excaseB}. Following \Cref{rem:nuoviinvarianti}, the invariants in \Cref{tab:ivariantcaseB} are chosen in a way that the invariants obtained via \Cref{cor:conjBk} are functionally independent. We checked this via a case-by-case analysis performed via computer algebra.

In the next proposition we summarise the fact that, in order to to build the
invariants, we only need elements of $\Sigma_{\mbox{\tiny\textit{(B)}}}$.

\begin{proposition}\label{bastasigmab}
    Let us consider $\Phi = g\circ \crem_3\in\Bir(\Pj^3)$, with $g\in\CB$.
    Then, the invariants of $\Phi$ are obtained as a suitable (nonlinear)
    combination of the elements of $\Sigma_{\mbox{\tiny\textit{(B)}}}$.
\end{proposition}

\begin{table}[ht]
    \centering
      \begin{tabular}{c|cccc}
    \hline&
$\ord$& \hspace{-10pt} $g$& {\small Action on $\Calr$}& \hspace{-15pt}{\small Action on $\Cals$ }\\
    \hline
 \hspace{-10pt} \mbox{\small\textit{(i)}} & 1&  \hspace{-10pt}{\footnotesize $
        \begin{bmatrix}
            1&\hspace{-7pt}0&\hspace{-7pt}0&\hspace{-7pt}0
            \\[-2pt] 
            0&\hspace{-7pt}1&\hspace{-7pt}0&\hspace{-7pt}0
            \\[-2pt]
            0&\hspace{-7pt} 0&\hspace{-7pt}1&\hspace{-7pt}0
            \\[-2pt] 
            0&\hspace{-7pt}0&\hspace{-7pt}0&\hspace{-7pt}1
    \end{bmatrix}$} &\hspace{-10pt} {\small $\Id_{\Calr}$} & \hspace{-15pt}{\small $\Id_{\Cals}$} \\\hline

 \hspace{-10pt} \mbox{\small\textit{(ii)}} & 2&  \hspace{-10pt}{\footnotesize $
        \begin{bmatrix}
           1&\hspace{-7pt}0&\hspace{-7pt}0&\hspace{-8pt}0
            \\[-2pt] 
            0&\hspace{-7pt}1&\hspace{-7pt}0&\hspace{-8pt}0
            \\[-2pt] 
            0&\hspace{-7pt} 0&\hspace{-7pt}1&\hspace{-8pt}0
            \\[-2pt] 
            0&\hspace{-7pt}0&\hspace{-7pt}0&\hspace{-8pt}-1
    \end{bmatrix}$} &\hspace{-10pt} {\footnotesize $(p_1\ q_1)(p_2\ q_2)(p_3\ q_3)(p_4\ q_4) $}&  \hspace{-15pt} {\footnotesize $  \begin{matrix}
        (S_{12}\ T_{12})(S_{13}\ T_{13}) (S_{14}\ T_{14})\\(S_{23}\ T_{23})(S_{24}\ T_{24}) (S_{34}\ T_{34})
    \end{matrix}$}\\\hline

 \hspace{-10pt} \mbox{\small\textit{(iii)}} & 2&  \hspace{-10pt}{\footnotesize $
        \begin{bmatrix}
        1& \hspace{-7pt}0&\hspace{-7pt} 0&\hspace{-7pt} 0\\[-2pt]
        0&\hspace{-7pt} 1&\hspace{-7pt} 0&\hspace{-7pt} 0\\[-2pt]
        0&\hspace{-7pt} 0&\hspace{-7pt} 0&\hspace{-7pt}1\\[-2pt]0&\hspace{-7pt} 0&\hspace{-7pt} 1&\hspace{-7pt} 0
    \end{bmatrix}$} &\hspace{-10pt} {\footnotesize $(e_3\ e_4)(p_3\ p_4)(q_1\ q_2) $}& \hspace{-15pt}{\footnotesize $  \begin{matrix}
        (S_{13}\ S_{23})(S_{14}\ S_{24})\\ (T_{13}\ T_{14})(T_{23}\ T_{24})
    \end{matrix}$} \\\hline

 \hspace{-10pt} \mbox{\small\textit{(iv)}} & 2&  \hspace{-10pt}{\footnotesize $
        \begin{bmatrix}
             1&\hspace{-7pt}0&\hspace{-8pt}0&\hspace{-10pt}0
            \\[-2pt] 
            0&\hspace{-7pt}1&\hspace{-8pt}0&\hspace{-10pt}0
            \\[-2pt] 
            0&\hspace{-7pt} 0&\hspace{-8pt}-1&\hspace{-10pt}0
            \\[-2pt] 
            0&\hspace{-7pt}0&\hspace{-8pt}0&\hspace{-10pt}-1
    \end{bmatrix}$} &\hspace{-10pt}{\footnotesize $(p_1\ p_2)(p_3\ p_4)(q_1\ q_2)(q_3\ q_4) $}& \hspace{-15pt}{\footnotesize $  \begin{matrix}
        (S_{13}\ S_{24})(S_{14}\ S_{23})\\ (T_{13}\ T_{24})(T_{14}\ T_{23})
    \end{matrix}$}  \\\hline

 \hspace{-10pt} \mbox{\small\textit{(v)}} & 2&  \hspace{-10pt}{\footnotesize $
        \begin{bmatrix}1&\hspace{-7pt} 0&\hspace{-7pt} 0&\hspace{-8pt} 0\\[-2pt]
        0&\hspace{-7pt} 0&\hspace{-7pt} 1&\hspace{-8pt} 0\\[-2pt]
        0&\hspace{-7pt} 1&\hspace{-7pt} 0&\hspace{-8pt} 0\\[-2pt]
        0&\hspace{-7pt} 0&\hspace{-7pt} 0&\hspace{-8pt} -1
    \end{bmatrix}$} &\hspace{-10pt} {\footnotesize $(e_2\ e_3)(p_1\ q_1)(p_2\ q_3)(p_3\ q_2)(p_4\ q_4) $}& \hspace{-15pt}{\footnotesize $  \begin{matrix}
        (S_{12}\ T_{13})(S_{13}\ T_{12})(S_{14}\ T_{14})\\ (S_{23}\ T_{23})(S_{24}\ T_{34})(S_{34}\ T_{24})
    \end{matrix}$}  \\\hline

 \hspace{-10pt} \mbox{\small\textit{(vi)}} & 2&  \hspace{-10pt}{\footnotesize $
        \begin{bmatrix}
0&\hspace{-7pt}1&\hspace{-7pt}0&\hspace{-7pt}0\\[-1pt]
1&\hspace{-7pt}0&\hspace{-7pt}0&\hspace{-7pt}0\\[-1pt]
0&\hspace{-7pt}0&\hspace{-7pt}0&\hspace{-7pt}1\\[-1pt]
0&\hspace{-7pt}0&\hspace{-7pt}1&\hspace{-7pt}0
    \end{bmatrix}$} &\hspace{-10pt} {\footnotesize $(e_1\ e_2)(e_3\ e_4)(p_1\ p_2)(p_3\ p_4) $}& \hspace{-15pt}{\footnotesize $ (T_{13}\ T_{24})(T_{14}\ T_{23})$}  \\\hline
    
 \hspace{-10pt} \mbox{\small\textit{(vii)}} & 2&  \hspace{-10pt}{\footnotesize $
        \begin{bmatrix}
        0&\hspace{-8pt}1&\hspace{-8pt}0&\hspace{-8pt}0\\[-1pt]
-1&\hspace{-8pt}0&\hspace{-8pt}0&\hspace{-8pt}0\\[-1pt]
0&\hspace{-8pt}0&\hspace{-8pt}0&\hspace{-8pt}1\\[-1pt]
0&\hspace{-8pt}0&\hspace{-8pt}-1&\hspace{-8pt}0
    \end{bmatrix}$} &\hspace{-10pt} {\footnotesize $ (e_1\ e_2)(e_3\ e_4)(p_1\ p_4)(p_2\ p_3)(q_1\ q_3)(q_2\ q_4)$}& \hspace{-15pt}{\footnotesize $  \begin{matrix}
        (S_{12}\ S_{34})(S_{14}\ S_{23})\\(T_{12}\ T_{34})(T_{13}\ T_{24}) 
    \end{matrix}$} \\\hline
      
   \hspace{-10pt}    \mbox{\small\textit{(viii)}} &3 & \hspace{-10pt} {\footnotesize $\begin{bmatrix}
           1&\hspace{-7pt}0&\hspace{-7pt}0&\hspace{-7pt}0
            \\[-2pt] 
            0&\hspace{-7pt}0&\hspace{-7pt}1&\hspace{-7pt}0
            \\[-2pt] 
            0&\hspace{-7pt} 0&\hspace{-7pt}0&\hspace{-7pt}1
            \\[-2pt] 
            0&\hspace{-7pt}1&\hspace{-7pt}0&\hspace{-7pt}0
    \end{bmatrix}$}&\hspace{-10pt} {\footnotesize $(e_2\ e_4\ e_3)(p_2\ p_4\ p_3)(q_1\ q_2\ q_3)$} & \hspace{-15pt} {\footnotesize $  \begin{matrix}
        (S_{12}\ S_{23}\ S_{13}) (S_{14}\ S_{24}\ S_{34})\\(T_{12}\ T_{14}\ T_{13}) (T_{23}\ T_{24}\ T_{34})
    \end{matrix}$}\\ \hline
      
   \hspace{-10pt}    \mbox{\small\textit{(ix)}} &4 & \hspace{-10pt} {\footnotesize $\begin{bmatrix}
          0&\hspace{-7pt}0&\hspace{-8pt}0&\hspace{-8pt}1\\[-1pt]
0&\hspace{-7pt}0&\hspace{-8pt}-1&\hspace{-8pt}0\\[-1pt]
1&\hspace{-7pt}0&\hspace{-8pt}0&\hspace{-8pt}0\\[-1pt]
0&\hspace{-7pt}1&\hspace{-8pt}0&\hspace{-8pt}0
    \end{bmatrix}$}&\hspace{-10pt} {\footnotesize $(e_1\ e_3\ e_2\ e_4)(p_1\ q_1\ p_4\ q_3)(p_2\ q_2\ p_3\ q_4) $} & \hspace{-15pt}{\footnotesize $  \begin{matrix}
        (S_{12}\ T_{34}\ S_{34}\ T_{12})(S_{13}\ T_{14})\\(S_{14}\ T_{24}\ S_{23}\ T_{13})(S_{24}\ T_{23})
    \end{matrix}$}\\ \hline

   \hspace{-10pt}    \mbox{\small\textit{(x)}} &4 & \hspace{-10pt} {\footnotesize $\begin{bmatrix}
   1&\hspace{-7pt}0&\hspace{-8pt}0&\hspace{-8pt}0\\[-1pt]
0&\hspace{-7pt}1&\hspace{-8pt}0&\hspace{-8pt}0\\[-1pt]
0&\hspace{-7pt}0&\hspace{-8pt}0&\hspace{-8pt}1\\[-1pt]
0&\hspace{-7pt}0&\hspace{-8pt}-1&\hspace{-8pt}0
    \end{bmatrix}$}&\hspace{-10pt} {\footnotesize $ (e_3\ e_4)(p_1\ q_1\ p_2\ q_2)(p_3\ q_4\ p_4\ q_3)$} & \hspace{-15pt}{\footnotesize $  \begin{matrix}
       (S_{12}\ T_{12})(S_{13}\ T_{23}\ S_{24}\ T_{14})\\(S_{14}\ T_{24}\ S_{23}\ T_{13})(S_{34}\ T_{34})
    \end{matrix}$}\\ \hline

 \hspace{-10pt} \mbox{\small\textit{(xi)}} & 4&  \hspace{-10pt}{\footnotesize $
        \begin{bmatrix}
        -1&\hspace{-10pt}0&\hspace{-8pt}0&\hspace{-7pt}0\\[-1pt]
0&\hspace{-10pt}0&\hspace{-8pt}1&\hspace{-7pt}0\\[-1pt]
0&\hspace{-10pt}-1&\hspace{-8pt}0&\hspace{-7pt}0\\[-1pt]
0&\hspace{-10pt}0&\hspace{-8pt}0&\hspace{-7pt}1
    \end{bmatrix}$} &\hspace{-10pt} {\footnotesize $(e_2\ e_3)(p_1\ p_3\ p_4\ p_2)(q_1\ q_3\ q_4\ q_2) $}& \hspace{-15pt}{\footnotesize $  \begin{matrix}
       (S_{12}\ S_{13}\ S_{34}\ S_{24})(S_{14}\ S_{23})\\(T_{12}\ T_{13}\ T_{34}\ T_{24})(T_{14}\ T_{23})
    \end{matrix}$}\\\hline

 \hspace{-10pt} \mbox{\small\textit{(xii)}} & 4&  \hspace{-10pt}{\footnotesize $
        \begin{bmatrix}
             0&\hspace{-7pt}1&\hspace{-7pt}0&\hspace{-8pt}0\\[-1pt]
1&\hspace{-7pt}0&\hspace{-7pt}0&\hspace{-8pt}0\\[-1pt]
0&\hspace{-7pt}0&\hspace{-7pt}0&\hspace{-8pt}-1\\[-1pt]
0&\hspace{-7pt}0&\hspace{-7pt}1&\hspace{-8pt}0
    \end{bmatrix}$} &\hspace{-10pt} {\footnotesize $ (e_1\ e_2)(e_3\ e_4)(p_1\ q_1\ p_2\ q_2)(p_3\ q_3\ p_4\ q_4)$}& \hspace{-15pt}{\footnotesize $  \begin{matrix}
       (S_{12}\ T_{12})(S_{13}\ T_{24}\ S_{24}\ T_{13})\\(S_{14}\ T_{23}\ S_{23}\ T_{14})(S_{34}\ T_{34})
    \end{matrix}$} \\\hline

 \hspace{-10pt} \mbox{\small\textit{(xiii)}} & 4&  \hspace{-10pt}{\footnotesize $
        \begin{bmatrix}0&\hspace{-7pt}1&\hspace{-7pt}0&\hspace{-7pt}0\\[-1pt]
0&\hspace{-7pt}0&\hspace{-7pt}1&\hspace{-7pt}0\\[-1pt]
0&\hspace{-7pt}0&\hspace{-7pt}0&\hspace{-7pt}1\\[-1pt]
1&\hspace{-7pt}0&\hspace{-7pt}0&\hspace{-7pt}0
    \end{bmatrix}$} &\hspace{-10pt} {\footnotesize $(e_1\ e_4\ e_3\ e_2)(p_1\ p_4\ p_3\ p_2)(q_1\ q_3)$}& \hspace{-15pt}{\footnotesize $ \begin{matrix}
       (S_{12}\ S_{23})(S_{14}\ S_{34})\\(T_{12}\ T_{14}\ T_{34}\ T_{23})(T_{13}\ T_{24})
    \end{matrix}$} \\\hline
      
 \hspace{-10pt}   \mbox{\small\textit{(xiv)}} &6&   \hspace{-10pt} {\footnotesize $
   \begin{bmatrix}
           -1&\hspace{-8pt}0&\hspace{-7pt}0&\hspace{-7pt}0
            \\[-2pt]  
            0&\hspace{-8pt}0&\hspace{-7pt}0&\hspace{-7pt}1
            \\[-2pt]  
            0&\hspace{-8pt} 1&\hspace{-7pt}0&\hspace{-7pt}0
            \\[-2pt]  
            0&\hspace{-8pt}0&\hspace{-7pt}1&\hspace{-7pt}0
    \end{bmatrix} $}&\hspace{-10pt} {\footnotesize $ (e_2\ e_3\ e_4)(p_1\ q_4)(p_2\ q_2\ p_4\ q_3\ p_3\ q_1)  $}& \hspace{-5pt} {\footnotesize $\begin{matrix}
        (S_{12}\ T_{24}\ S_{23}\ T_{34}\ S_{13}\ T_{23})\\(S_{14}\ T_{12}\ S_{24}\ T_{14}\ S_{34}\ T_{13})
    \end{matrix}$}\\
      \hline 
    \end{tabular}
    \caption{Representatives for the 14 conjugacy classes in $\CB$.}
    \label{tab:excaseB}
\end{table}

\begin{table}[ht]
    \begin{tabular}{c|c}
        \toprule
        &Invariants
        \\
        \midrule[\heavyrulewidth]
         \mbox{\textit{(ii)}} &
        {\small $\displaystyle\frac{ (x_2^2-x_4^2)(x_1+x_3)^2}{(x_3^2-x_4^2) (x_1+x_2)^2 }$,
         $\displaystyle \frac{(x_2-x_3)^2 (x_1^2-x_4^2)}{(x_3^2-x_4^2)(x_1+x_2)^2} $,
         $\displaystyle \frac{(x_2+x_3)^2(x_1^2-x_4^2) }{(x_3^2-x_4^2)(x_1+x_2)^2} $}
         \\[2ex]
        \midrule
         
         \mbox{\textit{(iii)}} &
        { $ \begin{array}{c}
           \mbox{\small$\displaystyle \frac{(x_2-x_4)(x_1+x_3)+(x_2-x_3)(x_1+x_4) }{(x_3-x_4)(x_1-x_2) },\    \frac{(x_2-x_3)(x_1-x_4)+(x_2-x_4)(x_1-x_3) }{(x_3+x_4)(x_1+x_2) }$},  \\[2ex]\mbox{\small$\displaystyle \frac{ (x_2-x_4)(x_1+x_3)-(x_2-x_3)(x_1+x_4)}{(x_3+x_4)(x_1+x_2) }$}
        \end{array}$} 
         \\[2ex]
        \midrule

         \mbox{\textit{(vi)}} &
        {\small $\displaystyle\frac{(x_2-x_3) (x_1-x_4) }{ (x_3+x_4) (x_1+x_2)}$,
         $\displaystyle \frac{(x_2+x_3) (x_1+x_4)}{(x_3+x_4) (x_1+x_2)} $,
         $\displaystyle \frac{ (x_2-x_4)(x_1+x_3)-(x_2+x_4) (x_1-x_3) }{ (x_3+x_4) (x_1+x_2)} $}
         \\[2ex]
        \midrule

        \mbox{\textit{(ix)}} &{ $ \begin{array}{c}
           \mbox{{\small $\displaystyle \frac{2x_1x_2+x_1x_3-x_1x_4-x_2x_3-x_2x_4}{  2 x_1x_2-x_1x_3+x_1x_4+x_2x_3+x_2x_4}$,
         $\displaystyle\frac{(x_2+x_4) (x_1+x_3) (x_2-x_3) (x_1+x_4)}{(x_3^2+x_4^2) (x_1^2+x_2^2)}$,}}  \\[2ex]\mbox{\small $\displaystyle \frac{(x_2-x_4) (x_1-x_3) (x_2+x_3) (x_1-x_4)}{(x_3^2+x_4^2) (x_1^2+x_2^2)}$}
        \end{array}$}
         
         \\[2ex]
        \midrule

        \mbox{\textit{(x)}} &{ $ \begin{array}{c}
           \mbox{{\small $\displaystyle \frac{2 x_1 x_2+x_1 x_3-x_1 x_4-x_2 x_3-x_2 x_4}{2 x_1 x_2-x_1 x_3+x_1 x_4+x_2 x_3+x_2 x_4}$,
         $\displaystyle\frac{(x_2+x_4) (x_1+x_3) (x_2-x_3) (x_1+x_4)}{(x_3^2+x_4^2) (x_1^2+x_2^2)}$,}}  \\[2ex]\mbox{\small $\displaystyle \frac{(x_2-x_4) (x_1-x_3) (x_2+x_3) (x_1-x_4)}{(x_3^2+x_4^2) (x_1^2+x_2^2)}$}
        \end{array}$}
         
         \\[2ex]
        \midrule

        \mbox{\textit{(xi)}} &{ $ \begin{array}{c}
           \mbox{{\small$\displaystyle \frac{(x_1 x_4-x_2 x_3)^2}{(x_1 x_4+x_2 x_3)^2  },\displaystyle\frac{x_1^2 x_2^2+x_1^2 x_3^2+4 x_1 x_2 x_3 x_4+x_2^2 x_4^2+x_3^2 x_4^2}{(x_1 x_4+x_2 x_3)^2}$}},  \\[2ex]\mbox{\small$\displaystyle \frac{x_1^2 x_2^2-2 x_1^2 x_2 x_3-x_1^2 x_3^2+x_2^2 x_4^2+2 x_2 x_3 x_4^2-x_3^2 x_4^2}{(x_1 x_4+x_2 x_3)^2}$}
        \end{array}$}
         
         \\[2ex]
        \midrule
        \mbox{\textit{(xii)}} &
        {\small $\displaystyle \frac{(x_1-x_2)}{(x_1+x_2)}$,
         $\displaystyle\frac{2 (x_2^2+x_4^2) (x_1^2+x_3^2)}{(x_1+x_2)^2 (x_3^2+x_4^2)}$,
         $\displaystyle \frac{2 (x_2^2+x_3^2) (x_1^2+x_4^2)}{(x_1+x_2)^2 (x_3^2+x_4^2)}$}
         \\[2ex]
        \midrule

        \mbox{\textit{(xiii)}} &{ $ \begin{array}{c}
           \mbox{\small$\displaystyle \frac{(x_2+x_4) (x_1+x_3)}{  (x_1 x_2+2 x_1 x_3+x_1 x_4+x_2 x_3+2 x_2 x_4+x_3 x_4)}$},\\[2ex]   \mbox{\small$\displaystyle\frac{(x_1 x_3-x_2 x_4)}{(x_1 x_2+2 x_1 x_3+x_1 x_4+x_2 x_3+2 x_2 x_4+x_3 x_4)}$},  \\[2ex]\mbox{\small$\displaystyle \frac{(x_1^2 x_2^2+x_1^2 x_4^2-4 x_1 x_2 x_3 x_4+x_2^2 x_3^2+x_3^2 x_4^2)}{(x_1^2 x_2^2+2 x_1^2 x_3^2+x_1^2 x_4^2-8 x_1 x_2 x_3 x_4+x_2^2 x_3^2+2 x_2^2 x_4^2+x_3^2 x_4^2)}$}
        \end{array}$}
         
         \\[2ex]
        \midrule
        
        \mbox{\textit{(xiv)}} &{ $ \begin{array}{c}
           \mbox{{\small $\displaystyle \frac{\ssum{ \Set{i,j,k}=\Set{2,3,4}}{ } (x_1^2+ x_i^2)(x_j-x_k)^2}{ \ssum{ \Set{i,j,k}=\Set{2,3,4}}{ } (x_1^2+ x_i^2)(x_j+x_k)^2 }$,
        $\displaystyle\frac{ \ssum{ \Set{i,j,k}=\Set{2,3,4}}{ }x_1x_i(x_j-x_k)^2 }{ \ssum{ \Set{i,j,k}=\Set{2,3,4}}{ }x_1x_i(x_j+x_k)^2 }$,}}  \\[2ex]\mbox{\small $\displaystyle \frac{ x_1 x_2+x_1 x_3+x_1 x_4+x_2 x_3+x_2 x_4+x_3 x_4)) (x_1 x_2+x_1 x_3+x_1 x_4-x_2 x_3-x_2 x_4-x_3 x_4}{ \ssum{ \Set{i,j,k}=\Set{2,3,4}}{ }x_1x_i(x_j+x_k)^2 }$}
        \end{array}$}
                 
         \\[3ex]
         \bottomrule
    \end{tabular}
    \caption{A resuming table of the invariants of the maps 
    $\Phi=g \circ \crem_{3}$ for $g$ chosen from \Cref{tab:excaseB} of order 2 and 3.}
    \label{tab:ivariantcaseB}
\end{table}

\begin{table}[ht]
    \begin{tabular}{c|cc}
    \toprule
    Case & Conjugacy class of the square & Invartiants
    \\
    \hline
    \textit{(ix)} & \textit{(vii)}  & $\Set{\tilde{R}_1, R_2^2, R_3^2}$
    \\
    \textit{(x)} & \textit{(iv)} & $\Set{\tilde{R}_1, R_2^2, R_3^2}$
    \\
    \textit{(xi)} &  \textit{(iv)} & $\Set{R_1, R_2, R_3^2}$
    \\
    \textit{(xii)} & \textit{(iv)} & $\Set{R_1^2, R_2, R_3}$
    \\
    \textit{(xiii)} &  \textit{(v)} & $\Set{R_1, \tilde{R}_2, R_3^2}$
    \\
    \textit{(xiv)} & \textit{(viii)}  & $\Set{R_1, R_2, R_3^2}$
    \\
    \bottomrule
    \end{tabular}
    \caption{Relations between conjugacy classes in $\CB$ and invariants. The invariants
    are build following \Cref{cor:conjBk}. Here, $R_i$ denotes the $i$-th invariant
    in \Cref{tab:ivariantcaseB} of the case in first column.}
    \label{tab:relationB}
\end{table}

\subsection{Case \textit{(C)}} \label{sss:casseCinva}

In this case we have no $\overline{H}$-invariant linear system of quadrics
(see \Cref{prop:noquadricsB}) and there is only a pencil $\Xi_\Calr$
of $2\overline{H}$-invariant quartics (see \Cref{prop:quarticsC}).

Consider the following three points $ S_{12,34},S_{13,24},S_{14,23}\in\Xi_\Calr$:
\begin{subequations}
    \begin{align}
       S_{12,34} &= H_{12}^++H_{12}^- + H_{34}^++H_{34}^-,
       \\
       S_{13,24} &= H_{13}^++H_{13}^-+ H_{24}^++H_{24}^-,
       \\
       S_{14,23} &= H_{14}^++H_{14}^-+H_{23}^++H_{23}^-,
    \end{align}
\end{subequations}
corresponding to the desmic surfaces (see \Cref{desmicclassic}).  Notice that
any $g$ of type \textit{(C)} acts on $\Set{S_{12,34},S_{13,24},S_{14,23}}$
and this action uniquely determines the action of $g$ on $\Xi_\Calr\cong\Pj^1$.

\begin{remark}\label{rem:invcasec}
    Let $g\in \Calc$ be an element of type \textit{(C)}. Then, the
    existence of an invariant of degree $d$, for $d\ge 4$ depends on
    the action of $g$ on $\Set{S_{12,34},S_{13,24},S_{14,23}}$.
    Precisely, it depend on the number of fixed points of
    $g|_{\Set{S_{12,34},S_{13,24},S_{14,23}}}$.
\end{remark}

In this case we will not present tables similar to
\Cref{tab:excaseA,tab:ivariantcaseA}. Instead, we
will give an example for each possible  fixed locus of
$g|_{\Set{S_{12,34},S_{13,24},S_{14,23}}}$ (or equivalently the fixed
locus of $g\colon \Xi_\Calr\to\Xi_\Calr$). These examples are listed
in \Cref{tab:ivariantcaseC}.

    \begin{table}[ht]
    \begin{tabular}{ccc}
        \toprule
        Fix pts &$g$ & Invariant
        \\\hline
        $\Xi_\Calr$ &{\footnotesize $
        \begin{bmatrix}
            1&\hspace{-7pt}-1&\hspace{-7pt}-1&\hspace{-7pt}1
            \\ 
            -1&\hspace{-7pt}1&\hspace{-7pt}-1&\hspace{-7pt}1
            \\ 
            -1&\hspace{-7pt}-1&\hspace{-7pt}1&\hspace{-7pt}1
            \\ 
            -1&\hspace{-7pt}-1&\hspace{-7pt}-1&\hspace{-7pt}-1
    \end{bmatrix}$} &{\small $\displaystyle\frac{(x_1^2-x_2^2)(x_3^2-x_4^2)}{(x_1^2-x_3^2)(x_2^2-x_4^2)} $}
    \\\hline
    $\Set{S_{12,34}}$ & {\footnotesize $
        \begin{bmatrix}
            1&\hspace{-7pt}-1&\hspace{-7pt}1&\hspace{-7pt}-1
            \\ 
            -1&\hspace{-7pt}1&\hspace{-7pt}1&\hspace{-7pt}-1
            \\ 
            -1&\hspace{-7pt}-1&\hspace{-7pt}1&\hspace{-7pt}1
            \\ 
            -1&\hspace{-7pt}-1&\hspace{-7pt}-1&\hspace{-7pt}-1
    \end{bmatrix}$}& {\small $\displaystyle\frac{(x_1^2-x_3^2)(x_2^2-x_4^2)(x_1^2-x_4^2)(x_2^2-x_3^2)}{(x_1^2-x_2^2)^2 (x_3^2-x_4^2)^2} $}
    \\\hline
    $\Set{S_{13,24}}$ & {\footnotesize $
        \begin{bmatrix}
            1&\hspace{-7pt}1&\hspace{-7pt}-1&\hspace{-7pt}-1
            \\ 
            -1&\hspace{-7pt}1&\hspace{-7pt}-1&\hspace{-7pt}1
            \\ 
            -1&\hspace{-7pt}1&\hspace{-7pt}1&\hspace{-7pt}-1
            \\ 
            -1&\hspace{-7pt}-1&\hspace{-7pt}-1&\hspace{-7pt}-1
    \end{bmatrix}$}& {\small $\displaystyle\frac{(x_1^2-x_2^2)(x_3^2-x_4^2)(x_1^2-x_4^2)(x_2^2-x_3^2)}{(x_1^2-x_3^2)^2 (x_2^2-x_4^2)^2} $}
    \\\hline
    $\Set{S_{14,23}}$ & {\footnotesize $
        \begin{bmatrix}
            1&\hspace{-7pt}1&\hspace{-7pt}-1&\hspace{-7pt}-1
            \\ 
            -1&\hspace{-7pt}1&\hspace{-7pt}1&\hspace{-7pt}-1
            \\ 
            -1&\hspace{-7pt}-1&\hspace{-7pt}-1&\hspace{-7pt}-1
            \\ 
            -1&\hspace{-7pt}1&\hspace{-7pt}-1&\hspace{-7pt}1
    \end{bmatrix}$}& {\small $\displaystyle\frac{(x_1^2-x_2^2)(x_3^2-x_4^2)(x_1^2-x_3^2)(x_2^2-x_4^2)}{(x_1^2-x_4^2)^2 (x_2^2-x_3^2)^2} $}
    \\\hline
    $\emptyset$ &{\footnotesize $
        \begin{bmatrix}
            1&\hspace{-7pt}1&\hspace{-7pt}-1&\hspace{-7pt}-1
            \\ 
            -1&\hspace{-7pt}1&\hspace{-7pt}1&\hspace{-7pt}-1
            \\ 
            -1&\hspace{-7pt}1&\hspace{-7pt}-1&\hspace{-7pt}1
            \\ 
            -1&\hspace{-7pt}-1&\hspace{-7pt}-1&\hspace{-7pt}-1
    \end{bmatrix}$}&{\small $\begin{matrix}
        \displaystyle\frac{(x_1^2-x_2^2) (x_3^2-x_4^2) (x_1^2-x_3^2)(x_2^2-x_4^2) (x_1^2-x_4^2) (x_2^2-x_3^2) }{ \ssum{i=0}{2} ((x_{\sigma(1)}-x_{\sigma(2)})(x_{\sigma(3)}-x_{\sigma(4)}))^2(x_{\sigma(1)}-x_{\sigma(3)})(x_{\sigma(2)}-x_{\sigma(4)}) }\\
        \mbox{ for } \sigma =(2 \ 3 \ 4 ).
    \end{matrix}$}
    \\\hline
    \end{tabular}
    \caption{Example of the invariant of the maps of the form
    $\Phi=g \circ \crem_{3}$ for $g$ of type \textit{(C)}.}
    \label{tab:ivariantcaseC}
    \end{table}

\begin{remark}
    \label{rem:nonexistence}
    We remark that no functionally independent invariants other than 
    the ones presented in \Cref{tab:ivariantcaseC} do exist for these maps.
    Indeed, if such an invariant would exist, one could use it to define
    a pre-symplectic structure compatible with the map (see \cite{ByrnesHaggarQuispel1999}). However, from
    the results in \cite{Bellon1999}, this would force the degree growth
    to be at most polynomial contradicting \Cref{prop:caseBgrowth}.
\end{remark}

\section{Conclusions}
\label{sec:concl}

In this paper, motivated by the study of the KHK discretisation of the
Euler top, we introduced a finite subgroup of $\Pj\Gl(4,\C)$, we called
it the Cremona-cubes group and we denoted it by $\Calc$. This group is
crafted in a way that it encompasses and generalises all the geometrical
properties of the KHK discretisation of the Euler top. Indeed, the KHK
discretisation of the Euler top is projectively equivalent to the map
$\Phi^{(0)} = g_0\circ \crem_3$, with $g_0$ as in \eqref{eq:eulerkhk2},
swapping $\Cale$ and $\Calp$. So, in order to generalise this behaviour,
we considered all the projectivities acting \emph{with finite order}
on all the special points of the standard Cremona transformation. This
yields the Cremona-cubes group (see \Cref{def:cubegroup}).

We studied the algebraic properties of the group $\Calc$ and we
splitted it in three disjoint subsets, namely \emph{(A)}, \emph{(B)},
and \emph{(C)}, depending on how the projectivity $g\in\Calc$ acts
on $\Set{\Cale,\Calp,\Calq}$ (see \Cref{orbits}). The birational
maps of the form $\Phi=g\circ \crem_3$, for $g\in\Calc$ have diffent
growth, covariance, and invariance properties depending on the type
of $g$. Type \emph{(A)} gives integrable maps, both in the algebraic
entropy (see \Cref{prop:caseAgrowth}) and in the Liouville--Poisson
sense (see \Cref{tab:ivariantcaseA}). Type \emph{(B)} consists of
permutations with signs of the homogeneous coordinates yielding
periodic maps (see \Cref{prop:caseC}).  Type \emph{(C)} gives
non-integrable maps (see \Cref{prop:caseBgrowth}), which however do
preserve an invariant of order at least four and at most twelve (see
\Cref{tab:ivariantcaseC}). The growth and convariance properties are
discussed in a unified way for all the elements of a given type (see
\Cref{prop:quadricsA,covcasec,covac2,prop:quarticsC} for the covariance).

On the other hand, the construction of the invariants is specific
to some subclasses of maps. For the maps defined from elemets
of type \emph{(A)} we characterise completely the invariants up
to their action on $\Calq$ (resp. $\Calp$) if $g$ swaps $\Cale$
and $\Calp$ (resp. $\Cale$ and $\Calq$): that is, if the actions of
$g,g'\in\Calc$ of type \emph{(A)} agree on $\Calp$ (resp. $\Calq$), then
$g\circ\crem_{3}$ and $g'\circ\crem_{3}$ possess the same invariants
(see \Cref{tab:excaseA,tab:ivariantcaseA}).  It is remarkable that
even though the invariants of $g\circ \crem_3$ depend only on the
action of $g$ on $\Calq$ (resp. $\Calp$) one may not be able to recover
this action starting only from the invariants (see case \emph{(ii)} in
\Cref{tab:excaseA,tab:ivariantcaseA}). This approach is different from
the construction made in \cite{Alonso_et_al2022}. There, to a pair of
quadrics, it is associated a \emph{single} birational map for which the
quadrics are invariant through the construction of some involution. The
study of invariants of the maps of the form $g\circ\crem_{3}$
with $g\in\Calc$ of type \emph{(B)} is based on the fact that these
projectivities constitute a group, namely $\CB\subset\Calc$.  We have
shown that, knowing invariants of $g\circ\crem_3$, for a given $g\in\CB$,
it is possible to recover the invariants of the maps $h\circ\crem_3$
for all $h\in\CB$ conjugated, in $\CB$ to $g$, or of the maps of the
form $g^k\circ\crem_3$, for $k\ge0$ (see \Cref{prop:conjB,cor:conjBk}).
We gave the invariants for a chosen representative in each of the ten
independent conjugacy classes (see \Cref{tab:ivariantcaseB}).  Lastly,
when $g\in\Calc$ is of type \emph{(C)}, we characterised completely the
invariants of $g\circ \crem_3$ according to its action on the three desmic
quartics belonging to the pencil introduced in \Cref{prop:quarticsC}.
Since $g$ acts on a set of three elements only three behaviours are
possible: fixing the three surfaces and hence the whole pencil, fixing
only one surface, or fixing no surface at all. We also presented the
invariats in this case (see \Cref{tab:ivariantcaseC}).

We remark that our construction of the invariants consists
of a geometric argument which makes our work more similar to
\cite{QRT1988,Alonso_et_al2022,GJTV_class} than to other papers
where invariants where successfully constructed for given maps with
a given algorithm (see \cite{JoshiViallet2017,Gubbiotti_Levi70,
CelledoniEvripidouMcLareOwewnQuispelTapleyvanderKamp2019,
PetreraPfadlerSuris2009,PetreraPfadlerSuris2011,PetreraSuris2010,GJTV_sanya}).
The main difference in this work is that we are able to characterise
at once maps with all the three possible behaviours, and, even
more interestingly, the number of integrable maps derived from this
construction is the same as the number of periodic and non-integrable
ones. This is strikingly surprising as integrable maps are deemed to be
very rare.

We note that, in the construction of the Cremona-cubes group $\Calc$ it is crucial the existence of a positive integer $k\in\N$ such that
\begin{equation}
    \label{eq:cond3}
    k\cdot  (\dim \Pj^3+1) = \abs{\Fix \crem_3}.
\end{equation}
Indeed, $\dim \Pj^3+1=4$ and $\abs{\Fix\crem_3}=8$ implying $k=2$. This
translates in the fact that the vertices of the three-dimensional cube splits
into the sets of vertices of two distinct three-dimensional simplices. In
general, the number $k$ corresponds to the number of distinct
$M$-dimensional simplices in which we want to split the set of vertices
of a $M$-dimensional cube. So, a similar construction might be possible
only if there exists a positive integer $k\in\N$ such that
\begin{equation}
    \label{eq:condn}
    k\cdot (\dim \Pj^M+1) = \abs{\Fix \crem_M}.
\end{equation}
Explicitly, since $\abs{\Fix \crem_M}=2^M$, we want the following equality to be satisfied
\begin{equation}
    \label{eq:condnexpl}
    k\cdot  (M+1) = 2^M.
\end{equation}
Now, from \Cref{eq:condnexpl}, it directly follows that  $k=2^h$ and $(M+1)=2^m$ where $h,m\in\N$ are positive integers such that $h+m=M$. This implies:
\begin{equation}
    \label{eq:ndef}
    M = 2^m-1, \quad k = 2^{2^m-1-m},
\end{equation}
so that everything is determined by the free parameter $m\in\N$. From \Cref{tab:2mk}, we see that
the next possible case is in $\Pj^7$. This already brings the number
of subsets of $\Fix\crem_3$ formed by orthogonal sets of
points of $\Pj^7$ to 16.  We expect  their dynamics to be
quite involved. Yet, we believe that it might be possible to construct an
analogue of the Cremona-cubes group with similar nice properties.  Work is
in progress in that direction.  Incidentally, we note that the number $k$
increases dramatically as $m$ raises. E.g., again from \Cref{tab:2mk}, we see that,
for $m=10$, $k$ is a number with 305 digits.

\begin{table}[hbt]
    \centering
    \begin{tabular}{c|cccccccccccc}
        \toprule
         $m$ &  1 & 2 & 3 & 4& 5& 6& 7& 8& 9& 10 %
         \\
         \midrule
         $M$ &  1 & 3 & 7 & 15 & 31 & 63 & 127 & 255 & 511 & 1023 %
         \\
         $k$ & 1 & 2 & 16 & 2048 & $\sim 10^{7}$ &  $\sim 10^{17}$ & $\sim 10^{37}$ & $\sim 10^{75}$ & $\sim 10^{151}$ & $\sim 10^{304}$ %
         \\
         \bottomrule
    \end{tabular}
    \caption{The first 10 values of $M$ and $k$ for which a construction similar to
    the one we carried out in this paper can be possible.}
    \label{tab:2mk}
\end{table}

We observe now that other generalisations without the need
of raising the dimension are possible.
Indeed, one can consider different singularity patterns contracting planes
to points. This, happens easily by considering degenerations of the Nambu
systems discussed in \Cref{ex:celledoni}: tuning the free parameters in
a way that some of the $\lambda_i$'s and/or $\kappa_i$'s collapse
alters the singularity structure. We are working to characterise these
degenerations and their geometric properties.
Moreover, in $\Pj^3$, it is possible to contract surfaces
both on points and on curves. The standard Cremona transformation \eqref{eq:C3} 
only contracts planes to points (see \Cref{sec:3dimcremona}). On the other hand the, map
\begin{equation}
    \label{eq:pseudocrem}
    \begin{tikzcd}[row sep=tiny]
        \theta \colon\Pj^3 \arrow[dashed]{r} & \Pj^3 \\
        {[}x_1:x_2:x_3:x_4{]} \arrow[mapsto]{r} & \left[x_{2}x_{3}: x_1x_3: x_1 x_2: x_3 x_4\right].
    \end{tikzcd}
\end{equation}
contracts the coordinate planes $\Set{x_i=0}$, for $i=1,2$, to the coordinate lines
$\Set{x_2=x_3=0}$ and $\Set{x_1=x_3=0}$ respectively, the coordinate plane $\Set{x_3=0}$ to
the point $[0:0:1:0]$, while the coordinate plane $\Set{x_4=0}$ is mapped to itself.
It is easy to see that composing $\theta$ with the projectivity $g_0$ in \eqref{eq:g0def}
yields a map $\Phi_\theta$ whose degree growth is heuristically computed to be:
\begin{equation}
    1, 2, 4, 7, 12, 18, 25, 34, 44, 55, 68, 82, 97, 114, 132, 151\ldots
\end{equation}
and it is fitted by the following generating function:
\begin{equation}
    \label{eq:gtheta}
    g_\theta(z) = -\frac{2z^4 + z^2 + 1}{(z - 1)^3(z^2 + z + 1)}.
\end{equation}
From \eqref{eq:gtheta}, we infer that the algebraic entropy of
$\Phi_\theta$ vanishes and that the degree growth is asymptotically quadratic.
Nevertheless, the singularity patterns are changed: there is one singularity pattern of the form:
\begin{equation}
    \label{eq:pseudocremsing1}
    \begin{tikzcd}[row sep=tiny]
        \Set{\text{plane}} \arrow{r} & \Set{\text{point}} \arrow{r} 
        & \Set{\text{point}} \arrow[dashed]{r} & \Set{\text{plane}},
    \end{tikzcd}
\end{equation}
analogous to the one of \Cref{fig:singularitieseuler}, but also two
of a new kind:
\begin{equation}
    \label{eq:pseudocremsing2}
    \begin{tikzcd}[row sep=tiny]
        \Set{\text{plane}} \arrow{r} & \Set{\text{line}} \arrow{r} 
        & \Set{\text{line}} \arrow{r} & \Set{\text{line}} \arrow[dashed]{r} 
        & \Set{\text{plane}}.
    \end{tikzcd}
\end{equation}
Work is in progress to understand the geometry of this map and its 
possible extensions.

\subsection*{Acknowledgments}

This work was made in the framework of the Project ``Meccanica dei
Sistemi discreti'' of the GNFM unit of INDAM.
In particular, MG acknowledge support by the MUR through the project PRIN 2020 
``Squarefree Gr\"obner degenerations, special varieties and related topics''~(Project number 2020355B8Y)
and GG acknowledge support of the GNFM through Progetto Giovani GNFM 2023:
``Strutture variazionali e applicazioni delle equazioni alle differenze ordinarie''~(CUP\_E53C22001930001).
MG would like to thank SISSA for its hospitality and support, during his visit
supported by the project ``Nested Hilbert schemes and GIT stability conditions''.

We thank Prof. Bert van Geemen and Prof. Claude-Michel Viallet for their help and assistance during
the preparation of this paper.
In particular, we thank Prof. Bert van Geemen for sharing his knowledge about
desmic surfaces.
We also thank Dr. Max Weinreich who after the appearance of our paper on
\texttt{arXiv} pointed out that our previous exposition in \Cref{rem:integrability}
was slightly misleading.  We are also especially thankful to the anonymous referee for bringing to 
        our knowledge the very interested paper \cite{BedfordKim2004}.

\printbibliography

\end{document}